\documentclass[12pt,reqno]{amsart}
\usepackage{amsmath,amsfonts,amssymb,amsthm,amsxtra,amscd}
\usepackage[all]{xy}
\usepackage{mathrsfs} 
\usepackage{graphicx} 
\numberwithin{equation}{section}

\theoremstyle{plain} 
\newtheorem{thm}{Theorem}[section] 
\newtheorem{prop}[thm]{Proposition} 
\newtheorem{lemma}[thm]{Lemma} 
\newtheorem{cor}[thm]{Corollary}

\theoremstyle{definition} 
\newtheorem{defin}{Definition}[section] 
 
\newtheorem{remark}[thm]{Remark}

\newcommand{\C}{{\mathbb C}}
 
\newcommand{\p}{{\mathbb P}}

\newcommand{\z}{{\mathbb Z}} 
\newcommand{\pj}{{{\mathbb P}^1}}
\newcommand{\pii}{{{\mathbb P}^2}}
\newcommand{\piii}{{{\mathbb P}^3}}
\newcommand{\piv}{{{\mathbb P}^4}}
\newcommand{\pv}{{{\mathbb P}^5}}  
\newcommand{\pvi}{{{\mathbb P}^6}}

\newcommand{\sce}{\mathscr{E}}
\newcommand{\scf}{\mathscr{F}} 
\newcommand{\scg}{\mathscr{G}}
\newcommand{\sco}{\mathscr{O}} 
\newcommand{\sch}{\mathscr{H}}
\newcommand{\sci}{\mathscr{I}} 

\newcommand{\sck}{\mathscr{K}}

\newcommand{\tH}{\text{H}} 
\newcommand{\h}{\text{h}}

\newcommand{\izo}{\overset{\sim}{\rightarrow}} 
 
\newcommand{\ra}{\rightarrow} 
\newcommand{\lra}{\longrightarrow} 
\newcommand{\xra}{\xrightarrow}  
\newcommand{\vb}{\, \vert \, } 
\newcommand{\prim}{{\, \prime}} 
\newcommand{\secund}{{\prime \prime}}
\newcommand{\Ker}{\text{Ker}\, }
\newcommand{\Cok}{\text{Coker}\, }

\newcommand{\e}{\varepsilon}

\begin{document}

\title[Vector bundles with $c_1 = 5$ on $\mathbb{P}^n$]{Globally generated 
vector bundles with $c_1 = 5$ on $\mathbb{P}^n$, $n \geq 4$}

\author[Anghel]{Cristian~Anghel} 
\address{Institute of Mathematics ``Simion Stoilow'' of the Romanian Academy, 
         P.O. Box 1-764,   
         RO--014700, Bucharest, Romania} 
\email{Cristian.Anghel@imar.ro}   

\author[Coand\u{a}]{Iustin~Coand\u{a}} 
\email{Iustin.Coanda@imar.ro} 

\author[Manolache]{Nicolae~Manolache} 
\email{Nicolae.Manolache@imar.ro} 




\begin{abstract}
We complete the classification of globally generated vector bundles with 
small $c_1$ on projective spaces by treating the case $c_1 = 5$ on 
$\mathbb{P}^n$, $n \geq 4$ (the case $c_1 \leq 3$ has been considered by 
Sierra and Ugaglia, while the cases $c_1 = 4$ on any projective space and 
$c_1 = 5$ on $\mathbb{P}^2$ and $\mathbb{P}^3$ have been studied in two of 
our previous papers). It turns out that there are very few indecomposable 
bundles of this kind: besides some obvious examples there are, roughly 
speaking, only the (first twist of the) rank 5 vector bundle which is the 
middle term of the monad defining the Horrocks bundle of rank 3 on 
$\mathbb{P}^5$, and its restriction to $\mathbb{P}^4$. We recall, in an 
appendix, from our preprint [arXiv:1805.11336], the main results allowing the 
classification of globally generated vector bundles with $c_1 = 5$ on 
$\mathbb{P}^3$. Since there are many such bundles, a large part of the main 
body of the paper is occupied with the proof of the fact that, except for the 
simplest ones, they do not extend to $\mathbb{P}^4$ as globally generated 
vector bundles.    
\end{abstract}

\maketitle 
\tableofcontents

\section*{Introduction} 

We classify, in this paper, the globally generated vector bundles with first 
Chern class $c_1 = 5$ on the $n$-dimensional projective space $\p^n$ (over 
an algebraically closed field $k$ of characteristic 0) for $n \geq 4$. This 
completes the classification of globally generated vector bundles with 
$c_1 \leq 5$ on projective spaces. Indeed, Sierra and Ugaglia \cite{su}, 
\cite{su2} solved the case $c_1 \leq 3$, while we treated the cases $c_1 = 4$ 
on any projective space and $c_1 = 5$ on $\pii$ in \cite{acm1} and the case 
$c_1 = 5$ on $\piii$ in \cite{acm3}. Moreover, Chiodera and Ellia \cite{ce} 
noticed that there is no indecomposable globally generated rank 2 vector 
bundle with $c_1 = 5$ on $\piv$. Besides their own interest, these 
classification results are useful in attacking other geometric problems$\, :$ 
see, for example, the paper of Fania and Mezzetti \cite{fm}. 

Our main result is the following$\, :$ 

\begin{thm}\label{T:main} 
Let $E$ be an indecomposable globally generated vector bundle with $c_1 = 5$  
on $\p^n$, $n \geq 4$, such that ${\fam0 H}^i(E^\vee) = 0$, $i = 0,\, 1$. Then 
one of the following holds$\, :$ 
\begin{enumerate}
\item[(i)] $E \simeq \sco_{\p^n}(5)$$\, ;$ 
\item[(ii)] $E \simeq P(\sco_{\p^n}(5))$$\, ;$ 
\item[(iii)] $n = 4$ and one has an exact sequence$\, :$ 
\[
0 \lra \Omega_\piv^3(3) \lra \Omega_\piv^2(2) \oplus \Omega_\piv^1(1) \lra 
E(-1) \lra 0\, ; 
\]
\item[(iv)] $n = 4$ and one has an exact sequence$\, :$ 
\[
0 \lra E(-1) \lra \Omega_\piv^2(2) \oplus \Omega_\piv^1(1) \lra \sco_\piv 
\lra 0\, ; 
\]
\item[(v)] $n = 5$ and one has an exact sequence$\, :$ 
\[
0 \lra \Omega_\pv^4(4) \lra \Omega_\pv^2(2) \lra E(-1) \lra 0\, ; 
\]
\item[(vi)] $n = 5$ and one has an exact sequence$\, :$ 
\[
0 \lra E(-1) \lra \Omega_\pv^2(2) \lra \sco_\pv \lra 0\, ; 
\]
\item[(vii)] $n = 6$ and $E \simeq \Omega_\pvi^1(2)$$\, ;$ 
\item[(viii)] $n = 6$ and $E \simeq \Omega_\pvi^4(5)$. 
\end{enumerate}
\end{thm}   

As a matter of notation$\, :$ if $E$ is a globally generated vector bundle on 
$\p^n$, $P(E)$ denotes the dual of the kernel of the evaluation morphism 
$\tH^0(E) \otimes_k \sco_{\p^n} \ra E$. It is globally generated and has Chern 
classes $c_1(P(E)) = c_1(E)$, $c_2(P(E)) = c_1(E)^2 - c_2(E)$ etc. This 
construction allows one, when classifying globally generated vector bundles, 
to assume that $c_2(E) \leq c_1(E)^2/2$. Notice that $\Omega_\pvi^4(5) 
\simeq P(\Omega_\pvi^1(2))$ and if $E$ is the bundle from item (iii) (resp., 
(v)) of the theorem then $P(E)$ is the bundle from item (iv) (resp., (vi)). 

As for the condition $\tH^i(E^\vee) = 0$, $i = 0,\, 1$, if $E$ is a globally 
generated vector bundle on $\p^n$ then $\tH^0(E^\vee) = 0$ if and only if $E$ 
has no direct summand isomorphic to $\sco_{\p^n}$ and, in this case, 
considering the universal extension$\, :$ 
\[
0 \lra \tH^1(E^\vee)^\vee \otimes_k \sco_{\p^n} \lra {\widetilde E} \lra 
E \lra 0\, , 
\]
$\widetilde E$ is globally generated, it has the same Chern classes as $E$, 
and $\tH^i({\widetilde E}^\vee) = 0$, $i = 0,\, 1$. 

It is stricking, again, how rare are the globally generated vector 
bundles, this time with $c_1 = 5$, on higher dimensional projective spaces. 
Notice that if $E$ is the vector bundle from item (v) of the theorem then 
$E(-1)$ is the middle term of the monad defining the Horrocks bundle of 
rank 3 on $\pv$ (see \cite{ho}). 

\vskip2mm 

The proof of Theorem~\ref{T:main} uses the classification of globally 
generated vector bundles with $c_1 = 5$ on $\piii$ from our lengthy preprint 
\cite{acm3}. Fortunately, we use here only the basic principles of this 
classification and we recall everything we need, with complete proofs (except 
for one fact), in Appendix~\ref{A:c1=5onp3}. More precisely, if $F$ is a 
globally generated vector bundle of rank $\geq 3$ with $c_1 = 5$ on $\piii$ 
such that $\tH^i(F^\vee) = 0$, $i = 0,\, 1$, and if $\tH^0(F(-2)) \neq 0$ then 
$F$ admits a direct summand of the form $\sco_\piii(a)$, for some $a$ with 
$2 \leq a \leq 5$. A nine pages long proof of this fact can be found in 
\cite[Appendix~A]{acm3} and we decided to not reproduce it in the present 
paper. On the other hand, if $\tH^0(F(-2)) = 0$ then, with some exceptions 
that can be described explicitly, $F$ can be realized as an extension$\, :$ 
\[
0 \lra (\text{rk}\, F - 3)\sco_\piii \lra F \lra G(2) \lra 0\, ,
\] 
where $G$ is a \emph{stable} rank 3 vector bundle with $c_1(G) = -1$. If 
$c_2(F) \leq 12$, which we can assume using the functor $P(\ast)$ defined 
above, then $c_2(G) \leq 4$. Taking advantage of the fact that the 
intermediate cohomology of $G$ (and its twists) can be described by a numerical 
invariant called the \emph{spectrum} of $G$, one can get a description of the 
Horrocks (or, sometimes, Beilinson) monad of $F$. The hard part of the 
classification on $\piii$ is to show that the cohomology bundles of these 
monads are really globally generated but, fortunately, we do not need this 
here. For a significant application of our constructions of globally generated 
vector bundles with $c_1 = 5$ on $\piii$, see, however, \cite{acm4}. 

\vskip2mm 

As for the classification problem we are concerned with in this paper, if $E$ 
is a globally generated vector bundle with $c_1 = 5$ on $\p^n$, $n \geq 4$, 
such that $\tH^i(E^\vee) = 0$, $i = 0,\, 1$, and if $\tH^0(E_\Pi(-2)) \neq 0$ for 
some fixed 3-plane $\Pi \subset \p^n$ then we show, in Section~1, that   
$E$ has a direct summand of the 
form $\sco_{\p^n}(a)$, for some $a$ with $2 \leq a \leq 5$. The proof of this 
fact uses two lifting results from \cite[Chap.~1]{acm1} that we recall, too. 
It follows that we can concentrate only on the case where $\tH^0(E_\Pi(-2)) 
= 0$, for every 3-plane $\Pi \subset \p^n$. This turns out to be a quite 
strong restriction. 

\vskip2mm 

We classify, in Section~2, the globally generated vector bundles $E$ with 
$c_1 = 5$ and $c_2 \leq 12$ on $\piv$ such that $\tH^i(E^\vee) = 0$, $i = 0,\, 
1$, and that $\tH^0(E_H(-2)) = 0$, for every hyperplane $H \subset \piv$. We 
spend most of the time showing that, except for the simplest ones, the 
globally generated vector bundles with $c_1 = 5$ on $\piii$ do not extend to 
$\piv$ as globally generated vector bundles. 

\vskip2mm 

Finally, we describe, in Section~3, the globally generated vector bundles with 
$c_1 = 5$ on $\p^n$, $n \geq 5$. This is easier because on $\piv$ there are 
very few such bundles. 

\vskip2mm 

Unfortunately, the method used in this paper (and in the previous ones), which 
consists in classifying globally generated vector bundles on $\piii$ (the 
case of $\pii$ is special$\, :$ see Ellia \cite{e}) and then trying to 
decide which of them extend to higher dimensional projective spaces, does 
not seem to work, anymore, for $c_1 > 5$. The reason is that on $\piii$ 
there are too many globally generated vector bundles. Moreover, in order 
to achieve the classification in the case $c_1 \leq 5$, we almost exhausted 
the results about vector bundles on projective spaces, obtained by several 
authors in the period when this was a quite active domain, namely the 1970s 
and 1980s. There might be possible to classify globally generated vector 
bundles with $c_1 \leq n$ on $\p^n$ but a different approach is needed. 
Note, in this context, that Theorem~\ref{T:main} settles the case $n = 6$ 
of \cite[Conjecture~0.3]{acm1} about globally generated vector bundles with 
$c_1 < n$ on $\p^n$ (the case $n \leq 5$ was settled in \cite{acm1}).     

\vskip5mm 

\noindent 
{\bf Notation.}\quad (i) We work over an algebraically closed field $k$ of 
characteristic 0. 

(ii) If $X$ is a $k$-scheme of finite type, with structure sheaf $\sco_X$, 
and $\scf$ an $\sco_X$-module we denote its dual $\sch om_{\sco_X}(\scf , 
\sco_X)$ by $\scf^\vee$. We use, most of the time, the additive notation 
$m\scf$ for the direct sum of $m$ copies of $\scf$. Similarly for modules 
over a ring. We shall writte, however, $k^m$ instead of $mk$. 

(iii) For $X$ and $\scf$ as above, if $Y$ is a closed subscheme of $X$ we 
put $\scf_Y := \scf \otimes_{\sco_X} \sco_Y$ and identify it, if necessary, 
with the restriction $\scf \vb Y := i^\ast\scf$, where $i \colon Y \ra X$ is 
the inclusion morphism. 

(iv) We denote by $\p^n$ the projective space $\p(V)$ parametrizing the 
1-dimensional $k$-vector spaces of $V := k^{n+1}$. Its homogeneous coordinate 
ring is $S := \text{Symm}(V^\vee)$. If $e_0 , \ldots , e_n$ is the canonical 
basis of $V$ and $X_0 , \ldots , X_n$ the dual basis of $V^\vee$ then $S$ is 
isomorphic to the polynomial $k$-algebra $k[X_0 , \ldots , X_n]$. We denote 
by $S_+$ the ideal $(X_0 , \ldots , X_n)$ of $S$ and by $\underline k$ the 
graded $S$-module $S/S_+$.  

(v) If $\scf$ is a coherent $\sco_{\p^n}$-module and $i \geq 0$ an integer 
we denote by $\tH^i_\ast(\scf)$ the graded $S$-module 
$\bigoplus_{l \in \z}\tH^i(\scf(l))$ and by $\h^i(\scf)$ the dimension of 
$\tH^i(\scf)$ as a $k$-vector space.

\section{Preliminaries}\label{S:prelim} 

Our main purpose, in this section, is to show how one can reduce the 
classification of globally generated vector bundles $E$ with $c_1 = 5$ on 
$\p^n$, $n \geq 4$, to the case where $\tH^0(E_\Pi(-2)) = 0$, for every 
3-plane $\Pi \subset \p^n$. We also record some auxiliary results that are 
needed in the sequel.  

We begin by recalling, from Sierra and Ugaglia 
\cite[Lemma~1~and~Lemma~2]{su2}, two observations 
allowing one to reduce the classification of globally generated vector bundles 
$E$ on $\p^n$ to the case where $\tH^i(E^\vee) = 0$, $i = 0,\, 1$, and 
$c_2 \leq c_1^2/2$ ($c_1$, $c_2$ being the first two Chern classes of $E$). 

\begin{remark}\label{R:p(e)} 
(a) If $E$ is a globally generated vector bundle on $\p^n$ we denote,  
following Brugui\`{e}res \cite[Appendix~A]{bru}, by $P(E)$   
the dual of the kernel of the evaluation morphism $\tH^0(E) \otimes_k 
\sco_{\p^n} \ra E$. $P(E)$ is a globally generated vector bundle with the 
property that $\tH^i(P(E)^\vee) = 0$, $i = 0,\, 1$, and if $\tH^i(E^\vee) = 0$, 
$i = 0,\, 1$, then $P(P(E)) \simeq E$.  

(b) If $\phi \colon E \ra E^\prim$ is a morphism of globally generated vector 
bundles with the property that its kernel and cokernel are direct sums of 
copies of $\sco_{\p^n}$ then $P(\phi) \colon P(E^\prim) \ra P(E)$ is an 
isomorphism. \emph{Indeed}, one checks easily that the kernel and cokernel of 
$P(\phi)$ are direct sums of copies of $\sco_{\p^n}$. Since $\tH^0(P(E)^\vee) 
= 0$ it follows that $\Cok P(\phi) = 0$, that is, $P(\phi)$ is an epimorphism. 
Since $\tH^1(P(E)^\vee) = 0$ it follows that $\Ker P(\phi)$ is a direct 
summand of $P(E^\prim)$. But this implies that $\Ker P(\phi) = 0$ because 
$\tH^0(P(E^\prim)^\vee) = 0$.  

(c) The Chern classes of $P(E)$ can be 
related to the Chern classes $c_1, \, c_2,\, \ldots$ of $E$ by the 
formulae$\, :$ 
\begin{gather*} 
c_1(P(E)) = c_1\, ,\  c_2(P(E)) = c_1^2 - c_2\, ,\  c_3(P(E)) = c_3 + 
c_1(c_1^2 - 2c_2)\, ,\\
c_4(P(E)) = -c_4 + c_2^2 + 2c_1c_3 - 3c_1^2c_2 + c_1^4\, ,\  \text{etc.}  
\end{gather*} 
In particular, if $c_2 > c_1^2/2$ then $c_2(P(E)) < c_1^2/2$. 
\end{remark} 

\begin{remark}\label{R:f} 
Let $E$ be a globally generated vector bundle on $\p^n$. Consider the globally 
generated vector bundle $F := P(P(E))$. Recall, from Remark~\ref{R:p(e)}(a), 
that one has, by definition, an exact sequence$\, :$ 
\[
0 \lra P(E)^\vee \overset{u}{\lra} \tH^0(E) \otimes_k \sco_{\p^n} 
\xra{\text{ev}_E} E \lra 0\, . 
\]  
Applying the Snake Lemma to the dual of the commutative diagram$\, :$ 
\[
\SelectTips{cm}{12}\xymatrix{0\ar[r] & 
E^\vee\ar[r]^-{\text{ev}_E^\vee}\ar @{-->}[d] & 
{\tH^0(E)^\vee\otimes_k\sco_{\p^n}}\ar[r]^-{u^\vee}\ar[d]_{\tH^0(u^\vee)} & 
P(E)\ar[r]\ar @{=}[d] & 0\\ 
0\ar[r] & F^\vee\ar[r] & 
{\tH^0(P(E))\otimes_k\sco_{\p^n}}\ar[r]^-{\text{ev}_{P(E)}} & 
P(E)\ar[r] & 0} 
\]
one gets an exact sequence$\, :$ 
\[
0 \lra \tH^1(E^\vee)^\vee \otimes_k \sco_{\p^n} \lra F \lra E 
\xra{\text{ev}_{E^\vee}^\vee} \tH^0(E^\vee)^\vee \otimes_k \sco_{\p^n} \lra 0\, . 
\]
Since $E$ is globally generated, $\text{ev}_{E^\vee}^\vee$ has a right inverse 
hence $E \simeq (\tH^0(E^\vee)^\vee \otimes_k \sco_{\p^n}) \oplus Q$, where $Q$ 
is the cokernel of $\tH^1(E^\vee)^\vee \otimes_k \sco_{\p^n} \ra F$. 

(b) With the notation from (a), one has$\, :$ $F$ is globally generated, 
$\tH^i(F^\vee) = 0$, $i = 0,\, 1$, $E$ and $F$ have the same Chern classes, 
$\tH^0(F(l)) \izo \tH^0(E(l))$ for $l \leq -1$, $\tH^i_\ast(F) \izo 
\tH_\ast^i(E)$ if $1 \leq i \leq n-2$, and, if $i = n-1$ or $n$,  
$\tH^i(F(l)) \izo \tH^i(E(l))$ for $l \geq -n$. 

(c) If one has an exact sequence $0 \ra E^\prim \ra m\sco_{\p^n} \ra E \ra 0$ 
then $P(E^{\prim \vee}) \simeq F$. \emph{Indeed}, the kernel and cokernel of the 
obvious morphism $E^\prim \ra P(E)^\vee$ are direct sums of copies of 
$\sco_{\p^n}$ hence the same is true for its dual $P(E) \ra E^{\prim \vee}$. 
One applies, now, Remark~\ref{R:p(e)}(b). 
\end{remark} 

The next two results are \cite[Lemma~1.18]{acm1} and a (more practical) 
variant of \cite[Lemma~1.19]{acm1} (combined with \cite[Remark~1.20(c)]{acm1}). 

\begin{lemma}\label{L:a+p(b)} 
Let $E$ be a globally generated vector bundle on $\p^n$, $n \geq 4$, such that 
${\fam0 H}^i(E^\vee) = 0$, $i = 0,\, 1$, and $H \subset \p^n$ a fixed 
hyperplane. Let $F$ be the vector bundle $P(P(E_H))$ on $H \simeq \p^{n-1}$ 
$($see Remark~\emph{\ref{R:f}}$)$. If $F \simeq A \oplus P(B)$, with 
$A$ and $B$ direct sums of line bundles on $H$ such that ${\fam0 H}^0(A^\vee) = 
0$ and ${\fam0 H}^0(B^\vee) = 0$, then $E \simeq {\widehat A} \oplus 
P(\widehat B)$, where $\widehat A$ and $\widehat B$ are direct sums of line 
bundles on $\p^n$ lifting $A$ and $B$, respectively. 
\qed 
\end{lemma} 

\begin{lemma}\label{L:a+p(b)+omega(2)} 
Let $E$ be a globally generated vector bundle on $\p^n$, $n \geq 4$, such that 
${\fam0 H}^i(E^\vee) = 0$, $i = 0,\, 1$, and $\Pi \subset \p^n$ a fixed 
$3$-plane. Let $F$ be the vector bundle $P(P(E_\Pi))$ on $\Pi \simeq \piii$ 
$($see Remark~\emph{\ref{R:f}}$)$. If $F \simeq A \oplus P(B) \oplus 
\Omega_\Pi(2)$, with $A$ and $B$ direct sums of line bundles on $\Pi$ such 
that ${\fam0 H}^0(A^\vee) = 0$, ${\fam0 H}^0(B^\vee) = 0$, ${\fam0 rk}\, A < n$  
and ${\fam0 rk}\, B < n$, then one of the following holds$\, :$ 
\begin{enumerate} 
\item[(i)] $A \simeq A_1 \oplus (n-3)\sco_\Pi(1)$ and $E \simeq {\widehat A}_1 
\oplus P(\widehat B) \oplus \Omega_{\p^n}(2)$, where ${\widehat A}_1$ and 
$\widehat B$ are direct sums of line bundles on $\p^n$ lifting $A_1$ and $B$, 
respectively$\, ;$  
\item[(ii)] $B \simeq B_1 \oplus (n-3)\sco_\Pi(1)$ and $E \simeq {\widehat A}
\oplus P({\widehat B}_1) \oplus \Omega_{\p^n}^{n-2}(n-1)$, where $\widehat A$ and 
${\widehat B}_1$ are direct sums of line bundles on $\p^n$ lifting $A$ and 
$B_1$, respectively. 
\end{enumerate}
\end{lemma} 

We include, for the reader's convenience, an argument that uses the first half 
of the proof of \cite[Lemma~1.19]{acm1} and some arguments that appear in  
\cite[Remark~1.20(c),(d)]{acm1}.

\begin{proof}[Proof of Lemma~\emph{\ref{L:a+p(b)+omega(2)}}]  
Consider a saturated flag $\Pi = \Pi_3 \subset \Pi_4 \subset \ldots \subset 
\Pi_n = \p^n$ of linear subspaces of $\p^n$ and put $E_i := E \vb \Pi_i$, 
$i = 3 , \ldots , n$. In particular, $E_3 = E_\Pi$ and $E_n = E$. One has 
$\tH^1_\ast(E_3) \simeq {\underline k}(2)$. One deduces, by induction, that 
$\tH^1(E_i(l)) = 0$ for $l \leq -3$, $i = 3 , \ldots , n$. It follows, in 
particular, that $\tH^1(E_4(-2))$ injects into $\tH^1(E_3(-2)) \simeq k$. 

\vskip2mm 

\noindent 
{\bf Case 1.}\quad $\tH^1(E_4(-2)) \neq 0$. 

\vskip2mm 

\noindent 
In this case, $\tH^1(E_4(-2)) \izo \tH^1(E_3(-2))$. One deduces that 
$\tH^2_\ast(E_4) = 0$. This implies, by induction, that $\tH^2_\ast(E_i) = 0$, 
$i = 4 , \ldots , n$ hence the restriction map $\tH^1(E_i(-2)) \ra 
\tH^1(E_{i-1}(-2))$ is bijective, $i = 4 , \ldots , n$ (recall that 
$\tH^1(E_i(-3)) = 0$, $i = 4, \ldots , n$). In particular, $\tH^1(E(-2)) 
\izo \tH^1(E_\Pi(-2))$.  

On the other hand, by Serre duality, $\tH^2_\ast(E_4^\vee) = 0$. This implies, 
by induction, that $\tH^2_\ast(E_i^\vee) = 0$, $i = 4 , \ldots , n$. Recalling 
that $\tH^1(E^\vee) = 0$, one gets, using the exact sequence$\, :$ 
\[
\tH^1(E_n^\vee) \lra \tH^1(E_{n-1}^\vee) \lra \tH^2(E_n^\vee(-1)) 
\]   
that $\tH^1(E_{n-1}^\vee) = 0$. It follows, by decreasing induction, that 
$\tH^1(E_i^\vee) = 0$, $i = n,\, n-1,\, \ldots ,\, 3$. In particular, 
$\tH^1(E_\Pi^\vee) = 0$. One deduces, from Remark~\ref{R:f}, that$\, :$ 
\[
E_\Pi \simeq A \oplus \Omega_\Pi(2) \oplus P(B) \oplus t\sco_\Pi\, , 
\]
for some integer $t \geq 0$. 

A non-zero element of $\tH^1(E(-2))$ defines an extension$\, :$ 
\[
0 \lra E \lra E^\prim \lra \sco_{\p^n}(2) \lra 0\, ,
\]
whose restriction to $\Pi$ is equivalent to the extension$\, :$ 
\[
0 \lra E_\Pi \lra A \oplus 4\sco_\Pi(1) \oplus P(B) \oplus t\sco_\Pi 
\xra{(0\, ,\, \e\, ,\, 0\, ,\, 0)} \sco_\Pi(2) \lra 0\, , 
\]
with $\e$ an epimorphism. Since $\tH^i(E^{\prim \vee}) = 0$, $i = 0,\, 1$, 
Lemma~\ref{L:a+p(b)} implies that $E^\prim \simeq {\widehat A} \oplus 
4\sco_{\p^n}(1) \oplus P(\widehat B)$ hence the above extension is 
equivalent to an extension of the form$\, :$ 
\[
0 \lra E \lra {\widehat A} \oplus 4\sco_{\p^n}(1) \oplus P(\widehat B) 
\overset{\phi}{\lra} \sco_{\p^n}(2) \lra 0\, .  
\] 

Now, $A = A_2 \oplus m\sco_\Pi(1)$, where $A_2$ is a direct sum of sheaves of 
the form $\sco_\Pi(a)$, with $a \geq 2$. Since $\tH^1(E(-2)) \neq 0$ it 
follows that $\phi \vb {\widehat A}_2 = 0$ hence $E \simeq {\widehat A}_2 
\oplus K$, where $K$ is the kernel of the epimorphism $\psi \colon 
(m+4)\sco_{\p^n}(1) \oplus P(\widehat B) \ra \sco_{\p^n}(2)$ induced by $\phi$. 
Let $\psi_1 \colon (m+4)\sco_{\p^n}(1) \ra \sco_{\p^n}(2)$ and $\psi_2 \colon 
P(\widehat B) \ra \sco_{\p^n}(2)$ be the components of $\psi$. 

\vskip2mm 

\noindent 
{\bf Claim.}\quad $\psi_1$ \emph{is an epimorphism}. 

\vskip2mm 

\noindent 
\emph{Indeed}, assume, by contradiction, that it is not. Then $\Cok \psi_1 
\simeq \sco_\Lambda(2)$, for some non-empty linear subspace $\Lambda$ of $\p^n$ 
such that $\Lambda \cap \Pi = \emptyset$ (because $\tH^1(E_\Pi(-1)) = 0$ hence 
$\tH^0((\psi_1)_\Pi(-1))$ is surjective). One has an exact sequence$\, :$ 
\[
0 \lra \Ker \psi_1 \lra K \lra P(\widehat B) 
\overset{{\overline \psi}_2}{\lra} \sco_\Lambda(2) \lra 0\, , 
\]  
where ${\overline \psi}_2$ is the composite morphism $P(\widehat B) 
\overset{\psi_2}{\lra} \sco_{\p^n}(2) \ra \sco_\Lambda(2)$. Let $W \subseteq 
\tH^0(\sco_\Lambda(2))$ be the image of $\tH^0({\overline \psi}_2)$. Since the 
kernel $\sck$ of ${\overline \psi}_2$ is globally generated, applying the Snake 
Lemma to the diagram$\, :$ 
\[
\SelectTips{cm}{12}\xymatrix{0\ar[r] & 
{\tH^0(\sck) \otimes \sco_\p}\ar[r]\ar[d] & 
{\tH^0(\widehat B)^\vee \otimes \sco_\p}\ar[r]\ar[d] & 
{W \otimes \sco_\p}\ar[r]\ar[d] & 0\\ 
0\ar[r] & \sck\ar[r] & P(\widehat B)\ar[r]^-{{\overline \psi}_2} & 
\sco_{\Lambda}(2)\ar[r] & 0}
\] 
one gets an exact sequence ${\widehat B}^\vee \ra W \otimes \sco_{\p^n} \ra 
\sco_{\Lambda}(2) \ra 0$. Any component of the degeneracy locus of 
the morphism ${\widehat B}^\vee \ra W \otimes \sco_{\p^n}$ must have 
codimension $\leq \text{rk}\, {\widehat B}^\vee - \dim_k W + 1$ hence  
$\text{codim}(\Lambda , \p^n) \leq \text{rk}\, {\widehat B}^\vee - \dim_k W + 
1$. Since $W$ generates $\sco_{\Lambda}(2)$ on $\Lambda$ one must have $\dim_k W 
\geq \dim \Lambda + 1$. One deduces that $\text{rk}\, {\widehat B}^\vee \geq n$ 
which \emph{contradicts} our hypothesis that $\text{rk}\, B < n$.   
 
\vskip2mm 

It follows, from the claim, that $K \simeq (m - n + 3)\sco_{\p^n}(1) \oplus 
K^\prim$, where $K^\prim$ sits into an exact sequence$\, :$   
\[
0 \lra \Omega_{\p^n}(2) \lra K^\prim \lra P(\widehat B) \lra 0\, . 
\]  
But $\text{Ext}_{\sco_{\p^n}}^1(P(\widehat B) , \Omega_{\p^n}(2)) = 0$ hence  
$K^\prim \simeq P(\widehat B) \oplus \Omega_{\p^n}(2)$. 

\vskip2mm 

\noindent 
{\bf Case 2.}\quad $\tH^1(E_4(-2)) = 0$. 

\vskip2mm 

\noindent 
In this case, $\tH^1(E_3(-2))$ injects into $\tH^2(E_4(-3))$ hence 
$\tH^2(E_4(-3)) \neq 0$. Consider the (globally generated) vector bundle 
$E^\prim := P(E)$ on $\p^n$ and the vector bundle $F^\prim := P(P(E^\prim_\Pi))$ 
associated to $E^\prim_\Pi$ on $\Pi$. 
Using the exact sequence$\, :$ 
\[
0 \lra E^{\prim \vee}_\Pi \lra \tH^0(E) \otimes_k \sco_\Pi \lra E_\Pi \lra 0\, ,  
\]    
and Remark~\ref{R:f}(c), one gets that $P(E^\prim_\Pi) \simeq F$ (as vector 
bundles on $\Pi$) hence $F^\prim \simeq P(F) \simeq B \oplus P(A) \oplus 
\Omega_\Pi(2)$. 
Moreover, $\h^1(E^\prim_4(-2)) = \h^3(E^{\prim \vee}_4(-3)) = \h^2(E_4(-3)) \neq 
0$.  
Consequently, $E^\prim$ satisfies the hypotheses of the lemma, with $A$ 
and $B$ interchanged, and of Case 1. 
One deduces that there is a decomposition $B \simeq B_1 \oplus 
(n-3)\sco_\Pi(1)$ such that $E^\prim \simeq {\widehat B}_1 \oplus P(\widehat A) 
\oplus \Omega_{\p^n}(2)$ hence $E \simeq P(E^\prim) \simeq {\widehat A} \oplus 
P({\widehat B}_1) \oplus \Omega_{\p^n}^{n-2}(n-1)$.  
\end{proof}

The next result achieves the goal stated at the beginning of the section. 

\begin{thm}\label{T:h0e(-2)neq0} 
Let $E$ be a globally generated vector bundle with $c_1 = 5$ on $\p^n$, 
$n \geq 4$, such that ${\fam0 H}^i(E^\vee) = 0$, $i = 0,\, 1$. Let $\Pi 
\subset \p^n$ be a fixed $3$-plane. If ${\fam0 H}^0(E_\Pi(-2)) \neq 0$ then 
$E \simeq \sco_{\p^n}(a) \oplus E^\prim$, where $a$ is an integer with $2 \leq a 
\leq 5$ and $E^\prim$ is a globally generated vector bundle with $c_1(E^\prim) = 
5 - a$.  
\end{thm}  

The globally generated vector bundles $E^\prim$ on $\p^n$ with $c_1(E^\prim) 
\leq 3$ and such that $\tH^i(E^{\prim \vee}) = 0$, $i = 0,\, 1$, have been 
classified by Sierra and Ugaglia \cite{su}, \cite{su2}. Their results are  
recalled in \cite[Thm.~0.1]{acm1}. On $\piv$, these bundles are direct 
sums of bundles of the form $\sco_\piv(b)$, $P(\sco_\piv(b))$ (both with $c_1 
= b$, $1 \leq b \leq 3$), $\Omega_\piv(2)$ and $\Omega_\piv^2(3)$ (both with 
$c_1 = 3$) while on $\p^n$, $n \geq 5$, they are direct sums of bundles of the 
form $\sco_{\p^n}(b)$ and $P(\sco_{\p^n}(b))$.   

\begin{proof}[Proof of Theorem~\emph{\ref{T:h0e(-2)neq0}}] 
The result is known if $\tH^0(E_\Pi(-3)) \neq 0$ (see \cite[Prop.~2.4]{acm1} 
and \cite[Prop.~2.11]{acm1}). Assume, now, that $\tH^0(E_\Pi(-2)) \neq 0$ and 
$\tH^0(E_\Pi(-3)) = 0$.  
Let $F$ be the vector bundle $P(P(E_\Pi))$ on $\Pi \simeq \piii$ (see 
Remark~\ref{R:f}).  
According to Prop.~\ref{P:h0f(-2)neq0} in Appendix~\ref{A:c1=5onp3}, either 
$F$ is a stable rank 2 vector bundle with $c_1(F) = 5$, $c_2(F) = 8$ or 
$F \simeq \sco_\Pi(2) \oplus F^\prim$ with $c_1(F^\prim) = 3$. In the former 
case, \cite[Cor.~1.5]{acm1} would imply that there exists a rank 2 vector 
bundle $E^\prim$ on $\piv$ with Chern classes $c_1(E^\prim) = 5$, $c_2(E^\prim) = 
8$ which would \emph{contradict} Schwarzenberger's congruence (recalled in  
Remark~\ref{R:generalities}(b) below). 

In the latter case, $F^\prim$ is a 
direct sum of bundles of the form $\sco_\Pi(b)$, $P(\sco_\Pi(b))$, or 
$\Omega_\Pi(2)$ (by the results of Sierra and Ugaglia).   

If $\Omega_\Pi(2)$ is not a direct summand of $F^\prim$ then  
Lemma~\ref{L:a+p(b)} implies that $\sco_{\p^n}(2)$ is a direct summand of 
$E$. 

If $F^\prim \simeq \sco_\Pi(1) \oplus \Omega_\Pi(2)$ then 
Lemma~\ref{L:a+p(b)+omega(2)} implies that $n = 4$ and $E \simeq \sco_\piv(2) 
\oplus \Omega_\piv(2)$ while if $F^\prim \simeq \text{T}_\Pi(-1) \oplus 
\Omega_\Pi(2)$ then the same result implies that $n = 4$ and $E \simeq 
\sco_\piv(2) \oplus \Omega_\piv^2(3)$.   
\end{proof} 

The second part of the section contains miscellaneous auxiliary results that 
are needed somewhere in the sequel. 

\begin{lemma}\label{L:erat(-1)} 
Let $E$ be a globally generated vector bundle on $\p^n$ such that 
${\fam0 H}^i(E^\vee) = 0$, $i = 0,\, 1$. If $\xi$ is a non-zero element of 
${\fam0 H}^1(E^\vee(-1))$ then there exists a locally split monomorphism 
$\phi \colon \Omega_{\p^n}(1) \ra E^\vee$ such that the image of 
${\fam0 H}^1(\phi(-1)) \colon {\fam0 H}^1(\Omega_{\p^n}) \ra 
{\fam0 H}^1(E^\vee(-1))$ is $k \xi$. 
\end{lemma}   

\begin{proof} 
Dualizing the exact sequence $0 \ra P(E)^\vee \ra \tH^0(E) \otimes \sco_{\p^n} 
\ra E \ra 0$ one gets that $\tH^0(E)^\vee \izo \tH^0(P(E))$ and 
$\tH^0(P(E)(-1)) \izo \tH^1(E^\vee(-1))$. It follows that $\xi$ corresponds to 
a global section $\sigma$ of $P(E)(-1)$. One uses, now, the commutative 
diagram$\, :$ 
\[
\SelectTips{cm}{12}\xymatrix{0\ar[r] & \Omega_{\p^n}(1)\ar[r]\ar @{-->}[d] & 
{\tH^0(\sco_{\p^n}(1)) \otimes \sco_{\p^n}}\ar[r]\ar[d] & 
\sco_{\p^n}(1)\ar[r]\ar[d]^\sigma & 0\\ 
0\ar[r] & E^\vee\ar[r] & {\tH^0(P(E)) \otimes \sco_{\p^n}}\ar[r] & 
P(E)\ar[r] & 0} 
\] 
taking into account the injectivity of the map $\tH^0(\sco_{\p^n}(1)) \ra 
\tH^0(P(E))$, $h \mapsto h\sigma$.  
\end{proof} 

The following elementary, well known result will be used several times in the 
sequel. Note that its particular case $r = b$ is the Bilinear Map Lemma 
\cite[Lemma~5.1]{ha} asserting that if $\mu \colon A \times B \ra C$ is a 
bilinear map such that $\mu(u , v) \neq 0$, $\forall \, u \in A \setminus 
\{0\}$, $\forall \, v \in B \setminus \{0\}$, then $\dim C \geq \dim A + 
\dim B - 1$.  

\begin{lemma}\label{L:bml} 
Let $A$, $B$ and $C$ be $k$-vector spaces, of finite dimension $a$, $b$ and 
$c$, respectively, $r$ an integer with $1 \leq r \leq \min(b , c)$ and 
$\phi \colon A \ra {\fam0 Hom}_k(B , C)$ a $k$-linear map. If $\phi(u) \colon 
B \ra C$ has rank $\geq r$, $\forall \, u \in A  \setminus \{0\}$, then $a 
\leq (b-r+1)(c-r+1)$. 
\qed 
\end{lemma}  

\begin{defin}\label{D:contraction} 
(a) Let $V$ denote the $k$-vector space $k^{n+1}$. Consider, for $i \geq 0$, 
the canonical pairing $\langle \, \ast \, ,\, \ast \, \rangle \colon 
\bigwedge^iV^\vee \times \bigwedge^iV \ra k$. One defines, for $\omega \in 
\bigwedge^pV$, the \emph{contraction} mapping $\ast \llcorner \, \omega \colon 
\bigwedge^{p+q}V^\vee \ra \bigwedge^qV^\vee$ by$\, :$ 
\[
\langle \alpha \, \llcorner\, \omega , \eta \rangle := \langle \alpha , 
\omega \wedge \eta \rangle \, ,\  \forall \, \alpha \in 
{\textstyle \bigwedge}^{p+q}V^\vee \, ,\  \forall \, \eta \in 
{\textstyle \bigwedge}^qV\, . 
\] 
By definition, $\ast \llcorner \, \omega$ is the dual of $\omega \wedge \ast 
\colon \bigwedge^qV \ra \bigwedge^{p+q}V$ and 
\[
(\alpha \, \llcorner \, \omega)\, \llcorner \, \eta = 
\alpha\, \llcorner \, (\omega \wedge \eta)\, ,\  \forall \, \alpha \in 
{\textstyle \bigwedge}^{p+q+r}V^\vee \, ,\  \forall \, \eta \in 
{\textstyle \bigwedge}^qV\, .
\]
Moreover, if one considers the isomorphisms $\bigwedge^{n+1-i}V \izo 
\bigwedge^iV^\vee$ identifying the exterior multiplication pairings  
$\bigwedge^{n+1-i}V \times \bigwedge^iV \ra \bigwedge^{n+1}V \simeq k$ with 
the above canonical pairings then $\ast \llcorner \, \omega$ can be identified 
with $\ast \wedge \omega \colon \bigwedge^{n+1-p-q}V \ra \bigwedge^{n+1-q}V$. 

(b) Recall that we view $\p^n$ as the projective space $\p(V)$ of  
1-dimensional vector subspaces of $V$. Consider the tautological geometric 
Koszul complex on $\p^n$$\, :$ 
\[
0 \ra \sco(-n-1) \otimes {\textstyle \bigwedge}^{n+1}V^\vee  
\xra{d_{n+1}} \sco(-n) \otimes {\textstyle \bigwedge}^nV^\vee \ra \cdots \ra 
\sco(-1) \otimes V^\vee \overset{d_1}{\lra} \sco \ra 0\, .  
\] 
$\Omega_{\p^n}^i(i)$ is isomorphic to the image of $d_{i+1}(i) \colon 
\sco_{\p^n}(-1) \otimes \bigwedge^{i+1}V^\vee \ra \sco_{\p^n} \otimes 
\bigwedge^iV^\vee$ and the reduced fibre of $d_{i+1}(i)$ at a point $[v] \in 
\p(V)$ can be identified with $\ast \llcorner \, v \colon \bigwedge^{i+1}V^\vee 
\ra \bigwedge^iV^\vee$. One gets, for $\omega \in \bigwedge^pV$, commutative 
diagrams$\, :$ 
\[
\SelectTips{cm}{12}\xymatrix @C=4.5pc
{{\sco_{\p^n}(-1) \otimes {\textstyle 
\bigwedge}^{p+q+1}V^\vee}\ar[r]^-{d_{p+q+1}(p+q)}\ar[d]_{\ast \llcorner \, \omega} & 
{\sco_{\p^n} \otimes {\textstyle 
\bigwedge}^{p+q}V^\vee}\ar[d]^{\ast \llcorner \, \omega}\\ 
{\sco_{\p^n}(-1) \otimes {\textstyle 
\bigwedge}^{q+1}V^\vee}\ar[r]^-{(-1)^pd_{q+1}(q)} & 
{\sco_{\p^n} \otimes {\textstyle \bigwedge}^qV^\vee}}
\]
hence $\ast \llcorner \, \omega \colon \sco_{\p^n} \otimes \bigwedge^{p+q}V^\vee  
\ra \sco_{\p^n} \otimes \bigwedge^qV^\vee$ induces a map    
$\phi \colon \Omega_{\p^n}^{p+q}(p+q) \ra \Omega_{\p^n}^q(q)$ such that 
$\tH^0(\phi(1))$ can be identified with $(-1)^p(\ast \llcorner \, \omega) 
\colon \bigwedge^{p+q+1}V^\vee \ra \bigwedge^{q+1}V^\vee$. One thus gets an 
injective map$\, :$ 
\[
{\textstyle \bigwedge}^pV \lra \text{Hom}_{\sco_{\p^n}}(\Omega_{\p^n}^{p+q}(p+q) , 
\Omega_{\p^n}^q(q)) 
\]
which turns out to be bijective, by dimensional reasons. Moreover, the mapping  
$\tH^0(\phi^\vee) \colon \tH^0(\Omega_{\p^n}^q(q)^\vee) \ra 
\tH^0(\Omega_{\p^n}^{p+q}(p+q)^\vee)$ can be identified with $\omega \wedge \ast 
\colon \bigwedge^qV \ra \bigwedge^{p+q}V$. 
\end{defin}

The next lemma is the basic fact in the construction of the 
Trautmann-Vetter-Tango bundle of rank $n-1$ on $\p^n$. 

\begin{lemma}\label{L:vetter} 
Using the notation from the above definition, let $W$ be a vector subspace 
of $\bigwedge^2V^\vee$ $($resp., $\bigwedge^{n-1}V$$)$. Consider the vector 
subspace $W^\perp$ of $\bigwedge^2V$ consisting of the elements $\eta$ such that 
$\langle \alpha , \eta \rangle = 0$, $\forall \, \alpha \in W$ $($resp., 
$\eta \wedge \omega = 0$, $\forall \, \omega \in W$$)$. Then $W$ generates 
globally $\Omega_{\p^n}^1(2)$ $($resp., $\Omega_{\p^n}^{n-1}(n-1)^\vee$$)$ if and 
only if $W^\perp$ contains no decomposable element of $\bigwedge^2V$, i.e., 
no element of the form $v \wedge w$, with $v,\, w \in V$ linearly 
independent.   
\end{lemma}

\begin{proof} 
$W$ generates $\Omega_{\p^n}^1(2)$ globally if and only if $W \llcorner \, v = 
\bigwedge^2V^\vee \llcorner \, v$ (inside $V^\vee$), $\forall \, v \in V 
\setminus \{0\}$. But $\bigwedge^2V^\vee \llcorner \, v$ is the kernel of the 
linear function $\ast \llcorner \, v \colon V^\vee \ra k$ (which is, actually, 
evaluation at $v$). If, for some $v \in V \setminus \{0\}$, $W \llcorner \, v$ 
is contained strictly in $\bigwedge^2V^\vee \llcorner \, v$ then there exists 
another linear function on $V^\vee$ vanishing on $W \llcorner \, v$. This 
linear function is of the form $\ast \llcorner \, w$, for some $w \in V 
\setminus kv$. It follows that $W \llcorner \, (v \wedge w) = (W \llcorner \, 
v) \llcorner \, w = (0)$. 

The assertion about $\Omega_{\p^n}^{n-1}(n-1)^\vee$ can be proven similarly 
(actually, $\Omega_{\p^n}^{n-1}(n-1)^\vee \simeq \Omega_{\p^n}^1(2)$).  
\end{proof}
 
\begin{lemma}\label{L:sasakura} 
Consider a morphism $\phi \colon \Omega_\piv^3(3) \oplus \Omega_\piv^2(2) \ra 
\Omega_\piv^1(1)$ defined by contraction with an $\omega \in \bigwedge^2V$ and 
a $v \in V$, where $V := k^5$ $($see Definition~\emph{\ref{D:contraction}}$)$. 
Then the following assertions are equivalent$\, :$ 
\begin{enumerate} 
\item[(i)] $\phi$ is an epimorphism$\, ;$ 
\item[(ii)] There exists a $k$-basis $v_0 , \ldots , v_4$ of $V$ such that 
$\omega = v_0 \wedge v_1 + v_2 \wedge v_3$ and $v = v_4$$\, ;$ 
\item[(iii)] ${\fam0 H}^0(\phi(1)) \colon {\fam0 H}^0(\Omega_\piv^3(4) \oplus 
\Omega_\piv^2(3)) \ra {\fam0 H}^0(\Omega_\piv^1(2))$ is surjective.  
\end{enumerate} 
\end{lemma}

\begin{proof} 
According to Definition~\ref{D:contraction}, $\tH^0(\phi(1))$ can be 
identified with the map $\bigwedge^4V^\vee \oplus \bigwedge^3V^\vee \ra 
\bigwedge^2V^\vee$ defined by contraction with $\omega$ and with $-v$ and 
this map can be identified with the map $V \oplus \bigwedge^2V \ra 
\bigwedge^3V$ defined by exterior multiplication to the right by $\omega$ and 
by $-v$. Consider the subspace $W := V \wedge \omega - \bigwedge^2V \wedge v$ 
of $\bigwedge^3V$. Since the isomorphism $\bigwedge^3V \izo 
\bigwedge^2V^\vee$ identifies the exterior multiplication pairing $\bigwedge^3V 
\times \bigwedge^2V \ra \bigwedge^5V$ with the canonical pairing 
$\bigwedge^2V^\vee \times \bigwedge^2V \ra k$, Lemma~\ref{L:vetter} implies 
that $\phi(1)$ is an epimorphism if and only if the subspace $W^\perp$ of 
$\bigwedge^2V$ contains no decomposable element. 
Now, one has$\, :$ 
\[
(V \wedge \omega)^\perp = \Ker({\textstyle \bigwedge}^2V 
\xra{\omega \wedge \ast} {\textstyle \bigwedge}^4V)\, ,\  
({\textstyle \bigwedge}^2V \wedge v)^\perp = \Ker({\textstyle \bigwedge}^2V 
\xra{v \wedge \ast} {\textstyle \bigwedge}^3V) \supseteq v \wedge V\, .  
\]

If $\omega = v_0 \wedge v_1$, with $v_0$, $v_1$ linearly independent then 
$(V \wedge \omega)^\perp \supseteq v_0 \wedge V + v_1 \wedge V$ hence $W^\perp$ 
contains decomposable elements. 

It remains that $\omega = v_0 \wedge v_1 + v_2 \wedge v_3$, with $v_0, \ldots , 
v_3$ linearly independent. Put $V^\prime := kv_0 + \cdots + kv_3$. Then$\, :$ 
\[
(V \wedge \omega)^\perp \supseteq \Ker({\textstyle \bigwedge}^2V^\prime  
\xra{\omega \wedge \ast} {\textstyle \bigwedge}^4V^\prime \simeq k) 
\]
hence if $v \in V^\prime$ then $W^\perp$ contains a (non-zero) decomposable 
element of $\bigwedge^2V$. 

Consequently, $v_0 , \ldots , v_3 , v$ must be linearly independent. In this 
case, $V^\prime \wedge \omega = \bigwedge^3V^\prime$ and $\bigwedge^2V \wedge v 
\supseteq \bigwedge^2V^\prime \wedge v$ hence $W = \bigwedge^3V$. 
\end{proof} 

\begin{cor}\label{C:sasakura} 
Consider an epimorphism $\e \colon \Omega_\piv^2(2) \oplus \Omega_\piv^1(1) 
\ra \sco_\piv$ defined by contraction with an $\omega \in \bigwedge^2V$ and a 
$v \in V$, where $V := k^5$. Then ${\fam0 Ker}\, \e(1)$ is globally generated 
if and only if there exists a $k$-basis $v_0, \ldots , v_4$ of $V$ such that 
$\omega = v_0 \wedge v_1 + v_2 \wedge v_3$ and $v = v_4$.  
\end{cor}

\begin{proof} 
Let $K$ be the kernel of $\e$. Applying the Snake Lemma to the diagram whose 
vertical morphisms are the evaluation morphisms of the terms of the short 
exact sequence$\, :$ 
\[
0 \lra K(1) \lra \Omega_\piv^2(3) \oplus \Omega_\piv^1(2) \xra{\e(1)} 
\sco_\piv(1) \lra 0\, , 
\]
one gets that $K(1)$ is globally generated if and only if the morphism 
$\Omega_\piv^3(3) \oplus \Omega_\piv^2(2) \ra \Omega_\piv^1(1)$ defined by 
contraction with $\omega$ and with $-v$ is an epimorphism. One can apply, 
now, Lemma~\ref{L:sasakura}. 
\end{proof}

\section{The case $c_1 = 5$ on $\piv$}\label{S:c1=5onp4}  

We classify, in this section, the globally generated vector bundles $E$ with 
$c_1 = 5$ on $\piv$ with the property that $\tH^i(E^\vee) = 0$, $i = 0,\, 1$, 
and such that $\tH^0(E_H(-2)) = 0$, for every hyperplane $H \subset \piv$. 
We actually use the results about the classification of the analogous bundles 
on $\piii$, recalled in Appendix~\ref{A:c1=5onp3}, and try to decide which 
of these bundles extend to $\piv$ (as globally generated vector bundles). We 
spend most of the time showing that many of them do not extend.   

We begin by collecting, in the next result, some general information about 
globally generated vector bundles with $c_1 = 5$ on $\piv$.  

\begin{remark}\label{R:generalities} 
Let $E$ be a globally generated vector bundle on $\piv$, of rank $r$, with 
Chern classes $c_1 = 5$, $c_2 \leq 12$, $c_3$, $c_4$, and such that 
$\tH^i(E^\vee) = 0$, $i = 0,\, 1$. 

\vskip2mm 

(a) $r - 1$ general global sections of $E$ define an exact sequence$\, :$ 
\[
0 \lra (r-1)\sco_\piv \lra E \lra \sci_Y(5) \lra 0\, , 
\] 
with $Y$ a \emph{nonsingular} surface in $\piv$ of degree $c_2$. Severi's 
theorem (asserting that the only surface in $\piv$ which is not linearly 
normal is the Veronese surface) implies that $\tH^1(E(l)) = 0$, for $l \leq 
-4$ (recalling \cite[Prop.~2.2]{acm1}). Moreover, Kodaira's vanishing theorem 
implies that $\tH^2(E(l)) \simeq \tH^2(\sci_Y(5 + l)) \simeq 
\tH^1(\sco_Y(5 + l)) = 0$, for $l \leq -6$. 

\vskip2mm 

(b) Applying the Riemann-Roch theorem (recalled in \cite[Thm.~7.3]{acm1}) to 
$E^\vee$ and taking into account that $\h^3(E^\vee) = \h^1(E(-5)) = 0$ and 
$\h^4(E^\vee) = \h^0(E(-5)) = 0$ (because, otherwise, $E \simeq \sco_\piv(5)$), 
one gets that$\, :$ 
\[
r = \frac{5c_3 + 2c_4 - c_2(c_2 - 10)}{12} + \h^2(E^\vee)\, . 
\]   
Moreover, Schwarzenberger's congruence $(2c_1 + 3)(c_3 - c_1c_2) + c_2^2 + c_2 
\equiv 2c_4 \pmod{12}$ (see \cite[Cor.~7.4]{acm1}) becomes, in our case$\, :$  
\[
c_2(c_2 - 4) + c_3 \equiv 2c_4 \pmod{12} 
\]
(recall, also, that $c_3 \equiv c_1c_2 \pmod{2}$ hence, in our case, $c_3 
\equiv c_2 \pmod{2}$).  

\vskip2mm 

(c) Assume, now, that there is a hyperplane $H \subset \piv$ such that 
$\tH^0(E_H(-2)) = 0$. In this case, $E$ must have rank $r \geq 3$ because 
Chiodera and Ellia \cite{ce} showed that every globally generated rank 2 vector 
bundle with $c_1 = 5$ on $\piv$ splits, i.e., is of the form $\sco_\piv(a) 
\oplus \sco_\piv(b)$, with $a \geq 0$, $b \geq 0$ and $a + b = 5$. 

According to Remark~\ref{R:f}, the (globally generated) 
vector bundle $F := P(P(E_H))$ on $H \simeq \piii$ satisfies $\tH^i(F^\vee) = 
0$, $i = 0,\, 1$, and there is an exact sequence$\, :$ 
\[
0 \lra s\sco_H \lra F \lra Q \lra 0\, , 
\]
with $s := \h^1(E_H^\vee)$, such that $E_H \simeq t\sco_H \oplus Q$, where 
$t := \h^0(E_H^\vee)$. Since $\tH^i(E^\vee) = 0$, $i = 0,\, 1$, it follows that 
$\h^0(E_H^\vee) = \h^1(E^\vee(-1))$. 
One has $\tH^1_\ast(E_H) \simeq \tH^1_\ast(F)$ and $\tH^2(E_H(l)) \simeq 
\tH^2(F(l))$ for $l \geq -3$. 

We assert that, under the above assumption, one has $\text{rk}\, Q \geq 3$ 
(hence $\text{rk}\, F \geq 3$). \emph{Indeed}, if $\text{rk}\, Q \leq 2$ then, 
by \cite[Cor.~1.5(b)]{acm1}, there exists a globally generated vector bundle 
$E^\prim$ on $\piv$ such that $E_H^\prim \simeq Q$. By the above mentioned result 
of Chiodera and Ellia, $E^\prim \simeq \sco_\piv(a) \oplus \sco_\piv(b)$ hence 
$Q \simeq \sco_H(a) \oplus \sco_H(b)$. But this \emph{contradicts} our 
assumption that $\tH^0(E_H(-2)) = 0$. 

One deduces, now, from 
Lemma~\ref{L:h2f(-2)=0}(b), that $\tH^2(E_H(l)) = 0$ for $l \geq -2$ hence 
$\tH^3(E(l)) = 0$ for $l \geq -3$. Moreover, $\tH^3(E(-5)) \simeq 
\tH^1(E^\vee)^\vee = 0$ and $\h^2(E_H(-3)) \geq \h^3(E(-4)) = \h^1(E^\vee(-1)) = 
t$. One also gets, from the Riemann-Roch formula, that$\, :$ 
\begin{gather*}
\h^2(E(-3)) - \h^1(E(-3)) = \chi(\sco_\piv(c_1 - 3)) + 
\frac{(2c_1 - 3)(c_3 - c_1c_2) + c_2^2 + c_2 - 2c_4}{12}\, ,
\end{gather*}    
where, of course, $c_1 = 5$.   

We would like to point out the following \emph{basic fact}$\, :$ either 
$F$ is one of the bundles from the conclusion of Prop.~\ref{P:fprimunstable} 
or it can be realized as an extension$\, :$ 
\[
0 \lra (\text{rk}\, F - 3)\sco_H \lra F \lra G(2) \lra 0\, , 
\] 
where $G$ is a \emph{stable} rank 3 vector bundle on $H \simeq \piii$ with 
$c_1(G) = -1$, $c_2(G) = c_2 - 8$, $c_3(G) = c_3 - 2c_2 + 12$. In the latter 
case one deduces easily, from the above exact sequence, that $\text{rk}\, F = 
3 + \h^2(G(-2))$. For further information (including the definition and the 
properties of the \emph{spectrum} of $G$) the reader is refered to 
Remark~\ref{R:fprimstable}.  

\vskip2mm 

(d) Assume, finally, that $\tH^0(E_H(-2)) = 0$, for \emph{every} hyperplane 
$H \subset \piv$. Then, as we noticed in (c), one has $\tH^2(E_H(l)) = 0$ for 
$l \geq -2$ and for any hyperplane $H \subset \piv$. Consider the exact 
sequences$\, :$ 
\begin{gather*}
0 = \tH^0(E_H(-2)) \ra \tH^1(E(-3)) \overset{h}{\lra} \tH^1(E(-2)) \ra 
\tH^1(E_H(-2)) \ra\\ 
\ra \tH^2(E(-3)) \overset{h}{\lra} \tH^2(E(-2)) \ra \tH^2(E_H(-2)) = 0\, .  
\end{gather*}
Applying the Bilinear Map Lemma \cite[Lemma~5.1]{ha} one gets that 
if $\h^1(E_H(-2)) \leq 3$ then $\tH^1(E(-3)) = 0$ and $\tH^2(E(-2)) = 0$. 
The latter vanishing implies that $\tH^2(E(l)) = 0$, $\forall \, l \geq -2$. 
Notice, also, that $\h^1(E_H(-2)) = \frac{1}{2}(5(c_2 - 8) - c_3)$, by 
Lemma~\ref{L:h2f(-2)=0}(b). 
\end{remark}

\begin{lemma}\label{L:ggkoszul} 
Let $d_0 \leq d_1 \leq \cdots \leq d_n$ be positive integers and let $K$ be 
the kernel of an epimorphism $\bigoplus_{i = 0}^n\sco_{\p^n}(-d_i) \ra 
\sco_{\p^n}$. Then $K(l)$ is globally generated if and only if $l \geq d_{n-1} 
+ d_n$. 
\end{lemma} 

\begin{proof} 
The epimorphism from the statement is defined by homogeneous polynomials 
$f_0 , \ldots , f_n$ of degree $d_0 , \ldots , d_n$, respectively. Let $C 
\subset \p^n$ be the complete intersection defined by $f_0 , \ldots , 
f_{n-2}$. Then $K_C \simeq \bigoplus_{i=0}^{n-2}\sco_C(-d_i) \oplus 
\sco_C(-d_{n-1} - d_n)$. It follows that if $K(l)$ is globally generated then 
$l \geq d_{n-1} + d_n$. 
The converse can be proven using the Koszul complex.  
\end{proof}

\begin{lemma}\label{L:h2eHvee=0} 
Let $E$ be a globally generated vector bundle on $\piv$ with $c_1 = 5$, $c_2 
\leq 12$, and such that ${\fam0 H}^i(E^\vee) = 0$, $i = 0,\, 1$, and let 
$H \subset \piv$ be a hyperplane. If ${\fam0 H}^0(E_H(-2)) = 0$ then  
${\fam0 H}^2(E_H^\vee) = 0$.  
\end{lemma}

\begin{proof} 
Assume, by contradiction, that $\tH^2(E_H^\vee) \neq 0$. By Serre duality, 
$\tH^2(E_H^\vee) \simeq \tH^1(E_H(-4))^\vee$. If $F$ is the (globally generated) 
vector bundle $P(P(E_H))$ on $H$ (see Remark~\ref{R:f}), then $\tH^1(E_H(-4)) 
\simeq \tH^1(F(-4))$. It follows (see Remark~\ref{R:fprimstable}) that either
\begin{enumerate}
\item[(1)] $F$ is as in item (ii) of the conclusion of 
Prop.~\ref{P:fprimunstable}, i.e., $F \simeq \sco_H(1) \oplus F_0$, where 
$F_0$ is the kernel of an epimorphism $4\sco_H(2) \ra \sco_H(4)$ 
\end{enumerate} 
or $F$ can be realized as an extension$\, :$ 
\[
0 \lra (\text{rk}\, F - 3)\sco_H \lra F \lra G(2) \lra 0\, ,
\]
where $G$ is a stable rank 3 vector bundle on $H \simeq \piii$ with $c_1(G) = 
-1$, $c_2(G) = c_2 - 8$, $c_3(G) = c_3 - 2c_2 + 12$ such that $\tH^1(G(-2)) 
\neq 0$. Using the properties of the spectrum of $G$ (recalled in 
Remark~\ref{R:fprimstable}), Lemma~\ref{L:spectrumg} and 
Lemma~\ref{L:impossiblespectra}, one deduces that one of the following 
holds$\, :$  
\begin{enumerate}
\item[(2)] $c_2(G) = 2$ and $G$ has spectrum $(1 , 0)$$\, ;$ 
\item[(3)] $c_2(G) = 3$ and $G$ has one of the spectra $(1 , 0 , 0)$, 
$(1 , 1 , 0)$$\, ;$ 
\item[(4)] $c_2(G) = 4$ and $G$ has one of the spectra $(1 , 0 , 0 , -1)$, 
$(1 , 0 , 0 , 0)$, $(1 , 1 , 0 , -1)$, $(1 , 1 , 0 , 0)$, $(1 , 1 , 1 , 0)$. 
\end{enumerate}
We shall eliminate all of these possibilities one by one. 

\vskip2mm 

\noindent 
{\bf Case 1.}\quad $F$ \emph{as in item} (ii) \emph{of the conclusion of 
Prop.}~\ref{P:fprimunstable}. 

\vskip2mm 

\noindent 
In this case $c_2 = 12$, $c_3 = 8$ hence, according to Schwarzenberger's 
congruence, one must have $c_4 > 0$. In particular, $r \geq 4$. Since 
$\tH^2(F(l)) = 0$ for $l \geq -3$ it follows that $t = 0$ (using the notation 
from Remark~\ref{R:generalities}(c)) hence $E_H \simeq F$. In particular, 
$\tH^1_\ast(E_H^\vee) = 0$ which implies that $\tH^1_\ast(E^\vee) = 0$. Applying 
\cite[Lemma~1.14(b)]{acm1} to $E^\vee$ one deduces that $E$ is the kernel of 
an epimorphism $\sco_\piv(1) \oplus 4\sco_\piv(2) \ra \sco_\piv(4)$. But this 
\emph{contradicts}, according to Lemma~\ref{L:ggkoszul}, the fact that $E$ is 
globally generated. Consequently, \emph{this case cannot occur}. 

\vskip2mm 

The case where $G$ has spectrum $(1 , 0 , 0 , -1)$ can be eliminated 
similarly, using Lemma~\ref{L:(1,0,0,-1)}. 

\vskip2mm 

\noindent 
{\bf Case 2.}\quad $G$ \emph{has spectrum} $(1 , 0)$. 

\vskip2mm 

\noindent 
In this case, $\text{rk}\, F = 3$, $c_2(G) = 2$, and $c_3(G) = -4$ hence $c_2 = 
c_2(F) = 10$ and $c_3 = c_3(F) = 4$ (see Remark~\ref{R:fprimstable}). 
Since $\tH^2(F(l)) = 0$ for $l \geq -3$ it follows that, using the notation 
from Remark~\ref{R:generalities}(c), one has $t = 0$ hence $E$ has rank 
$r \leq 3$. Using Schwarzenberger's congruence one gets a 
\emph{contradiction} hence \emph{this case cannot occur}, either. 

\vskip2mm 

The cases where $G$ has one of the spectra $(1 , 0 , 0)$, $(1 , 1 , 0)$, 
$(1 , 0 , 0 , 0)$, $(1 , 1 , 0 , 0)$, $(1 , 1 , 1 , 0)$ can be eliminated 
similarly.             

\vskip2mm 

\noindent 
{\bf Case 3.}\quad $G$ \emph{has spectrum} $(1 , 1 , 0 , -1)$. 

\vskip2mm 

\noindent 
In this case, $\text{rk}\, F = 4$, $c_2(G) = 4$ and $c_3(G) = -6$ hence 
$c_2 = 12$ and $c_3 = 6$. Since $\tH^2(F(l)) = 0$ for $l \geq -3$ it follows 
that $t = 0$ (using the notation from Remark~\ref{R:generalities}(c)). On the 
other hand, Schwarzenberger's congruence implies that $c_4 > 0$ hence $E$ has 
rank $r \geq 4$. One deduces that $E_H \simeq F$. 

Now, $\tH^2(E_H^\vee) \neq 0$ implies that $\tH^2(E^\vee) \neq 0$ because 
$\tH^3(E^\vee(-1)) \simeq \tH^1(E(-4))^\vee = 0$ (see 
Remark~\ref{R:generalities}(a)). On the other hand, by 
Lemma~\ref{L:(1,1,0,-1)}, $\tH^1(E_H^\vee(1)) = 0$ hence $\tH^2(E^\vee) \neq 0$ 
implies that 
$\tH^2(E^\vee(1)) \neq 0$. But $\tH^2(E^\vee(1)) \simeq \tH^2(E(-6))^\vee$ and 
$\tH^2(E(-6)) = 0$ by Kodaira vanishing (see Remark~\ref{R:generalities}(a)). 
This \emph{contradiction} shows that Case 3 \emph{cannot occur}.   
\end{proof}

\begin{cor}\label{C:h2eHvee=0} 
Under the hypothesis of Lemma~\emph{\ref{L:h2eHvee=0}},  
${\fam0 h}^1(E_H^\vee) = {\fam0 h}^2(E^\vee(-1)) - {\fam0 h}^2(E^\vee)$. 
Moreover, if ${\fam0 H}^0(E_H(-2)) = 0$, for every hyperplane $H \subset 
\piv$, then either ${\fam0 H}^2(E^\vee) = 0$ or ${\fam0 h}^1(E_H^\vee) \geq 4$, 
$\forall \, H \subset \piv$ hyperplane, and ${\fam0 h}^2(E^\vee(-1)) \geq 5$.  
\end{cor}

\begin{proof} 
One uses the exact sequence$\, :$ 
\[
0 = \tH^1(E^\vee) \lra \tH^1(E_H^\vee) \lra \tH^2(E^\vee(-1)) \overset{h}{\lra}  
\tH^2(E^\vee) \lra \tH^2(E_H^\vee) = 0 
\]
($h = 0$ being an equation of $H$ in $\piv$) and, for the second part, the 
Bilinear Map Lemma \cite[Lemma~5.1]{ha}. 
\end{proof} 

\begin{remark}\label{R:spectrumg} 
Under the hypothesis of Lemma~\ref{L:h2eHvee=0}, let $F$ be the vector bundle 
$P(P(E_H))$ on $H$ (see Remark~\ref{R:f}). Assume that $F$ can be realized as 
an extension$\, :$ 
\[
0 \lra (\text{rk}\, F - 3)\sco_H \lra F \lra G(2) \lra 0,  
\] 
where $G$ is a \emph{stable} rank 3 vector bundle on $H$ (see the last part  
of Remark~\ref{R:generalities}(c)).  
Lemma~\ref{L:h2eHvee=0} implies that $\tH^1(E_H(-4)) = 0$ hence $\tH^1(G(-2)) 
= 0$. Taking into account Lemma~\ref{L:spectrumg}, one deduces that 
the spectrum $(k_1 , \ldots , k_m)$ of $G$ must satisfy the inequalities 
$0 \geq k_1 \geq \cdots \geq k_m \geq -2$.  
\end{remark}

\begin{lemma}\label{L:c3geqc2} 
Let $E$ be a globally generated vector bundle on $\piv$,   
with $c_1 = 5$, $c_2 \leq 12$, and such that ${\fam0 H}^i(E^\vee) = 0$, 
$i = 0,\, 1$. Assume, also, that ${\fam0 H}^0(E_H(-2)) = 0$, for every 
hyperplane $H \subset \piv$. Then $c_2 \geq 10$ and $c_3 \geq c_2$. 
\end{lemma}

\begin{proof} 
Assume, by contradiction, that $c_2 = 9$ (see Prop.~\ref{P:c2geq9}) or that 
$10 \leq c_2 \leq 12$ and $c_3 < c_2$. Let $H \subset \piv$ be an arbitrary 
hyperplane and let $F$ be the vector bundle $P(P(E_H))$ on $H$ (see 
Remark~\ref{R:f}). Then either$\, :$ 
\begin{enumerate} 
\item[(1)] $F$ is as in item (i) of Prop.~\ref{P:fprimunstable}   
\end{enumerate}   
($F$ cannot be as in item (ii) of Prop.~\ref{P:fprimunstable} by 
Lemma~\ref{L:h2eHvee=0}) or it can be realized as an extension $0 \ra 
(\text{rk}\, F - 3)\sco_H \ra F \ra G(2) \ra 0$, where $G$ is a stable rank 3 
vector bundle on $H \simeq \piii$. Let $(k_1 , \ldots , k_m)$ be the spectrum 
of $G$ ($m = c_2 - 8$). Since $-2\sum k_i = c_3 - c_2 + 4$ it follows that 
$\sum k_i \geq -1$. On the other hand, by Remark~\ref{R:spectrumg}, $k_1 \leq 
0$. On deduces that one of the following holds$\, :$ 
\begin{enumerate} 
\item[(2)] $G$ has one of the spectra $(0)$, $(0 , 0)$, $(0 , 0 , 0)$, 
$(0 , 0 , 0 , 0)$ in which case $9 \leq c_2 \leq 12$, $c_3 = c_2 - 4$ and 
$\text{rk}\, F = 3$$\, ;$ 
\item[(3)] $G$ has one of the spectra $(0 , -1)$, $(0 , 0 , -1)$, 
$(0 , 0 , 0 , -1)$ in which case $10 \leq c_2 \leq 12$, $c_3 = c_2 - 2$ and 
$\text{rk}\, F = 4$. 
\end{enumerate} 
Notice, also, that if $F$ is as in item (i) of Prop.~\ref{P:fprimunstable} 
then $9 \leq c_2 \leq 12$, $c_3 = c_2 - 4$ and $\text{rk}\, F = 3$. 

In all of the cases, $\tH^2(E_H(l)) = 0$ for $l \geq -3$ hence, using the 
notation from Remark~\ref{R:generalities}(c), $t = 0$. Consequently, $3 \leq r 
\leq \text{rk}\, F$. Moreover, among the above mentioned Chern classes, the 
only ones that satisfy the congruence $c_2(c_2 - 4) + c_3 \equiv 0 \pmod{12}$ 
are $c_2 = 11$, $c_3 = 7$. One deduces that either $r = 3$, $c_2 = 11$, 
$c_3 = 7$ or $c_4 > 0$. In the latter case $r \geq 4$ hence $\text{rk}\, F = 
4$ and $E_H = F$. In both cases, $\h^1(E_H^\vee) = \h^1(F^\vee) = 0$.   
One can get rid of the former case using the relation$\, :$ 
\[
r = \frac{5c_3 + 2c_4 - c_2(c_2 - 10)}{12} + \h^2(E^\vee)\, . 
\]   
from Remark~\ref{R:generalities}(b). \emph{Indeed}, this relation gives 
$\h^2(E^\vee) = 1$ while Cor.~\ref{C:h2eHvee=0} implies that $\tH^2(E^\vee) = 
0$. 

Assume, now, that $r = 4$. Then $F$ is as in item (3) above. Since, as we 
already saw, $\h^1(E_H^\vee) = 0$, Cor.~\ref{C:h2eHvee=0} implies that 
$\tH^2(E^\vee) = 0$ and $\tH^2(E^\vee(-1)) = 0$. By Serre duality, $\tH^2(E(-4)) 
= 0$. 

Moreover, if $F$ is as in item (3) above and $c_2 \in \{10,\, 11\}$ then 
Remark~\ref{R:generalities}(d) implies that $\tH^1(E(-3)) = 0$. Using the 
exact sequence$\, :$ 
\[
0 = \tH^1(E(-3)) \lra \tH^1(E_H(-3)) \lra \tH^2(E(-4)) = 0 
\] 
one gets that $\tH^1(E_H(-3)) = 0$. But, according to the spectrum, one must 
have $\h^1(E_H(-3)) \in \{1,\, 2\}$ and this is a \emph{contradiction}. 

Assume, finally, that $F$ is as in item (3) above with $c_2 = 12$. As we 
noticed above, $\tH^2(E(-4)) = 0$. Since $\tH^2(E_H(l)) = 0$ for $l \geq -3$, 
one deduces that $\tH^2(E(l)) = 0$ for $l \geq -4$. Since $\tH^1(E(-4)) = 0$ 
(see Remark~\ref{R:generalities}(a)) it follows that $\h^1(E(-3)) = 
\h^1(E_H(-3)) = 3$. Since $\tH^0(E_H(-2)) = 0$ one gets that $\h^1(E(-2)) - 
\h^1(E(-3)) = \h^1(E_H(-2)) = 5$ hence $\h^1(E(-2)) = 8$. 
Consider, now, the exact sequence$\, :$ 
\[
0 \ra \tH^0(E(-1)) \ra \tH^0(E_H(-1)) \ra \tH^1(E(-2)) \overset{h}{\ra} 
\tH^1(E(-1)) \ra \tH^1(E_H(-1)) \ra 0\, . 
\]    
Lemma~\ref{L:h0f(-1)c2=12} implies that $\h^0(E_H(-1)) \leq 1$ hence, by 
Lemma~\ref{L:h2f(-2)=0}(b), 
\[
\h^1(E_H(-1)) = \frac{1}{2}(7(c_2 - 10) - c_3) + \h^0(E_H(-1)) \leq 3\, . 
\]
Using the exact sequence above (for any linear form $h$ on $\piv$) and the 
Bilinear Map Lemma \cite[Lemma~5.1]{ha}, one gets that $\tH^0(E(-1)) = 0$. 
This implies that $\h^1(E(-1)) - \h^1(E(-2)) = \h^1(E_H(-1)) - \h^0(E_H(-1)) 
= 2$ hence $\h^1(E(-1)) = 10$. 

Since $\h^1(E_H(-1)) \leq 3$, one deduces, from Lemma~\ref{L:h2f(-2)=0}(b), 
that $\h^1(E_H) = 0$. Since this happens for every hyperplane $H \subset 
\piv$, the Bilinear Map Lemma implies that $\h^1(E(-1)) \geq \h^1(E) + 4$. 

We want, finally, to estimate $\h^0(E)$ using the exact sequence$\, :$ 
\[
0 \lra \tH^0(E) \lra \tH^0(E_H) \lra \tH^1(E(-1)) \overset{h}{\lra} \tH^1(E) 
\lra 0\, . 
\]    
By Riemann-Roch, $\h^0(E_H) = \chi(E_H) = 10$ hence $\h^0(E) \leq \h^0(E_H) 
- 4 = 6$. Since there is no epimorphism $6\sco_\piv \ra E$ (its kernel would 
have rank 2 and strictly positive $c_3$) one gets a \emph{contradiction} 
and this eliminates the case where $F$ is as in item (3) above with 
$c_2 = 12$.   
\end{proof}

\begin{prop}\label{P:c3=c2onp4} 
Let $E$ be a globally generated vector bundle on $\piv$,  
with $c_1 = 5$, $10 \leq c_2 \leq 12$, $c_3 = c_2$, and such that 
${\fam0 H}^i(E^\vee) = 0$, $i = 0,\, 1$. Assume, also, that 
${\fam0 H}^0(E_H(-2)) = 0$, for every hyperplane $H \subset \piv$. Then 
$c_2 = 10$ and $E \simeq 5\sco_\piv(1)$. 
\end{prop} 

Note that this proposition completes the classification of globally generated 
vector bundles with $c_1 = 5$ and $c_2 = 10$ on $\piv$. Indeed, by 
Prop.~\ref{P:fprimunstable} and Remark~\ref{R:fprimstable}, if $F$ is a 
globally generated vector bundle on $\piii$ with $c_1 = 5$, $c_2 = 10$, and 
such that $\tH^0(F(-2)) = 0$ then $c_3 \leq 10$.  

\begin{proof}[Proof of Prop.~\emph{\ref{P:c3=c2onp4}}]  
Let $H \subset \piv$ be an arbitrary hyperplane, of equation $h = 0$, and 
let $F_{[h]}$ be the vector bundle $P(P(E_H))$ on $H$ (see Remark~\ref{R:f}). 
Then, according to Remark~\ref{R:fprimstable},  
one of the following holds$\, :$ 
\begin{enumerate} 
\item[(i)] $F_{[h]}$ is as in item (iii) of Prop.~\ref{P:fprimunstable}$\, ;$ 
\item[(ii)] One has an exact sequence $0 \ra (\text{rk}\, F_{[h]} - 3)\sco_H 
\ra F_{[h]} \ra G_{[h]}(2) \ra 0$, where $G_{[h]}$ is a stable rank 3 vector 
bundle on $H$ with $c_1(G_{[h]}) = -1$, $2 \leq c_2(G_{[h]}) \leq 4$ and 
spectrum $(-1 , -1)$, $(0 , -1 , -1)$, $(0 , 0 , -1 , -1)$, respectively. 
\end{enumerate} 
Since $\text{rk}\, F_{[h]} = 5$ in both cases (in case (ii) one uses the 
formula $\text{rk}\, F_{[h]} = 3 + \h^2(G_{[h]}(-2))$) and $c_3(F_{[h]}) \neq 0$ 
it follows that, using the notation from Remark~\ref{R:generalities}, $s \leq 
2$, i.e., $\h^1(E_H^\vee) \leq 2$ for every hyperplane $H \subset \piv$. 
Cor.~\ref{C:h2eHvee=0} implies, now, that $\tH^2(E^\vee) = 0$ and that 
$s = \h^2(E^\vee(-1)) = \h^2(E(-4))$. Recall, also, that $t = \h^1(E^\vee(-1)) 
= \h^3(E(-4))$ ($s$ and $t$ are defined in Remark~\ref{R:generalities}(c)).     

\vskip2mm 

\noindent 
{\bf Claim 1.}\quad $\tH^2(E(-3)) = 0$. 

\vskip2mm 

\noindent 
\emph{Indeed}, by Lemma~\ref{L:h2f(-2)=0}(b), $\h^1(E_H(-2)) = 0$ for $c_2 = 
10$, $\h^1(E_H(-2)) = 2$ for $c_2 = 11$ and $\h^1(E_H(-2)) = 4$ for $c_2 = 12$.
Remark~\ref{R:generalities}(d) implies that if $c_2 \in \{10,\, 11\}$ then 
$\tH^1(E(-3)) = 0$ and $\tH^2(E(-2)) = 0$. Moreover, if $c_2 = 10$ one also 
has $\tH^1(E(-2)) = 0$ and $\tH^2(E(-3)) = 0$, because $\tH^1(E_H(-2)) = 0$.  

Now, by the Riemann-Roch formula (see Remark~\ref{R:generalities}(c)), 
$\h^2(E(-3)) - \h^1(E(-3))$ is equal to 
$(5 - c_4)/6$ if $c_2 = 10$, to $(2 - c_4)/6$ if $c_2 = 11$ and to 
$-c_4/6$ if $c_2 = 12$. Since $c_4 \geq 0$, one deduces that $c_4 = 5$ if 
$c_2 = 10$, that $\h^2(E(-3)) = 0$ and $c_4 = 2$ if $c_2 = 11$, and that 
$\h^2(E(-3)) = 0$ if $\h^1(E(-3)) = 0$ in the case $c_2 = 12$. 

Assume, finally, that $c_2 = 12$ and $\tH^1(E(-3)) \neq 0$. Since 
$\h^1(E_H(-2)) = 4$, for every hyperplane $H \subset \piv$, using the exact 
sequence from Remark~\ref{R:generalities}(d) and the Bilinear Map 
Lemma one deduces that the map $\tH^1(E(-2)) \ra \tH^1(E_H(-2))$ is surjective, 
$\forall \, H$. It follows that the multiplication by any non-zero linear 
form $h \colon \tH^2(E(-3)) \ra \tH^2(E(-2))$ is bijective, which implies 
that $\tH^2(E(-3)) = 0$ and Claim 1 is proven. 

\vskip2mm 

\noindent 
{\bf Claim 2.}\quad $F_{[h]}$ \emph{is as in item} (ii) \emph{above, for every 
hyperplane} $H \subset \piv$. 
 
\vskip2mm        

\noindent 
\emph{Indeed}, assume, by contradiction, that there exists a hyperplane 
$H \subset \piv$ such that $F_{[h]}$ can be realized as an extension 
$0 \ra M(2) \ra F_{[h]} \ra \text{T}_H(-1) \ra 0$, where $M$ is a stable rank 
2 vector bundle on $H$ with $c_1(M) = 0$, $c_2(M) = c_2 - 9$ and such that 
$\tH^1(M(-2)) = 0$ (which implies that $\tH^2(M(-2)) = 0$). Since 
$\tH^i(E(-3)) = 0$, $i = 2,\, 3$ (by Claim 1 and 
Remark~\ref{R:generalities}(c)), 
one gets that $\tH^2(E_H(-3)) \izo \tH^3(E(-4))$ hence $\h^3(E(-4)) = 1$. 
Using the notation from Remark~\ref{R:generalities}(c), it follows that 
$t = 1$ hence $r \leq 6$. 

Now, since $\tH^3(E(-4)) \neq 0$, Lemma~\ref{L:erat(-1)} implies that there 
exists an epimorphism $\e \colon E \ra \text{T}_\piv(-1)$.  
The kernel $K$ of $\e$ is a 
vector bundle of rank $\leq 2$ and, since $\text{T}_\piv(-1)_H \simeq \sco_H 
\oplus \text{T}_H(-1)$, one has $c_i(K) = c_i(M(2))$, $i = 1,\, 2,\, 3$. 
One deduces that $K = {\widetilde M}(2)$, where $\widetilde M$ is a rank 2 
vector bundle on $\piv$ with $c_1({\widetilde M}) = 0$, $c_2({\widetilde M}) 
= c_2 - 9$. Moreover, since $\tH^0(E(-2)) = 0$ one has $\tH^0({\widetilde M}) 
= 0$, i.e., $\widetilde M$ is stable. But, if $c_2 \in \{10,\, 11\}$, such 
a bundle \emph{cannot exist} because its Chern classes do not satisfy 
Schwarzenberger's congruence while, for $c_2 = 12$, it cannot exist according 
to a result of Barth and Elencwajg \cite{be} (which says that there is no 
stable rank 2 vector bundle on $\piv$ with $c_1 = 0$, $c_2 = 3$). 
Consequently, Claim 2 is proven. 

\vskip2mm 

\noindent 
{\bf Claim 3.}\quad \emph{If} $c_2 = 10$ \emph{then} $E \simeq 5\sco_\piv(1)$. 

\vskip2mm 

\noindent 
\emph{Indeed}, as we saw in the proof of Claim 1, $\tH^1(E(-2)) = 0$, 
$\tH^2(E(-3)) = 0$ and $c_4 = 5$. 
Using a formula from Remark~\ref{R:generalities}(b), one deduces that $r = 5$. 
Moreover, since $\tH^3(E(-4)) = 0$ (because $\tH^2(E_H(-3)) = 0$, for every 
hyperplane $H \subset \piv$, by Claim 2) and $\tH^4(E(-5)) \simeq 
\tH^0(E^\vee)^\vee = 0$, $E$ is $(-1)$-regular. In particular, $E(-1)$ is 
globally generated and $c_1(E(-1)) = 0$ hence $E(-1) \simeq 5\sco_\piv$. 

\vskip2mm 

\noindent 
{\bf Claim 4.}\quad \emph{The case} $c_2 = 11$ \emph{cannot occur}. 

\vskip2mm 

\noindent 
\emph{Indeed}, assume, by contradiction, that it does. We saw in the proof 
of Claim 1 that $\h^1(E_H(-2)) = 2$, for every hyperplane $H \subset \piv$, 
that $\tH^1(E(-3)) = 0$ and that $\tH^2(E(-3)) = 0$. One deduces that 
$\tH^1(E(-2)) \izo \tH^1(E_H(-2))$, $\tH^1(E_H(-3)) \izo \tH^2(E(-4))$ and 
$\tH^2(E_H(-3)) \izo \tH^3(E(-4))$ hence, taking into account Claim 2,  
$\h^1(E(-2)) = 2$, $\h^2(E(-4)) = 1$ and $\h^3(E(-4)) = 0$. 
Since 
$\h^4(E(-5)) = \h^0(E^\vee) = 0$, the Castelnuovo-Mumford lemma implies that 
$\tH^2(E(l)) = 0$ for $l \geq -3$. One also has $\h^2(E(-5)) = \h^2(E^\vee) 
= 0$ (by Cor.~\ref{C:h2eHvee=0}). 

Consider, now, the exact sequences$\, :$ 
\[
0 \ra \tH^0(E(-1)) \ra \tH^0(E_H(-1)) \ra \tH^1(E(-2)) \overset{h}{\ra} 
\tH^1(E(-1)) \ra \tH^1(E_H(-1)) \ra 0\, .  
\]  
The inequality $\h^1(E_H(-1)) \leq \max \, (\h^1(E_H(-2)) - 3 , 0)$ from 
the proof of Lemma~\ref{L:h0f(-1)c2=11} implies that $\tH^1(E_H(-1)) = 0$. 
Since this happens for every hyperplane $H \subset \piv$ and since 
$\h^1(E(-2)) = 2$, the Bilinear Map Lemma \cite[Lemma~5.1]{ha} implies that 
$\tH^1(E(-1)) = 0$. Moreover, by Riemann-Roch on $H$, $\h^0(E_H(-1)) = 2$. 
Taking into account that $\h^1(E(-2)) = 2$, it follows that $\tH^0(E(-1)) = 
0$. Applying Beilinson's theorem (recalled in \cite[Thm.~1.23]{acm1} and 
\cite[Remark~1.25]{acm1}) to $E(-1)$ one deduces an exact sequence  
$
0 \ra \Omega_\piv^3(3) \ra 2\Omega_\piv(1) \ra E(-1) \ra 0 
$.  
In order to get a \emph{contradiction} it suffices to prove the 
following$\, :$ 

\vskip2mm 

\noindent 
{\bf Subclaim 4.1.}\quad \emph{There is no locally split monomorphism} 
$\Omega_\piv^3(3) \ra 2\Omega_\piv(1)$. 

\vskip2mm 

\noindent 
\emph{Indeed}, according to Definition~\ref{D:contraction}, any morphism 
$\phi \colon \Omega_\piv^3(3) \ra 2\Omega_\piv(1)$ is defined by contraction 
with two elements $\omega ,\, \omega^\prim$ of $\bigwedge^2V$ (where $V = k^5$). 
We want to show that the dual morphism $\phi^\vee \colon 2\Omega_\piv(1)^\vee \ra 
\Omega_\piv^3(3)^\vee$ cannot be an epimorphism. Let 
$W$ be the subspace $\omega \wedge V + \omega^\prim \wedge V$ of 
$\bigwedge^3V$ (recall the description of $\tH^0(\phi^\vee)$ from the above 
mentioned definition). According to Lemma~\ref{L:vetter}, we have to show that 
$W^\perp$ contains a decomposable element of $\bigwedge^2V$. We consider, for 
that, only the generic case. More precisely, we assume that there exist two 
bases $u_0, \ldots , u_4$ and $u_0^\prime , \ldots , u_4^\prime$ of $V$ such that 
$\omega = u_0 \wedge u_1 + u_2 \wedge u_3$ and $\omega^\prim = u_0^\prime \wedge 
u_1^\prime + u_2^\prime \wedge u_3^\prime$. Moreover, putting $U := ku_0 + \cdots 
+ ku_3$ and $U^\prime := ku_0^\prime + \cdots + ku_3^\prime$, we assume that $U + 
U^\prime = V$. One has $\omega \wedge V = \bigwedge^3U + k(\omega \wedge u_4)$  
and   $\omega^\prim \wedge V = \bigwedge^3U^\prime + k(\omega^\prime \wedge 
u_4^\prime)$. Moreover, $(\bigwedge^3U)^\perp = \bigwedge^2U$, 
$(\bigwedge^3U^\prime)^\perp = \bigwedge^2U^\prime$ and $\bigwedge^2U \cap 
\bigwedge^2U^\prime = \bigwedge^2(U \cap U^\prime)$ which is a 3-dimensional 
vector subspace of $\bigwedge^2V$ consisting of decomposable elements. 
Exterior multiplications  by $\omega \wedge u_4$ and $\omega^\prim \wedge 
u_4^\prime$ define linear functions on $\bigwedge^2(U \cap U^\prime)$ hence there 
exists a non-zero element $\eta$ of $\bigwedge^2(U \cap U^\prime)$ such that 
$\eta \wedge \omega \wedge u_4 = 0$ and $\eta \wedge \omega^\prim \wedge 
u_4^\prime = 0$. $\eta$ belongs to $W^\perp$ and it is decomposable. This proves 
the subclaim and, with it, Claim 4. 

\vskip2mm 

\noindent 
{\bf Claim 5.}\quad \emph{The case} $c_2 = 12$ \emph{cannot occur}. 

\vskip2mm 

\noindent 
\emph{Indeed}, assume, by contradiction, that it does. Let $H \subset \piv$ 
be an arbitrary hyperplane.  Since $\tH^2(E_H(l)) = 0$, $\forall \, l \geq 
-2$, one gets, from Claim 1, that $\tH^2(E(l)) = 0$ for $l \geq -3$. 
Lemma~\ref{L:h0f(-1)c2=12} implies that $\h^0(E_H(-1)) \leq 1$ hence, by 
Lemma~\ref{L:h2f(-2)=0}(b), $\h^1(E_H(-1)) = 1 + \h^0(E_H(-1)) \leq 2$. The 
last inequality in Lemma~\ref{L:h2f(-2)=0}(b) implies, now, that $\h^1(E_H) 
= 0$. 

As we saw in the proof of Claim 1, $\h^1(E(-3)) = c_4/6$. Since 
$\tH^0(E_H(-2)) = 0$ and $\tH^2(E(-3)) = 0$ one gets that $\h^1(E(-2)) = 
\h^1(E(-3)) + \h^1(E_H(-2)) = \h^1(E(-3)) + 4$. Consider the exact 
sequence$\, :$ 
\[
0 \ra \tH^0(E(-1)) \ra \tH^0(E_H(-1)) \ra \tH^1(E(-2)) \overset{h}{\ra} 
\tH^1(E(-1)) \ra \tH^1(E_H(-1)) \ra 0\, .  
\]       
Since $\h^0(E_H(-1)) \leq 1$ and $\h^1(E_H(-1)) \leq 2$ the Bilinear Map Lemma 
implies that $\tH^0(E(-1)) = 0$ (recall that $H$ is an 
\emph{arbitrary} hyperplane). The above exact sequence shows, now, that$\, :$ 
\[
\h^1(E(-1)) - \h^1(E(-2)) = \h^1(E_H(-1)) - \h^0(E_H(-1)) = 1 
\]
hence $\h^1(E(-1)) = \h^1(E(-3)) + 5$. We want to evaluate, next, $\h^0(E)$ 
using the exact sequence$\, :$ 
\[
0 = \tH^0(E(-1)) \ra \tH^0(E) \ra \tH^0(E_H) \ra \tH^1(E(-1)) 
\overset{h}{\lra} \tH^1(E) \ra \tH^1(E_H) = 0\, . 
\]
Firstly, the Bilinear Map Lemma implies that 
$\h^1(E(-1)) - \h^1(E) \geq 4$ (recall, again, that $H$ is an arbitrary 
hyperplane). Secondly, by Riemann-Roch on $H$, $\h^0(E_H) = (r-1) + 8 = 
r + 7$ hence $\h^0(E) \leq \h^0(E_H) - 4 = r + 3$. Since $E$ is globally 
generated, there exists an epimorphism $(r + 3)\sco_\piv \ra E$. The kernel 
$K$ of this epimorphism is a rank 3 vector bundle. But an easy computation 
shows that $c_4(K) = -c_4 + c_2^2 + 2c_1c_3 - 3c_1^2c_2 + c_1^4$ which implies 
that $c_4(K) \neq 0$ because the first four terms are divisible by 6 (recall 
that $c_4 = 6\h^1(E(-3))$) while $c_1^4$ is not. This \emph{contradiction} 
concludes the proof of Claim 5.  
\end{proof} 

\begin{prop}\label{P:c1=5c2=11onp4} 
Let $E$ be a globally generated vector bundle on $\piv$, 
with Chern classes $c_1 = 5$, $c_2 = 11$, $c_3$, $c_4$, and such that 
${\fam0 H}^i(E^\vee) = 0$, $i = 0,\, 1$. Assume, also, that 
${\fam0 H}^0(E_H(-2)) = 0$, for every hyperplane $H \subset \piv$. Then one 
of the following holds$\, :$ 
\begin{enumerate} 
\item[(i)] $c_3 = 15$, $c_4 = 16$ and $E \simeq 4\sco_\piv(1) \oplus 
{\fam0 T}_\piv(-1)$$\, ;$ 
\item[(ii)] $c_3 = 13$, $c_4 = 9$ and $E \simeq 2\sco_\piv(1) \oplus 
\Omega_\piv(2)$. 
\end{enumerate}
\end{prop} 

\begin{proof} 
According to Lemma~\ref{L:c3geqc2} and to Prop.~\ref{P:c3=c2onp4}, one 
must have $c_3 \geq 13$ (recall that $c_3 \equiv c_1c_2 \mod{2}$). 
Let $H \subset \piv$ be an arbitrary hyperplane, of equation $h = 0$, and 
let $F_{[h]}$ be the vector bundle $P(P(E_H))$ on $H$ (see Remark~\ref{R:f}). 
Then, according to Remark~\ref{R:fprimstable}, 
$F_{[h]}$ can be realized as an extension$\, :$ 
\[
0 \lra (\text{rk}\, F_{[h]} - 3)\sco_H \lra F_{[h]} \lra G_{[h]}(2) \lra 0\, ,  
\] 
where $G_{[h]}$ is a \emph{stable} rank 3 vector bundle with $c_1(G_{[h]}) = 
-1$, $c_2(G_{[h]}) = 3$, $c_3(G_{[h]}) \geq 3$. One deduces that $G_{[h]}$ has  
one of the following spectra$\, :$ $(0 , -1 , -2)$, $(-1 , -1 , -1)$ and 
$(-1 , -1 , -2)$. 

If the spectrum of $G_{[h]}$ is $(-1 , -1 , -2)$, for at least one hyperplane 
$H \subset \piv$, then $c_3(G_{[h]}) = 5$ hence $c_3 = 15$. It is easy to 
show (see \cite[Prop.~3.4]{acm3}) that, in this case, $F_{[h]} \simeq 
4\sco_H(1) \oplus \text{T}_H(-1)$. One deduces, from Lemma~\ref{L:a+p(b)}, 
that $E \simeq 4\sco_\piv(1) \oplus \text{T}_\piv(-1)$. 

Similarly, if the spectrum of $G_{[h]}$ is $(-1 , -1 , -1)$, for at least one 
hyperplane $H \subset \piv$, then $c_3(G_{[h]}) = 3$ hence $c_3 = 13$ and 
$F_{[h]} \simeq 3\sco_H(1) \oplus \Omega_H(2)$ (by \cite[Prop.~3.4]{acm3}). 
Lemma~\ref{L:a+p(b)+omega(2)} implies that, in this case, $E \simeq 
2\sco_\piv(1) \oplus \Omega_\piv(2)$. 

It remains to investigate the case where $G_{[h]}$ has spectrum $(0 , -1 , 
-2)$, for every hyperplane $H \subset \piv$. \emph{We want, actually, to 
eliminate this case}. Assume, by contradiction, that it occurs. Then 
$c_3(G_{[h]}) = 3$ hence $c_3 = 13$. Moreover, $\text{rk}\, F_{[h]} = 3 + 
\h^2(G_{[h]}(-2)) = 6$ (see the last part of Remark~\ref{R:generalities}(c)). 
Since $\h^2(F_{[h]}(-3)) = \h^2(G_{[h]}(-1)) = 1$, one has $t \leq 1$ (see 
Remark~\ref{R:generalities}(c) for the notation) hence $E$ has rank $r \leq 7$. 

Now, one has $\h^1(E_H(-3)) = 1$ (use the spectrum). Moreover, by 
Lemma~\ref{L:h2f(-2)=0}(b), $\h^1(E_H(-2)) = 1$ and $\h^1(E_H(-1)) = 
\h^0(E_H(-1)) - 3$. But Lemma~\ref{L:h0f(-1)c2=11} implies that 
$\h^0(E_H(-1)) \leq 3$ hence $\h^0(E_H(-1)) = 3$ and $\h^1(E_H(-1)) = 0$. 
Remark~\ref{R:generalities}(d) implies that $\tH^1(E(-3)) = 0$ and that 
$\tH^2(E(l)) = 0$ for $l \geq -2$. The formula from 
Remark~\ref{R:generalities}(c) shows, now, that $\h^2(E(-3)) = (9 - c_4)/6$. 

\vskip2mm 

\noindent 
{\bf Claim 1.}\quad $\tH^3(E(-4)) = 0$. 

\vskip2mm 

\noindent 
\emph{Indeed}, assume, by contradiction, that $\tH^3(E(-4)) \neq 0$. Then, 
by Lemma~\ref{L:erat(-1)}, $E$ can be realized as an extension$\, :$ 
\[
0 \lra E_1 \lra E \lra \text{T}_\piv(-1) \lra 0\, , 
\] 
where $E_1$ is a vector bundle of rank $r - 4 \leq 3$. One must have 
$1 + c_1(E_1) + \cdots + c_i(E_1) = c_i$, $i = 1, \ldots , 4$, hence 
$c_1(E_1) = 4$, $c_2(E_1) = 6$, $c_3(E_1) = 2$ and, since $c_4(E_1) = 0$, 
$c_4 = 13$. But this \emph{contradicts} the formula $\h^2(E(-3)) = 
(9 - c_4)/6$ and Claim 1 is proven. 

\vskip2mm 

It follows, from Claim 1 and from the fact that $\tH^1(E(-3)) = 0$, that 
one has, for every hyperplane $H \subset \piv$, an exact sequence$\, :$ 
\[
0 \lra \tH^1(E_H(-3)) \lra \tH^2(E(-4)) \overset{h}{\lra} \tH^2(E(-3)) 
\lra \tH^2(E_H(-3)) \lra 0\, . 
\]
Since $\h^1(E_H(-3)) = 1$, $\h^2(E_H(-3)) = 1$ and $\h^2(E(-3)) = 
(9 - c_4)/6 \leq 1$, one gets that $\h^2(E(-3)) = 1$ and $\h^2(E(-4)) = 1$. 
Using the exact sequence$\, :$ 
\[
0 \lra \tH^1(E(-2)) \lra \tH^1(E_H(-2)) \lra \tH^2(E(-3)) \lra \tH^2(E(-2)) = 
0  
\]
and the fact, noticed above, that $\h^1(E_H(-2)) = 1$, 
one gets that $\tH^1(E(-2)) = 0$. Since $\tH^1(E_H(-1)) = 0$ it follows that 
$\tH^1(E(-1)) = 0$ and, moreover, $\tH^0(E(-1)) \izo \tH^0(E_H(-1))$ hence 
$\h^0(E(-1)) = 3$. 

Putting together the cohomological information obtained so far one deduces, 
applying Beilinson's theorem (recalled in \cite[Thm.~1.23]{acm1} and 
\cite[Remark~1.25]{acm1}) to $E(-1)$, that one has an exact sequence$\, :$ 
\[
0 \lra \Omega_\piv^3(3) \lra 3\sco_\piv \oplus \Omega_\piv^2(2) \lra E(-1) 
\lra 0\, .   
\]   
In order to get the desired \emph{contradiction} it suffices to prove the 
following$\, :$ 

\vskip2mm 

\noindent 
{\bf Claim 2.}\quad \emph{There is no locally split monomorphism} 
$\Omega_\piv^3(3) \ra 3\sco_\piv \oplus \Omega_\piv^2(2)$. 

\vskip2mm 

\noindent 
\emph{Indeed}, according to Definition~\ref{D:contraction}, any morphism 
$\phi \colon \Omega_\piv^3(3) \ra 3\sco_\piv \oplus \Omega_\piv^2(2)$ is defined 
by contraction with three elements $\omega_1,\, \omega_2,\, \omega_3$ 
of $\bigwedge^3V$ and with a vector $v_0 \in V$ (where $V = k^5$).    
We want to show that the dual morphism $\phi^\vee \colon 3\sco_\piv \oplus 
\Omega_\piv^2(2)^\vee \ra \Omega_\piv^3(3)^\vee$ cannot be an epimorphism. Let 
$W$ be the subspace $v_0 \wedge \bigwedge^2V + \sum k\omega_i$ of 
$\bigwedge^3V$ (recall the description of $\tH^0(\phi^\vee)$ from the above 
mentioned definition). According to Lemma~\ref{L:vetter}, we have to show that 
$W^\perp$ contains a decomposable element of $\bigwedge^2V$.
One has $(v_0 \wedge \bigwedge^2V)^\perp = 
V \wedge v_0$. Exterior multiplication to the right by $\omega_i$ defines a 
linear function on $V \wedge v_0$, $i = 1,\, 2,\, 3$. Since $V \wedge v_0$ 
has dimension 4, there exists $v_1 \in V \setminus kv_0$ such that 
$v_1 \wedge v_0 \wedge \omega_i = 0$, $i = 1, \, 2,\, 3$. It follows that 
$W^\perp$ contains the decomposable element $v_1 \wedge v_0$. This concludes the 
proof of Claim 2 and, with it, of the proposition.      
\end{proof} 

\begin{lemma}\label{L:c2=12c3=14onp4} 
There exists no globally generated vector bundle $E$ on $\piv$, with Chern 
classes $c_1 = 5$, $c_2 = 12$, $c_3 = 14$, $c_4$, such that 
${\fam0 H}^i(E^\vee) = 0$, $i = 0,\, 1$, and ${\fam0 H}^0(E_H(-2)) = 0$, for 
every hyperplane $H \subset \piv$. 
\end{lemma}

\begin{proof} 
Assume, by contradiction, that such a bundle exists. Since $\h^1(E_H(-2)) = 3$ 
(by Lemma~\ref{L:h2f(-2)=0}(b)), Remark~\ref{R:generalities}(d) implies that 
$\tH^2(E(l)) = 0$ for $l \geq -2$, and that $\tH^1(E(-3)) = 0$. It follows, 
from the formula in Remark~\ref{R:generalities}(c), that $\h^2(E(-3)) = 
(7 - c_4)/6$. One deduces that, for every hyperplane $H \subset 
\piv$ of equation $h = 0$, one has an exact sequence$\, :$ 
\[
0 \lra \tH^1(E(-2)) \lra \tH^1(E_H(-2)) \lra \tH^2(E(-3)) \lra 0\, . 
\] 
Since $\h^2(E(-3)) \leq 1$ it follows that $2 \leq \h^1(E(-2)) \leq 3$ (recall 
that $\h^1(E_H(-2)) = 3$). Consider, now, the exact sequence$\, :$ 
\[
0 \ra \tH^0(E(-1)) \ra \tH^0(E_H(-1)) \ra \tH^1(E(-2)) \overset{h}{\ra} 
\tH^1(E(-1)) \ra \tH^1(E_H(-1)) \ra 0\, . 
\] 
Since $\h^1(E_H(-1)) = \h^0(E_H(-1)) \leq 1$, by Lemma~\ref{L:h2f(-2)=0}(b) and 
Lemma~\ref{L:h0f(-1)c2=12}, the Bilinear Map Lemma \cite[Lemma~5.1]{ha} 
implies that $\tH^0(E(-1)) = 0$. One deduces that $\h^1(E(-1)) = \h^1(E(-2))$ 
and that the multiplication by any non-zero linear form $h \colon \tH^1(E(-2)) 
\ra \tH^1(E(-1))$ has corank $\leq 1$. Applying, now, Lemma~\ref{L:bml} to the 
map $\tH^0(\sco_\piv(1)) \ra \text{Hom}_k(\tH^1(E(-2)) , \tH^1(E(-1)))$ one gets 
a \emph{contradiction} (recall that $\h^1(E(-2)) \in \{2,\, 3\}$).   
\end{proof} 

\begin{prop}\label{P:c1=5c2=12onp4} 
Let $E$ be a globally generated vector bundle on $\piv$ with Chern classes 
$c_1 = 5$, $c_2 = 12$, $c_3$, $c_4$, such that ${\fam0 H}^i(E^\vee) = 0$, 
$i = 0,\, 1$, and ${\fam0 H}^0(E_H(-2)) = 0$, for every hyperplane $H \subset 
\piv$. Then one of the following holds$\, :$ 
\begin{enumerate} 
\item[(i)] $c_3 = 20$, $c_4 = 28$ and $E \simeq 3\sco_\piv(1) \oplus 
2{\fam0 T}_\piv(-1)$$\, ;$ 
\item[(ii)] $c_3 = 18$, $c_4 = 21$ and $E \simeq \sco_\piv(1) \oplus 
{\fam0 T}_\piv(-1) \oplus \Omega_\piv(2)$$\, ;$ 
\item[(iii)] $c_3 = 18$, $c_4 = 15$ and $E \simeq 2\sco_\piv(1) \oplus 
\Omega_\piv^2(3)$$\, ;$ 
\item[(iv)] $c_3 = 16$, $c_4 = 8$ and $E \simeq \sco_\piv(1) \oplus E_0$, 
where $E_0(-1)$ is the cohomology sheaf of a monad of the form$\, :$ 
\[
0 \lra \Omega_\piv^3(3) \lra \Omega_\piv^2(2) \oplus \Omega_\piv^1(1) \lra 
\sco_\piv \lra 0\, ,  
\] 
i.e., $E_0(-1)$ is the bundle $\mathcal{G}$ from the statement of the main 
result of the paper of Abo, Decker and Sasakura \emph{\cite{ads}}$\, ;$ 
\item[(v)] $c_3 = 16$, $c_4 = 8$ and one has an exact sequence$\, :$ 
\[
0 \lra \Omega_\piv^3(3) \lra \Omega_\piv^2(2) \oplus \Omega_\piv^1(1) \lra 
E(-1) \lra 0\, . 
\]
\end{enumerate}
\end{prop}

\begin{proof} 
According to Lemma~\ref{L:c3geqc2}, Prop.~\ref{P:c3=c2onp4} and 
Lemma~\ref{L:c2=12c3=14onp4} one must have $c_3 \geq 16$ (recall that 
$c_3 \equiv c_1c_2 \pmod{2}$). If $H \subset \piv$ is an arbitrary hyperplane, 
of equation $h = 0$, let $F_{[h]}$ denote the vector bundle $P(P(E_H))$ on $H$ 
(see Remark~\ref{R:f}). By Lemma~\ref{L:h2f(-2)=0}(b), 
\[
\h^0(F_{[h]}(-1)) - \h^1(F_{[h]}(-1)) = \frac{1}{2}(c_3 - 14) 
\] 
hence $\h^0(F_{[h]}(-1)) \geq 2$ if $c_3 \geq 18$. In this case, by 
Remark~\ref{R:h0f(-1)c2=12}, either $c_3 = 20$ and $F_{[h]} \simeq 3\sco_H(1) 
\oplus 2\text{T}_H(-1)$ or $c_3 = 18$ and $F_{[h]} \simeq 2\sco_H(1) \oplus 
\text{T}_H(-1) \oplus \Omega_H(2)$. In the former case one deduces, from 
Lemma~\ref{L:a+p(b)}, that $E$ is as in item (i) of the statement while, 
in the latter case, $E$ is as in item (ii) or in item (iii) of the statement, 
by Lemma~\ref{L:a+p(b)+omega(2)}. It thus remains to consider the case $c_3 = 
16$. In this case, by Remark~\ref{R:fprimstable}, $F_{[h]}$ can be realized as 
an extension$\, :$ 
\[
0 \lra (\text{rk}\, F_{[h]} - 3) \lra F_{[h]} \lra G_{[h]}(2) \lra 0\, , 
\] 
where $G_{[h]}$ is a stable rank 3 vector bundle on $H$ with $c_1(G_{[h]}) = 
-1$, $c_2(G_{[h]}) = 4$, $c_3(G_{[h]}) = 4$. Taking into account 
Remark~\ref{R:spectrumg}, the possible spectra of $G_{[h]}$ are $(0 , -1 , -1 
, -2)$ and $(-1 , -1 , -1 , -1)$. In both cases $\h^2(G_{[h]}(-2)) = 4$ hence 
$F_{[h]}$ has rank 7 (see the last part of Remark~\ref{R:generalities}(c)). One 
also has, by Lemma~\ref{L:h2f(-2)=0}(b), $\h^1(E_H(-2)) = 2$ (and 
$\h^1(E_H(-1)) = \h^0(E_H(-1)) - 1$) hence, by Remark~\ref{R:generalities}(d), 
$\tH^1(E(-3)) = 0$ and  $\tH^2(E(l)) = 0$ for $l \geq -2$. One deduces, from 
the formula in Remark~\ref{R:generalities}(c), that $\h^2(E(-3)) = 
(14 - c_4)/6$. One gets an exact sequence$\, :$ 
\[
0 \lra \tH^1(E(-2)) \lra \tH^1(E_H(-2)) \lra \tH^2(E(-3)) \lra 0  
\] 
hence $\h^1(E(-2)) + \h^2(E(-3)) = 2$. 

\vskip2mm 

\noindent 
{\bf Claim 1.}\quad $\tH^3(E(-4)) = 0$. 

\vskip2mm 

\noindent 
\emph{Indeed}, assume, by contradiction, that $\tH^3(E(-4)) \neq 0$. For 
every hyperplane $H \subset \piv$, one has an exact sequence$\, :$ 
\[
0 \ra \tH^1(E_H(-3)) \ra \tH^2(E(-4)) \overset{h}{\ra} \tH^2(E(-3)) \ra 
\tH^2(E_H(-3)) \ra \tH^3(E(-4)) \ra 0\, . 
\] 
Since $\h^2(E_H(-3)) = \h^2(G_{[h]}(-1)) \leq 1$ (use the spectrum) one gets  
that $\h^3(E(-4)) = \h^2(E_H(-3)) = 1$ (hence, in particular, $G_{[h]}$ has 
spectrum $(0 , -1 , -1 , -2)$). It follows that the multiplication by 
any non-zero linear form $h \colon \tH^2(E(-4)) \ra \tH^2(E(-3))$ is 
surjective. Since $\h^1(E_H(-3)) = 1$, the Bilinear Map Lemma 
\cite[Lemma~5.1]{ha} implies that $\tH^2(E(-3)) = 0$ hence $c_4 = 14$. 
Moreover, one gets that $\h^2(E(-4)) = \h^1(E_H(-3)) = 1$ hence, by 
Cor.~\ref{C:h2eHvee=0}, $\tH^2(E^\vee) = 0$. Using a formula from 
Remark~\ref{R:generalities}(b), one deduces that $E$ has rank $r = 7$.  

Now, the assumption $\tH^3(E(-4)) \neq 0$ implies, by Lemma~\ref{L:erat(-1)}, 
that $E$ can be realized as an extension$\, :$ 
\[
0 \lra E_1 \lra E \lra \text{T}_\piv(-1) \lra 0\, , 
\] 
where $E_1$ is a vector bundle of rank $r - 4 = 3$. One gets that $c_4(E_1) = 
c_4 - c_3 = -2$ and this \emph{contradicts} the fact that $E_1$ has rank 3. 
This contradiction proves the claim. 

\vskip2mm 

One deduces, from Claim 1, that one has, for every hyperplane $H \subset 
\piv$, an exact sequence$\, :$ 
\[
0 \lra \tH^1(E_H(-3)) \lra \tH^2(E(-4)) \overset{h}{\lra} \tH^2(E(-3)) 
\lra \tH^2(E_H(-3)) \lra 0\, . 
\]   
Since $\h^1(E_H(-3)) = \h^2(E_H(-3)) \leq 1$ (use the spectrum), one gets that 
$\h^2(E(-4)) = \h^2(E(-3))$. 

\vskip2mm 

\noindent 
{\bf Claim 2.}\quad $\tH^2(E(-3)) \neq 0$. 

\vskip2mm 

\noindent 
\emph{Indeed}, assume, by contradiction, that $\tH^2(E(-3)) = 0$. It follows 
that $\h^1(E_H(-3)) = \h^2(E_H(-3)) = 0$, for every hyperplane $H \subset 
\piv$. Moreover, using the formula preceding Claim 1, $\h^1(E(-2)) = 
2$. Consider, for an arbitrary hyperplane $H \subset \piv$, the exact 
sequence$\, :$ 
\[
0 \ra \tH^0(E(-1)) \ra \tH^0(E_H(-1)) \ra \tH^1(E(-2)) \overset{h}{\ra} 
\tH^1(E(-1)) \ra \tH^1(E_H(-1)) \ra 0\, . 
\] 
Since $\h^1(E_H(-3)) = 0$, the last assertion in Lemma~\ref{L:h0f(-1)c2=12} 
implies that $\h^0(E_H(-1)) \leq 1$ hence, actually, $\h^0(E_H(-1)) = 1$ and 
$\h^1(E_H(-1)) = 0$ (recall that $\h^1(E_H(-1)) = \h^0(E_H(-1)) - 1$, by 
Lemma~\ref{L:h2f(-2)=0}(b)). Because this happens for every hyperplane $H 
\subset \piv$, the Bilinear Map Lemma \cite[Lemma~5.1]{ha} implies that 
$\tH^1(E(-1)) = 0$ and this clearly \emph{contradicts} the fact that 
$\h^1(E(-2)) = 2$ and $\h^0(E_H(-1)) = 1$.   

\vskip2mm 

Consequently, one has $\h^2(E(-4)) = \h^2(E(-3)) \in \{1\, ,\, 2\}$. Since 
the multiplication by any non-zero linear form $h \colon \tH^2(E(-4)) \ra 
\tH^2(E(-3))$ has corank $\leq 1$ one must have $\h^2(E(-4)) = \h^2(E(-3)) 
= 1$ (there is no injective linear map $k^5 \ra \text{Hom}_k(k^2 , k^2)$). 
One deduces that $c_4 = 8$ and that $\h^1(E(-2)) = 1$ (by the formula 
preceding Claim 1). 
The last assertion in Lemma~\ref{L:h0f(-1)c2=12} 
implies that $\h^0(E_H(-1)) \leq 2$ hence $\h^1(E_H(-1)) = \h^0(E_H(-1)) - 1 
\leq 1$, for every hyperplane $H \subset \piv$. Using the exact sequence from 
the proof of Claim 2 and the Bilinear Map Lemma one deduces easily that one 
must have $\h^1(E(-1)) \leq 1$. One also deduces that $\h^0(E(-1)) =  
\h^1(E(-1))$ (because $\h^1(E(-2)) = 1$).  
The cohomological information obtained so far suffices 
to conclude that the Beilinson monad of $E(-1)$ has one of the forms$\, :$ 
\begin{gather*} 
0 \lra \Omega_\piv^3(3) \lra \Omega_\piv^2(2) \oplus \Omega_\piv^1(1) 
\lra 0 \lra 0\, ,\\ 
0 \lra \Omega_\piv^3(3) \overset{\alpha}{\lra} \Omega_\piv^2(2) \oplus 
\Omega_\piv^1(1) \oplus \sco_\piv  \overset{\beta}{\lra} \sco_\piv \lra 0  
\end{gather*}  
(with the direct sums as the term of cohomological degree 0). If the 
Beilinson monad of $E(-1)$ has the first form then $E$ is as in item (v) of 
the statement. 

Assume, finally, that the Beilinson monad of $E(-1)$ has the second form. By 
the basic properties of Beilinson monads, the component $\beta_2 \colon 
\sco_\piv \ra \sco_\piv$ of $\beta$ is 0. It follows that the component 
$\beta_1 \colon \Omega_\piv^2(2) \oplus \Omega_\piv^1(1) \ra \sco_\piv$ of 
$\beta$ is an epimorphism. Since $E$ is globally generated, $\Ker \beta(1)$ 
must be globally generated hence $\Ker \beta_1(1)$ is globally generated. 
Cor.~\ref{C:sasakura} implies that there exists a $k$-basis $v_0 , \ldots , 
v_4$ of $V := k^5$ such that $\beta_1$ is defined by contraction with 
$\omega := v_0 \wedge v_1 + v_2 \wedge v_3$ and with $v := v_4$. The component 
$\alpha_1 \colon \Omega_\piv^3(3) \ra \Omega_\piv^2(2) \oplus \Omega_\piv^1(1)$ 
of $\alpha$ is defined by contraction with a $w \in V$ and an $\eta \in 
\bigwedge^2V$. The condition $\beta_1 \circ \alpha_1 = 0$ is equivalent to 
$w \wedge \omega + \eta \wedge v = 0$ in $\bigwedge^3V$. Put $V^\prime := kv_0 
+ \cdots + kv_3 \subset V$. Since $\bigwedge^3V = \bigwedge^3V^\prime \oplus  
(\bigwedge^2V^\prime \wedge v_4)$ and since $\ast \wedge \omega$ maps $V^\prime$ 
isomorphically onto $\bigwedge^3V^\prime$, one deduces that one must have 
$w = -cv_4$, for some $c \in k$. This implies that $\eta = c\omega + u \wedge 
v_4$, for some $u \in V^\prime$. 

Now, since there is no locally split monomorphism $\Omega_\piv^3(3) \ra 
\Omega_\piv^1(1) \oplus \sco_\piv$ (the cokernel of such a monomorphism would 
be isomorphic to $\sco_\piv(2)$) it follows that $c \neq 0$. 
$\tH^0(\alpha_1^\vee)$ can be identified to the map $\bigwedge^2V \oplus V \ra 
\bigwedge^3V$ defined by $w \wedge \ast$ and $\eta \wedge \ast$. One deduces 
that $\tH^0(\alpha_1^\vee)$ is surjective (because its image contains 
$v_4 \wedge \bigwedge^2V^\prime$ and $\omega \wedge V^\prime = 
\bigwedge^3V^\prime$) hence $\alpha_1^\vee$ is an epimorphism hence $\alpha_1$ is 
a locally split monomorphism. One thus gets a monad$\, :$ 
\[
0 \lra \Omega_\piv^3(3) \overset{\alpha_1}{\lra} \Omega_\piv^2(2) \oplus 
\Omega_\piv^1(1)  \overset{\beta_1}{\lra} \sco_\piv \lra 0\, . 
\]
Let $E_1$ be the cohomology sheaf of this monad ($E_1$ is, of course, locally 
free). One gets an exact sequence$\, :$ 
\[
0 \lra \sco_\piv \lra E(-1) \lra E_1 \lra 0\, . 
\]
Since $\tH^0(\alpha_1^\vee)$ is surjective it follows that $\tH^1(E_1^\vee) = 0$ 
hence $E(-1) \simeq \sco_\piv \oplus E_1$ hence $E$ is as in item (iv) of the 
statement. 
\end{proof}

\section{The case $c_1 = 5$ on $\p^n$, 
$n \geq 5$}\label{S:c1=5onp5} 

We classify, in this section, the globally generated vector bundles $E$ with 
$c_1 = 5$ on $\p^n$, $n \geq 5$, with the property that $\tH^i(E^\vee) = 0$, 
$i = 0,\, 1$, and that $\tH^0(E_\Pi(-2)) = 0$ for every $3$-plane $\Pi 
\subset \p^n$. We use the analogous classification for vector bundles on 
$\piv$ from the preceding section and the following two auxiliary results.   

\begin{lemma}\label{L:horrocks} 
Consider  a morphism $\phi \colon \Omega_\pv^3(3) \ra \Omega_\pv^1(1)$ defined 
by contraction with an element $\omega$ of $\bigwedge^2V$, where $V := k^6$ 
$($see Definition~\emph{\ref{D:contraction}}$)$. Then the following 
assertions are equivalent$\, :$ 
\begin{enumerate} 
\item[(i)] $\phi$ is an epimorphism$\, ;$ 
\item[(ii)] There exists a $k$-basis $v_0 , \ldots , v_5$ of $V$ such that 
$\omega = v_0 \wedge v_1 + v_2 \wedge v_3 + v_4 \wedge v_5$$\, ;$ 
\item[(iii)] ${\fam0 H}^0(\phi(1))$ is bijective. 
\end{enumerate} 
\end{lemma}

\begin{proof} 
$\tH^0(\phi(1)) \colon \tH^0(\Omega_\pv^3(4)) \ra \tH^0(\Omega_\pv^1(2))$ is 
the map $\ast \llcorner \, \omega \colon \bigwedge^4V^\vee \ra 
\bigwedge^2V^\vee$ which can be identified with the map $\ast \wedge \omega 
\colon \bigwedge^2V \ra \bigwedge^4V$. Let $W$ be the subspace $\bigwedge^2V 
\wedge \omega$ of $\bigwedge^4V$. By Lemma~\ref{L:vetter}, $\phi(1)$ is an 
epimorphism if and only if the subspace $W^\perp$ of $\bigwedge^2V$ contains 
no decomposable element. 

If $\omega = v_0 \wedge v_1$, with $v_0,\, v_1 \in V$ linearly independent 
then $W^\perp$ contains the element $v_0 \wedge v_1$. 

If $\omega = v_0 \wedge v_1 + v_2 \wedge v_3$, with $v_0 , \ldots , v_3 \in V$ 
linearly independent then $W^\perp$ contains $v_0 \wedge v_2$. 

One deduces that if $\phi$ is an epimorphism then there exists a $k$-basis 
$v_0 , \ldots , v_5$ of $V$ such that $\omega = v_0 \wedge v_1 + v_2 \wedge v_3 
+ v_4 \wedge v_5$. We assert that, in this case, $W = \bigwedge^4V$. 

\emph{Indeed}, any subset of $\{0 , \ldots , 5\}$ consisting of 4 elements 
contains one of the subsets $\{0,\, 1\}$, $\{2,\, 3\}$, $\{4,\, 5\}$. If it 
contains, for example, $\{0,\, 1\}$ and the other two elements $i,\, j$ belong 
one to $\{2,\, 3\}$ and the other one to $\{4,\, 5\}$ then$\, :$ 
\[
v_0 \wedge v_1 \wedge v_i \wedge v_j = v_i \wedge v_j \wedge \omega \in W\, . 
\]  
On the other hand, one has$\, :$ 
\begin{gather*} 
W \ni v_0 \wedge v_1 \wedge \omega = v_0 \wedge v_1 \wedge v_2 \wedge v_3 
+ v_0 \wedge v_1 \wedge v_4 \wedge v_5\, ,\\
W \ni v_2 \wedge v_3 \wedge \omega = v_0 \wedge v_1 \wedge v_2 \wedge v_3 
+ v_2 \wedge v_3 \wedge v_4 \wedge v_5\, ,\\ 
W \ni v_4 \wedge v_5 \wedge \omega = v_0 \wedge v_1 \wedge v_4 \wedge v_5  
+ v_2 \wedge v_3 \wedge v_4 \wedge v_5\, , 
\end{gather*}
hence $v_0 \wedge v_1 \wedge v_2 \wedge v_3$, $v_0 \wedge v_1 \wedge v_4 \wedge 
v_5$ and $v_2 \wedge v_3 \wedge v_4 \wedge v_5$ belong to $W$ (one uses the 
fact that $\text{char}\, k \neq 2$). 
\end{proof} 

\begin{cor}\label{C:horrocks} 
Consider a morphism $\phi \colon \Omega_\pv^2(2) \ra \sco_\pv$ defined by 
contraction with an $\omega \in \bigwedge^2V$, where $V := k^6$. 

\emph{(a)} $\phi$ is an epimorphism if and only if there exists a $k$-basis 
$v_0 , \ldots , v_5$ of $V$ such that either $\omega = v_0 \wedge v_1 + v_2 
\wedge v_3$ or $\omega = v_0 \wedge v_1 + v_2 \wedge v_3 + v_4 \wedge v_5$. 

\emph{(b)} If $\phi$ is an epimorphism then ${\fam0 Ker}\, \phi(1)$ is 
globally generated if and only if there exists a $k$-basis $v_0, \ldots , v_5$ 
of $V$ such that $\omega = v_0 \wedge v_1 + v_2 \wedge v_3 + v_4 \wedge v_5$.  
\end{cor}

\begin{proof} 
(a) $\phi$ is an epimorphism if and only if $\tH^0(\phi(1))$ is surjective, 
i.e., if and only if the contraction mapping $\ast \llcorner \, \omega \colon 
\bigwedge^3V^\vee \ra V^\vee$ is surjective. On the other hand, this mapping 
can be identified with $\ast \wedge \omega \colon \bigwedge^3V \ra 
\bigwedge^5V$. If there is a basis $v_0 , \ldots , v_5$ of $V$ such that 
$\omega = v_0 \wedge v_1$ then $v_1 \wedge \ldots \wedge v_5$ does not belong 
to $\bigwedge^3V \wedge \omega$. 

(b) One uses the same kind of argument as in the proof of 
Cor.~\ref{C:sasakura}. 
\end{proof} 

\begin{prop}\label{P:c1=5onp5} 
Let $E$ be a globally generated vector bundle on $\p^n$, $n \geq 5$, with  
$c_1 = 5$, $c_2 \leq 12$, such that ${\fam0 H}^i(E^\vee) = 0$, 
$i = 0,\, 1$, and ${\fam0 H}^0(E_\Pi(-2)) = 0$, for every $3$-plane $\Pi 
\subset \p^n$. Then one of the following holds$\, :$ 
\begin{enumerate} 
\item[(i)] $c_2 = 10$ and $E \simeq 5\sco_{\p^n}(1)$$\, ;$ 
\item[(ii)] $c_2 = 11$ and $E \simeq 4\sco_{\p^n}(1) \oplus 
{\fam0 T}_{\p^n}(-1)$$\, ;$ 
\item[(iii)] $c_2 = 12$ and $E \simeq 3\sco_{\p^n}(1) \oplus 
2{\fam0 T}_{\p^n}(-1)$$\, ;$ 
\item[(iv)] $n = 5$, $c_2 = 11$ and $E \simeq \sco_\pv(1) \oplus 
\Omega_\pv(2)$$\, ;$ 
\item[(v)] $n = 6$, $c_2 = 11$ and $E \simeq \Omega_\pvi(2)$$\, ;$
\item[(vi)] $n = 5$, $c_2 = 12$ and $E \simeq {\fam0 T}_\pv(-1) \oplus 
\Omega_\pv(2)$$\, ;$ 
\item[(vii)] $n = 5$, $c_2 = 12$ and one has an exact sequence$\, :$ 
\[
0 \lra \Omega_\pv^4(4) \lra \Omega_\pv^2(2) \lra E(-1) \lra 0\, .
\]
\end{enumerate}
\end{prop} 

\begin{proof}
According to Lemma~\ref{L:c3geqc2}, Prop.~\ref{P:c3=c2onp4}, 
Prop.~\ref{P:c1=5c2=11onp4} and Prop.~\ref{P:c1=5c2=12onp4} the pair of 
Chern classes $(c_2 , c_3)$ of $E$ must take one of the values $(10 , 10)$,  
$(11 , 15)$, $(11 , 13)$, $(12 , 20)$, $(12 , 18)$, $(12 , 16)$. 
Let $\Pi \subset \p^n$ be a 3-plane and let $F$ be the vector bundle 
$P(P(E_\Pi))$ on $\Pi$ (see Remark~\ref{R:f}). Taking into account the 
precise description of the globally generated vector bundles on $\piv$ from 
the above mentioned results, one sees that, for the first five possible pairs 
of Chern classes, $F$ is isomorphic to one of the bundles$\, :$ $5\sco_\Pi(1)$, 
$4\sco_\Pi(1) \oplus \text{T}_\Pi(-1)$, $3\sco_\Pi(1) \oplus \Omega_\Pi(2)$, 
$3\sco_\Pi(1) \oplus 2\text{T}_\Pi(-1)$, $2\sco_\Pi(1) \oplus \text{T}_\Pi(-1) 
\oplus \Omega_\Pi(2)$. It follows, from Lemma~\ref{L:a+p(b)} and 
Lemma~\ref{L:a+p(b)+omega(2)}, that, in the first five cases, $E$ is as in one 
of the items (i)--(vi) from the statement. 

Assume, from now on, that $(c_2 , c_3) = (12 , 16)$. If $H \subset \p^n$ is 
an arbitrary hyperplane, of equation $h = 0$, then, as we noticed in 
Remark~\ref{R:f}, the (globally generated) vector bundle $F_{[h]} := P(P(E_H))$ 
on $H$ satisfies $\tH^i(F_{[h]}^\vee) = 0$, $i = 0,\, 1$, and $E_H \simeq 
t\sco_H \oplus Q_{[h]}$ where $t = \h^0(E_H^\vee)$ and $Q_{[h]}$ is a quotient of 
$F_{[h]}$ by a trivial subbundle $s\sco_H$, where $s = \h^1(E_H^\vee)$. 

\vskip2mm 

\noindent 
{\bf Case 1.}\quad $n = 5$ (\emph{and, of course,} $c_2 = 12$, $c_3 = 16$). 

\vskip2mm 

\noindent
In this case, by Prop.~\ref{P:c1=5c2=12onp4}, $F_{[h]}(-1)$ is the cohomology 
sheaf of a monad of the form$\, :$ 
\[
0 \lra \Omega_H^3(3) \lra \Omega_H^2(2) \oplus \Omega_H^1(1) \oplus \sco_H 
\lra \sco_H \lra 0\, ,  
\]   
in which the component $\sco_H \ra \sco_H$ of the differential from the right 
can be non-zero. One deduces the following cohomological information about 
$E_H$$\, :$ 
\begin{enumerate} 
\item[(1)] $\tH^1(E_H(l)) = 0$ for $l \leq -3$ and $l \geq 0$, $\h^1(E_H(-2)) 
= 1$, $\h^1(E_H(-1)) = \h^0(E_H(-1)) \leq 1$$\, ;$ 
\item[(2)] $\tH^2(E_H(l)) = 0$ for $l \leq -5$ and $l \geq -2$, $\h^2(E_H(-4)) 
= \h^2(E_H(-3)) = 1$$\, ;$ 
\item[(3)] $\tH^3(E_H(l)) = 0$ for $l \geq -4$, $\h^3(E_H(-5)) = \h^1(E_H^\vee) 
= s$. 
\end{enumerate} 
Moreover, since $F_{[h]}$ has rank 6 and $c_4(F_{[h]}) = 8 \neq 0$, one must 
have $s \leq 2$. 

Now, since $\tH^5(E(-6)) \simeq \tH^0(E^\vee)^\vee = 0$ it follows that 
$\tH^5(E(l)) = 0$ for $l \geq -6$. 

One gets, from (3), that $\tH^4(E(l)) = 0$ for $l \geq -5$. Moreover, 
$\tH^4(E(-6)) \simeq \tH^1(E^\vee)^\vee = 0$. 

Using the exact sequence$\, :$ 
\[
\tH^2(E_H(-4)) \lra \tH^3(E(-5)) \overset{h}{\lra} \tH^3(E(-4)) \lra 
\tH^3(E_H(-4)) = 0 
\]  
and the Bilinear Map Lemma \cite[Lemma~5.1]{ha} (recall that the hyperplane 
$H$ is arbitrary) one deduces that $\tH^3(E(-4)) = 0$. Together with (3) 
this implies that $\tH^3(E(l)) = 0$ for $l \geq -4$. Moreover, using the 
exact sequence$\, :$ 
\[
0 = \tH^2(E_H(-5)) \lra \tH^3(E(-6)) \overset{h}{\lra} \tH^3(E(-5)) 
\lra \tH^3(E_H(-5)) 
\]  
and the Bilinear Map Lemma one gets that $\tH^3(E(-6)) = 0$. 

It follows, from (2), that $\tH^2(E(l)) = 0$ for $l \leq -5$. Using the 
exact sequence$\, :$ 
\[
\tH^1(E_H(-2)) \lra \tH^2(E(-3)) \overset{h}{\lra} \tH^2(E(-2)) \lra 
\tH^2(E_H(-2)) = 0 
\] 
and the Bilinear Map Lemma one deduces that $\tH^2(E(-2)) = 0$. Together 
with (2) this implies that $\tH^2(E(l)) = 0$ for $l \geq -2$. 

One gets, from (1), that $\tH^1(E(l)) = 0$, for $l \leq -3$. Using the exact 
sequence$\, :$ 
\[
0 = \tH^1(E(-3)) \ra \tH^2(E(-4)) \overset{h}{\ra} \tH^2(E(-3)) \ra 
\tH^2(E_H(-3)) \ra \tH^3(E(-4)) = 0 
\]
and the Bilinear Map Lemma one deduces that $\tH^2(E(-4)) = 0$ and 
$\tH^2(E(-3)) \izo \tH^2(E_H(-3))$ hence $\h^2(E(-3)) = 1$. Since 
$\tH^i(E(-4)) = 0$, $i = 2,\, 3$, it follows that $\tH^2(E_H(-4)) \izo 
\tH^3(E(-5))$ hence $\h^3(E(-5)) = 1$. 

Finally, using the exact sequence$\, :$ 
\[
0 = \tH^1(E(-3)) \ra \tH^1(E(-2)) \ra \tH^1(E_H(-2)) \ra \tH^2(E(-3)) \ra 
\tH^2(E(-2)) = 0 
\]
and recalling that $\h^1(E_H(-2)) = 1 = \h^2(E(-3))$, one obtains that 
$\tH^1(E(-2)) = 0$. Since $\tH^i(E(-2)) = 0$, $i = 0,\, 1,\, 2$, it follows 
that $\tH^i(E(-1)) \izo \tH^i(E_H(-1))$, $i = 0,\, 1$. 

We have gathered enough cohomological information to conclude that the 
Beilinson monad of $E(-1)$ has one of the following two forms$\, :$ 
\begin{gather*} 
0 \lra \Omega_\pv^4(4) \lra \Omega_\pv^2(2) \lra 0 \lra 0\, ,\\ 
0 \lra \Omega_\pv^4(4) \overset{\alpha}{\lra} \Omega_\pv^2(2) \oplus \sco_\pv 
\overset{\beta}{\lra} \sco_\pv \lra 0\, . 
\end{gather*}
If the monad of $E(-1)$ has the first form then $E$ is as in item (vii) from 
the statement. We assert that $E(-1)$ \emph{cannot have a monad of the second 
form}. 
\emph{Indeed}, assume, by contradiction, that it does. Since, by the basic 
properties of the Beilinson monad, the component $\beta_2 \colon \sco_\pv \ra 
\sco_\pv$ of $\beta$ is 0, the component $\beta_1 \colon  \Omega_\pv^2(2) \ra 
\sco_\pv$ must be an epimorphism. $E$ globally generated implies that 
$\Ker \beta(1)$ is globally generated hence $\Ker \beta_1(1)$ is globally 
generated. Cor.~\ref{C:horrocks}(b) implies that there exists a $k$-basis 
$v_0 , \ldots , v_5$ of $V := k^6$ such that $\beta_1$ is defined by 
contraction with $\omega := v_0 \wedge v_1 + v_2 \wedge v_3 + v_4 \wedge v_5 
\in \bigwedge^2V$. The component $\alpha_1 \colon \Omega_\pv^4(4) \ra 
\Omega_\pv^2(2)$ is defined by contraction with an element $\eta$ of 
$\bigwedge^2V$. The condition $\beta_1 \circ \alpha_1 = 0$ is equivalent to 
$\eta \wedge \omega = 0$ (in $\bigwedge^4V$). But, as we saw in the final 
part of the proof of Lemma~\ref{L:horrocks}, $\ast \wedge \omega \colon 
\bigwedge^2V \ra \bigwedge^4V$ is bijective hence $\eta = 0$. Since there is 
no locally split monomorphism $\Omega_\pv^4(4) \ra \sco_\pv$ we have got the 
desired \emph{contradiction}. 

\vskip2mm 

\noindent 
{\bf Case 2.}\quad $n \geq 6$ (\emph{and, of course,} $c_2 = 12$, $c_3 = 16$). 

\vskip2mm 

\noindent  
We will show that this case \emph{cannot occur}. Assume, by contradiction, 
that it does. We can suppose, of course, that $n = 6$. Using the notation 
from the beginning of the proof, Case 1 implies that one has an exact 
sequence$\, :$ 
\[
0 \lra \Omega_H^4(4) \lra \Omega_H^2(2) \lra F_{[h]}(-1) \lra 0\, ,
\]    
for every hyperplane $H \subset \pvi$. It follows that $\tH^1_\ast(E_H) = 0$ 
and this implies that $\tH^i_\ast(E) = 0$, $i = 1,\, 2$. Using the exact 
sequence$\, :$ 
\[
\tH^2(E_H(-4)) \lra \tH^3(E(-5)) \overset{h}{\lra} \tH^3(E(-4)) \lra 
\tH^3(E_H(-4)) 
\] 
and the fact that $\tH^i(E_H(-4)) = 0$, $i = 2,\, 3$, for every hyperplane 
$H \subset \pvi$, one gets that $\tH^3(E(-5)) = 0$ and $\tH^3(E(-4)) = 0$. 
But this \emph{contradicts} the fact that $\tH^2(E_H(-3))$, which is 
1-dimensional, injects into $\tH^3(E(-4))$ (because $\tH^2(E(-3)) = 0$). 
\end{proof}

\appendix 
\section{Overview of the case $c_1 = 5$ on 
$\piii$}\label{A:c1=5onp3} 

We explain in this appendix, for ease of reference, the method used in 
\cite{acm3} to classify globally generated vector bundles with $c_1 = 5$ on 
$\piii$. Most of the results are of a technical nature but the way in which 
the method works effectively can be seen in the proof of 
Lemma~\ref{L:(1,0,0,-1)} below.  

Firstly, let us recall the following result, which is a particular case of 
\cite[Prop.~3.5]{acm2}, and for which a short self-contained proof can be 
found in \cite[Appendix~A]{acm3}. 

\begin{prop}\label{P:h0f(-2)neq0} 
Let $F$ be a globally generated vector bundle on $\piii$, with Chern classes 
$c_1 = 5$, $c_2$, $c_3$, and such that ${\fam0 H}^i(F^\vee) = 0$, $i = 0,\, 1$. 
If ${\fam0 H}^0(F(-3)) = 0$ and ${\fam0 H}^0(F(-2)) \neq 0$ then either 
$\sco_\piii(2)$ is a direct summand of $F$ or $F \simeq M(3)$, for some stable 
rank $2$ vector bundle $M$ with $c_1(M) = -1$, $c_2(M) = 2$ \emph{(}in which 
case $c_2(F) = 8$\emph{)}. 
\end{prop} 

\begin{proof} 
This result is proven in \cite[Prop.~A.1]{acm3} under the hypothesis $c_2 
\leq 12$. The case $c_2 \geq 13$ is, however, easy. Indeed, the dependency 
locus of $r - 1$ general global sections of $F$ is a nonsingular (but not, 
necessarily, connected) curve $Y$, whence one gets an exact sequence$\, :$ 
\[
0 \lra (r-1)\sco_\piii \lra F \lra \sci_Y(5) \lra 0\, . 
\]
The degree of $Y$ is $c_2$. Since $\sci_Y(5)$ is globally generated and 
$\tH^0(\sci_Y(3)) \neq 0$ it follows that $Y$ is contained in a complete 
intersection of type $(3 , 5)$. One deduces that either $Y$ is a complete 
intersection of type $(3 , 5)$ or $c_2 \leq 14$ and $Y$ is directly linked by 
a complete intersection of type $(3 , 5)$ to a (locally Cohen-Macaulay) curve 
$Y^\prime$ of degree $15 - c_2$. 

In the former case, one gets an exact sequence$\, :$ 
\[
0 \lra \sco_\piii(-3) \lra \sco_\piii(2) \oplus r\sco_\piii \lra F \lra 0\, .  
\] 
Since $\tH^i(F^\vee) = 0$, $i = 0,\, 1$, it follows that, by dualizing the 
exact sequence, the map $\tH^0(r\sco_\piii) \ra \tH^0(\sco_\piii(3))$ is 
bijective. One deduces, easily, that $\sco_\piii(2)$ is a direct summand of $F$. 

In the latter case, one gets, from the exact sequence of liaison (recalled in 
\cite[Remark~2.6]{acm1})$\, :$ 
\[
0 \lra \sco_\piii(-8) \lra \sco_\piii(-3) \oplus \sco_\piii(-5) \lra \sci_Y 
\lra \omega_{Y^\prime}(-4) \lra 0\, , 
\] 
that $\omega_{Y^\prime}(1)$ is globally generated. If $c_2 \geq 13$ then 
$\text{deg}\, Y^\prime \in \{1,\, 2\}$. The condition $\omega_{Y^\prime}(1)$ 
globally generated implies that $Y^\prime$ has degree 2 and it is a complete 
intersection of type $(1 , 2)$ or a double structure on a line $L \subset 
\piii$. Such a double structure is defined by an exact sequence $0 \ra 
\sci_{Y^\prime} \ra \sci_L \ra \sco_L(l) \ra 0$, for some $l \geq -1$. It is 
well known that one has $\omega_{Y^\prime} \simeq \sco_{Y^\prime}(-l-2)$ hence 
$\omega_{Y^\prime}(1)$ globally generated implies that $l = -1$, i.e., $Y^\prime$ 
is a complete intersection of type $(1 , 2)$ in this case, too. Using a result 
of Ferrand about resolutions under liaison (also recalled in 
\cite[Remark~2.6]{acm1}) one gets a resolution$\, :$ 
\[
0 \lra \sco_\piii(-2) \oplus \sco_\piii(-1) \lra 
2\sco_\piii \oplus \sco_\piii(2) \lra \sci_Y(5) \lra 0 
\] 
and one concludes as in the case where $Y$ is a complete intersection of 
type $(3 , 5)$. 
\end{proof} 

Now, using Prop.~\ref{P:h0f(-2)neq0}, \cite[Prop.~2.4]{acm1} and 
\cite[Prop.~2.10]{acm1} one sees that, in order to classify globally 
generated vector bundles $F$ on $\piii$ with $c_1 = 5$, one can assume that 
$\tH^0(F(-2)) = 0$. The next result provides some preliminary cohomological 
information about such a bundle. 

\begin{lemma}\label{L:h2f(-2)=0} 
Let $F$ be a globally generated vector bundle on $\piii$ of rank $r \geq 3$, 
with Chern classes $c_1 = 5$, $c_2 \leq 12$, $c_3$, and such that 
${\fam0 H}^i(F^\vee) = 0$, $i = 0,\, 1$. Then$\, :$ 

\emph{(a)} ${\fam0 H}^1(F(l)) = 0$ for $l \leq -5$.  

\emph{(b)} If, moreover, ${\fam0 H}^0(F(-2)) = 0$ then ${\fam0 H}^2(F(l)) = 
0$, for $l \geq -2$, and one has$\, :$ 
\[
{\fam0 h}^1(F(-2)) = \frac{1}{2}(5(c_2 - 8) - c_3)\, ,\  
{\fam0 h}^1(F(-1)) = \frac{1}{2}(7(c_2 - 10) - c_3) + {\fam0 h}^0(F(-1))\, , 
\]
and ${\fam0 h}^1(F) \leq \max\, ({\fam0 h}^1(F(-1)) - 3 , 0)$.    
\end{lemma}

\begin{proof} 
(a) The dependency locus of $r - 1$ general global 
sections of $F$ is a nonsingular curve $Y$ of degree $c_2$. One gets an exact 
sequence$\, :$ 
\[
0 \lra (r-1)\sco_\piii \lra F \lra \sci_Y(5) \lra 0\, . 
\]
According to \cite[Lemma~1.1]{acm3}, $Y$ is \emph{connected} hence 
$\tH^1(F(l)) = 0$ for $l \leq -5$. 

(b) $r - 3$ general global sections of $F$ define an exact sequence$\, :$ 
\[
0 \lra (r - 3)\sco_\piii \lra F \lra F^\prim \lra 0\, ,  
\]
with $F^\prim$ a rank $3$ vector bundle. Consider the \emph{normalized} rank 3 
vector bundle $G := F^\prim(-2)$. It has Chern classes $c_1(G) = -1$, $c_2(G) = 
c_2 - 8$, $c_3(G) = c_3 - 2c_2 + 12$. Using the exact sequence$\, :$ 
\[
0 \lra (r - 3)\sco_\piii \lra F \lra G(2) \lra 0 
\] 
and its dual one deduces that $\tH^0(G) = 0$ and $\tH^0(G^\vee(-2)) = 0$. 
If $\tH^0(G^\vee(-1)) = 0$ then $G$ is \emph{stable}. In this case, according 
to the restriction theorem of Schneider \cite{sch} (see, also, Ein et al.  
\cite[Thm.~3.4]{ehv}) either $G \simeq \Omega_\piii(1)$ (in which case $c_2(G) 
= 1$) or the restriction $G_H$ of $G$ to a \emph{general} plane $H \subset 
\piii$ is stable (in which case $c_2(G) \geq 2$). In the latter case 
$\tH^0(G_H) = 0$ hence $\tH^0(F_H(-2)) = 0$. 

If $\tH^0(G^\vee(-1)) \neq 0$, a non-zero global section  
of $G^\vee(-1)$ defines a non-zero morphism $\phi : G \ra 
\sco_\piii(-1)$. The image of $\phi$ is of the form $\sci_Z(-1)$, where $Z$ 
is a closed subscheme of $\piii$, of codimension $\geq 2$ (because 
$\tH^0(G^\vee(-2)) = 0$).  
Since $G(2)$ is globally generated, $\sci_Z(1)$ globally generated, 
hence $Z$ must be the empty set, a simple point or a line. But 
$c_3(G^\vee(-1)) = -c_3 + c_2 -4 \equiv 0 \pmod{2}$ (because $c_3 \equiv 
c_1c_2 \equiv c_2 \pmod{2}$). One deduces that $Z$ cannot be a simple point 
hence $G$ can be realized as an extension of one of the following forms$\, :$ 
\begin{gather*} 
(\text{A})\ \  0 \ra M \ra G \ra \sco_\piii(-1) \ra 0\, ,\\   
(\text{B})\ \  0 \ra M \ra G \ra \sci_L(-1) \ra 0\, ,  
\end{gather*}
where $M$ is a rank 2 vector bundle with $c_1(M) = 0$ and $\tH^0(M) = 0$ 
(hence it is stable) and $L$ is a line in $\piii$. Moreover, $c_2(M) = c_2 - 
8$ and $c_3 = c_2 - 4$ in case (A) while in case (B), $c_2(M) = c_2 - 9$ and 
$c_3 = c_2$.  
According to the restriction theorem of Barth \cite{b} (see, also, Ein et al. 
\cite[Thm.~3.3]{ehv}) either $M$ can be described by an exact sequence 
$0 \ra \sco_\piii(-1) \ra \Omega_\piii(1) \ra M \ra 0$ (in which case $c_2(M) 
= 1$; these bundles are called \emph{nullcorrelation} bundles) or the 
restriction $M_H$ of $M$ to a general plane $H \subset \piii$ is stable (in
 which case $c_2(M) \geq 2$). In the latter case $\tH^0(M_H) = 0$ hence 
$\tH^0(F_H(-2)) = 0$. 

Now, with the above notation, if $G \simeq \Omega_\piii(1)$ (resp., if $M$ is 
a nullcorrelation bundle) then $\tH^2(G) = 0$ (resp., $\tH^2(M) = 0$). It 
follows that, in order to prove that $\tH^2(F(-2)) = 0$, one can assume that 
$\tH^0(F_H(-2)) = 0$, for the \emph{general} plane $H \subset \piii$. 
Consider, for an \emph{arbitrary} plane $H \subset \piii$, the exact 
sequence$\, :$ 
\[
\tH^1(F_H(-2)) \lra \tH^2(F(-3)) \overset{h}{\lra} \tH^2(F(-2)) \lra 
\tH^2(F_H(-2))\, . 
\]   
One has $\tH^2(F_H(-2)) \simeq \tH^0(F_H^\vee(-1))^\vee = 0$ (since $F_H$ is 
globally generated, $F_H^\vee$ embeds into a direct sum of copies of $\sco_H$). 
Applying the Bilinear Map Lemma \cite[Lemma~5.1]{ha} one deduces that if 
$\tH^2(F(-2)) \neq 0$ then $\h^2(F(-3)) - \h^2(F(-2)) \geq 3$. But, for a 
general plane $H \subset \piii$, one has, by Riemann-Roch, $\h^1(F_H(-2)) 
= c_2 - 10 \leq 2$ and this \emph{contradiction} shows that, in fact, 
$\tH^2(F(-2)) = 0$. Since $\tH^3(F(-3)) \simeq \tH^0(F^\vee(-1))^\vee = 0$, the 
Castelnuovo-Mumford lemma (in the slightly more general form quoted in 
\cite[Lemma~1.21]{acm1}) implies that $\tH^2(F(l)) = 0$, $\forall \, l \geq 
-2$.  

The next two relations from item (b) of the statement can be deduced from the  
Riemann-Roch formula (recalled in \cite[Thm.~4.5]{acm1}).    
Finally, since $\tH^1(F_H) = 0$, for every plane $H \subset 
\piii$ (by the proof of \cite[Prop.~3.6]{acm1}), it follows that the 
multiplication by any non-zero linear form $h \colon \tH^1(F(-1)) \ra \tH^1(F)$ 
is surjective hence, using again the Bilinear Map Lemma, one gets the last 
inequality from the statement.    
\end{proof}
  
Prop.~\ref{P:fprimunstable} below shows that, except for a few cases in 
which the bundle $F$ can be explicitly described, the rank 3 vector bundle $G$ 
associated to $F$ in the proof of Lemma~\ref{L:h2f(-2)=0}(b) is 
\emph{stable}. This reduces the classification of globally generated vector 
bundles $F$ on $\piii$ with $c_1 = 5$ and $\tH^0(F(-2)) = 0$ to the 
classification of stable rank 3 vector bundles $G$ on $\piii$, with $c_1(G) = 
-1$, $c_2(G) \leq 4$ and such that $G(2)$ is globally generated. 
In order to prove Prop.~\ref{P:fprimunstable} one needs two auxiliary results.  

\begin{lemma}\label{L:h0f(-1)c2=11} 
Let $F$ be a globally generated vector bundle on $\piii$ with Chern classes 
$c_1 = 5$, $c_2 = 11$, $c_3$, and such that ${\fam0 H}^0(F(-2)) = 0$. 
Then$\, :$ 
\[
{\fam0 h}^0(F(-1)) \leq \max\left(\frac{1}{2}(c_3 - 7)\, ,\, 
1\right)\, . 
\]
\end{lemma} 

\begin{proof} 
It follows, from the description of globally generated vector bundles with 
$c_1 = 5$, $c_2 = 11$ on $\pii$ from the proof of \cite[Prop.~3.6]{acm1}, 
that $\tH^1(F_H(-1)) = 0$, for every plane $H \subset \piii$. 
Using the exact sequences $\tH^1(F(-2)) \overset{h}{\lra} 
\tH^1(F(-1)) \ra \tH^1(F_H(-1)) = 0$ and applying the Bilinear Map Lemma, one 
gets that$\, :$ 
\[
\h^1(F(-1)) \leq \max\, (\h^1(F(-2)) - 3\, ,\, 0)\, . 
\]  
On the other hand, by Lemma~\ref{L:h2f(-2)=0}(b)$\, :$ 
\begin{gather*}
\h^1(F(-2)) = \frac{1}{2}(15 - c_3)\  \text{and} \ 
\h^0(F(-1)) = \frac{1}{2}(c_3 - 7) + \h^1(F(-1))\, . 
\end{gather*}
The inequality from the statement is now clear. 
\end{proof}

\begin{lemma}\label{L:h0f(-1)c2=12} 
Let $F$ be a globally generated vector bundle on $\piii$ with Chern classes 
$c_1 = 5$, $c_2 = 12$, $c_3$, and such that ${\fam0 H}^0(F(-2)) = 0$. If 
${\fam0 h}^0(F(-1)) \geq 2$ then $c_3 \in \{16,\, 18,\, 20\}$. Moreover, if 
$c_3 = 16$ then ${\fam0 h}^1(F(-3)) = 1$ and ${\fam0 h}^0(F(-1)) = 2$.   
\end{lemma}   

\begin{proof} 
Let $r$ be the rank of $F$. As we saw in the proof of 
Lemma~\ref{L:h2f(-2)=0}(a), $F$ can be realized as an extension$\, :$  
\[
0 \lra (r-1)\sco_\piii \lra F \lra \sci_Y(5) \lra 0\, , 
\]
with $Y$ a nonsingular connected curve of degree $c_2 = 12$.  
Our hypotheses imply that $\tH^0(\sci_Y(3)) = 0$ 
and $\h^0(\sci_Y(4)) \geq 2$. It follows that $Y$ is directly linked, by a 
complete intersection of type $(4,4)$, to a curve $Y^\prime$ of degree 4.  
Since $2\, \text{deg}\, Y > 4 \times 4$, $Y^\prime$ must be locally complete 
intersection except at finitely many points, where it is locally 
Cohen-Macaulay. The fundamental exact sequence of liaison (recalled in 
\cite[Remark~2.6]{acm1})$\, :$ 
\[
0 \lra \sco_\piii(-8) \lra 2\sco_\piii(-4) \lra \sci_Y \lra \omega_{Y^\prime}(-4) 
\lra 0 
\] 
implies that $\omega_{Y^\prime}(1)$ is globally generated. It follows that a 
general global section of $\omega_{Y^\prime}(1)$ generates this sheaf except at 
finitely many points hence it defines an extension$\, :$ 
\[
0 \lra \sco_\piii(-2) \lra \scg \lra \sci_{Y^\prime}(1) \lra 0
\]
with $\scg$ a rank 2 reflexive sheaf with $c_1(\scg) = -1$, $c_2(\scg) = 
\text{deg}\, Y^\prime - 2 = 2$ (see \cite[Thm.~4.1]{ha}). Since $\chi(\scg) 
= \chi(\sci_{Y^\prime}(1)) = \chi(\sco_\piii(1)) - \chi(\sco_{Y^\prime}(1))$ and 
$\chi(\sco_{Y^\prime}(1)) = \text{deg}\, Y^\prime + \chi(\sco_{Y^\prime})$,  
the Riemann-Roch formula for $\chi(\scg)$ (see, for example, 
\cite[Thm.~4.5]{acm1}) implies that that $c_3(\scg) = 4 - 
2\chi(\sco_{Y^\prime})$. One can show, similarly, that $c_3 = -12 - 
2\chi(\sco_Y)$. 
On the other hand, by a basic formula in liaison theory (recalled in the 
footnote on page 24 in \cite{acm1}), one has$\, :$ 
\[
\chi(\sco_{Y^\prime}) - \chi(\sco_Y) = 
\frac{1}{2}(4 + 4 - 4)(\text{deg}\, Y - \text{deg}\, Y^\prime) = 16\, . 
\]   
It follows that $c_3 = c_3(\scg) + 16$. 

Now, if $\tH^0(\sci_{Y^\prime}(1)) \neq 0$ then $Y^\prime$ is a complete 
intersection of type $(1,4)$, hence $\omega_{Y^\prime} \simeq \sco_{Y^\prime}(1)$. 
It follows that $\tH^0(\omega_{Y^\prime}(-1)) \neq 0$, hence $\tH^0(\sci_Y(3)) 
\neq 0$, a \emph{contradiction}. 

It remains that $\tH^0(\sci_{Y^\prime}(1)) = 0$ hence $\scg$ is \emph{stable}. 
\cite[Thm.~8.2(b)]{ha} implies, now, that $c_3(\scg) \in \{0,\, 2,\, 4\}$ 
hence $c_3 \in \{16,\, 18,\, 20\}$. 

Assume, finally, that $c_3 = 16$. In this case $c_3(\scg) = 0$ hence $\scg$ is 
a rank 2 vector bundle. These bundles have been studied, independently, by 
Hartshorne and Sols \cite{hs} and by Manolache \cite{ma}. One has 
$\tH^1(F(-3)) \simeq \tH^1(\sci_Y(2))$ and, by the well known behaviour of the 
Hartshorne-Rao module $\tH^1_\ast(\sci_C)$ ($C$ space curve) under liason, 
$\tH^1(\sci_Y(2)) \simeq \tH^1(\sci_{Y^\prime}(2))^\vee$. But 
$\h^1(\sci_{Y^\prime}(2)) = \h^1(\scg(1)) = 1$ (see, for example, 
\cite[Prop.~2.2]{hs}) hence $\h^1(F(-3)) = 1$. Moreover, dualizing the 
extension defining $\scg$ and using the fact that $\tH^1(\scg(-2)) = 0$ one 
gets that $\tH^0(\omega_{Y^\prime}) = 0$ hence $\h^0(\sci_Y(4)) = 2$ hence 
$\h^0(F(-1)) = 2$.    
\end{proof} 

\begin{remark}\label{R:h0f(-1)c2=12} 
Since the rank 2 reflexive sheaves $\scg$ appearing in the proof of 
Lemma~\ref{L:h0f(-1)c2=12} can be described concretely, one gets (see 
\cite[Prop.~4.1]{acm3}) that if $F$ is a globally generated vector bundle on 
$\piii$ with Chern classes $c_1 = 5$, $c_2 = 12$, $c_3$, such that 
$\tH^i(F^\vee) = 0$, $i = 0,\, 1$, $\tH^0(F(-2)) = 0$ and $\h^0(F(-1)) \geq 2$ 
then one of the following holds$\, :$ 
\begin{enumerate} 
\item[(i)] $c_3 = 20$ and $F \simeq 3\sco_\piii(1) \oplus 
2{\fam0 T}_\piii(-1)$$\, ;$ 
\item[(ii)] $c_3 = 18$ and $F \simeq 2\sco_\piii(1) \oplus {\fam0 T}_\piii(-1) 
\oplus \Omega_\piii(2)$$\, ;$ 
\item[(iii)] $c_3 = 16$ and $F \simeq \sco_\piii(1) \oplus F_0$, where, up to a 
linear change of coordinates, $F_0$ is the cohomology of the monad$\, :$ 
\[
\sco_\piii(-1) \xra{\left(\begin{smallmatrix} s\\ u 
\end{smallmatrix}\right)} 2\sco_\piii(2) \oplus 
2\sco_\piii(1) \oplus 4\sco_\piii \xra{(p\, ,\, 0)} \sco_\piii(3) 
\] 
where $\sco_\piii(-1) \overset{s}{\ra} 2\sco_\piii(2) \oplus 2\sco_\piii(1) 
\overset{p}{\ra} \sco_\piii(3)$ is a subcomplex of the Koszul complex defined 
by $x_0,\, x_1,\, x_2^2,\, x_3^2$ and $u : \sco_\piii(-1) \ra 4\sco_\piii$ is 
defined by $x_0, \ldots ,x_3$.  
\end{enumerate}
\end{remark}

\begin{prop}\label{P:fprimunstable} 
Let $F$ be a globally generated vector bundle of rank $r \geq 3$ on $\piii$, 
with Chern classes $c_1 = 5$, $c_2 \leq 12$, $c_3$, such that 
${\fam0 H}^i(F^\vee) = 0$, $i = 0,\, 1$. Assume, also, that 
${\fam0 H}^0(F(-2)) = 0$. As we saw in the proof of 
Lemma~\emph{\ref{L:h2f(-2)=0}(b)}, $F$ can be realized as an extension$\, :$ 
\[
0 \lra (r - 3)\sco_\piii \lra F \lra G(2) \lra 0\, ,  
\]
where $G$ is a rank $3$ vector bundle with $c_1(G) = -1$, $c_2(G) = c_2 - 8$, 
$c_3(G) = c_3 - 2c_2 + 12$. If $G$ is not stable then one of the following 
holds$\, :$ 
\begin{enumerate} 
\item[(i)] $r = 3$, $c_3 = c_2 - 4$, and $F$ can be realized as an 
extension$\, :$ 
\[
0 \lra M(2) \lra F \lra \sco_\piii(1) \lra 0\, , 
\] 
where $M$ is a rank $2$ vector bundle with $c_1(M) = 0$, $c_2(M) = c_2 - 8$, 
${\fam0 H}^0(M) = 0$ and ${\fam0 H}^1(M(-2)) = 0$ $($i.e., $M$ is a 
mathematical instanton bundle of charge $c_2 - 8$$)$$\, ;$ 
\item[(ii)] $r = 4$, $c_2 = 12$, $c_3 = c_2 - 4 = 8$, and $F \simeq 
\sco_\piii(1) \oplus F_0$, where $F_0$ is the kernel of an epimorphism 
$4\sco_\piii(2) \ra \sco_\piii(4)$, and the image of the morphism $\sco_\piii 
\ra F$ is contained in $F_0$$\, ;$  
\item[(iii)] $r = 5$, $c_3 = c_2$ and there exists an exact sequence$\, :$ 
\[
0 \lra M(2) \lra F \lra {\fam0 T}_\piii(-1) \lra 0\, , 
\]
where $M$ is a rank $2$ vector bundle with $c_1(M) = 0$, $c_2(M) = c_2 - 9$, 
${\fam0 H}^0(M) = 0$ and ${\fam0 H}^1(M(-2)) = 0$. 
\end{enumerate} 
\end{prop}

\begin{proof} 
Firstly, since $\tH^i(F(-4)) \simeq \tH^{3-i}(F^\vee)^\vee = 0$, 
$i = 2,\, 3$, one has $\tH^2(G(-2)) \izo \tH^3((r-3)\sco_\piii(-4))$ hence 
$r = 3 + \h^2(G(-2))$. 

If $G$ is not stable then one has the \emph{alternatives} (A) and (B) from the 
proof of Lemma~\ref{L:h2f(-2)=0}(b).   
It is well known (see \cite[Thm.~8.1(c)]{ha}) that if $\scf$ is a rank 2 
reflexive sheaf on $\piii$ with $c_1(\scf) = 0$, $c_2(\scf) \leq 2$ and 
$\tH^0(\scf) = 0$ then $\tH^1(\scf(-2)) = 0$.  

\vskip2mm 

\noindent 
{\bf Claim 1.}\quad \emph{If} $c_2 = 11$ \emph{then the bundle} $M$ 
\emph{from} (A) \emph{satisfies} $\tH^1(M(-2)) = 0$. 

\vskip2mm 

\noindent 
\emph{Indeed}, $M$ has, \emph{a priori} two possible spectra$\, :$ 
$(1 , 0 , -1)$ and $(0 , 0 , 0)$ (see \cite[Sect.~7]{ha} for the definition 
and the properties of the spectrum of a stable rank 2 reflexive sheaf on 
$\piii$). But if $M$ has spectrum $(1 , 0 , -1)$ then $\h^0(M(1)) = 2$ (see 
\cite[Lemma~9.15]{ha}) and this \emph{contradicts} the fact that, by 
Lemma~\ref{L:h0f(-1)c2=11}, $\h^0(G(1)) \leq 1$ (because $\h^0(F(-1)) \leq 
1$). 

\vskip2mm 

\noindent 
{\bf Claim 2.}\quad \emph{If the vector bundle} $M$ \emph{from} (A) 
\emph{satisfies} $\tH^1(M(-2)) = 0$ \emph{then} $r = 3$, \emph{i.e.,} 
$F = G(2)$. 

\vskip2mm 

\noindent 
\emph{Indeed}, $\tH^2(M(-2)) = 0$ by Serre duality and the fact that $M \simeq 
M^\vee$. It follows that $\tH^2(G(-2)) = 0$ hence $r = 3$ by the formula 
from the beginning of the proof. 

\vskip2mm 

\noindent 
{\bf Claim 3.}\quad \emph{If} $c_2 = 12$ \emph{and the bundle} $M$ 
\emph{from} (A) \emph{satisfies} $\tH^1(M(-2)) \neq 0$ \emph{then} $F$ 
\emph{is as in item} (ii) \emph{from the statement}.  

\vskip2mm 

\noindent 
\emph{Indeed}, in this case, $M$ has spectrum $(1,0,0,-1)$. According to Chang 
\cite[Prop.~1.5]{ch}, either $M$ has an unstable plane $H$ of order 1 or 
it can be realized as the cohomology sheaf of a selfdual monad$\, :$ 
\[
0 \lra \sco_\piii(-2) \lra 4\sco_\piii \lra \sco_\piii(2) \lra 0\, . 
\] 
The former case cannot, however, occur because, in that case, there exists an 
epimorphism $M \ra \sci_{Z,H}(-1) \ra 0$ where $Z$ is a 0-dimensional 
subscheme of $H$, of length 5, and this would \emph{contradict} the fact that 
$M(3)$ must be globally generated (since $G(2)$ is globally generated, the 
diagram of evaluation morphisms corresponding to the exact sequence (A) 
tensorized by $\sco_\piii(2)$ induces an epimorphism from $\Omega_\piii(1)$, 
which is the kernel of the evaluation morphism of $\sco_\piii(1)$, to the 
cokernel of the evaluation morphism of $M(2)$). 

It thus remains that $M$ is the cohomology of a monad as above. Let $K$ 
be the kernel of the epimorphism $4\sco_\piii \ra \sco_\piii(2)$ from the 
monad. $K$ admits a (Koszul) resolution of the form$\, :$ 
\[
0 \lra \sco_\piii(-6) \lra 4\sco_\piii(-4) \lra 6\sco_\piii(-2) \lra K \lra 0\, . 
\]     
One deduces that $\tH^1(M(1)) \simeq \tH^3(\sco_\piii(-5))$ and 
$\tH^1(M(2)) \simeq \tH^3(\sco_\piii(-4)) \simeq k$. It follows that the 
multiplication map $\tH^1(M(1)) \otimes_k \tH^0(\sco_\piii(1)) \ra 
\tH^1(M(2))$ is a 
perfect pairing, that is, if $\xi \in \tH^1(M(1))$ is annihilated by 
every linear form $h \in \tH^0(\sco_\piii(1))$ then $\xi = 0$. 
Since $G(2)$ is globally generated, the map $\tH^0(G(2)) \ra 
\tH^0(\sco_\piii(1))$ must be surjective hence the connecting map  
$\tH^0(\sco_\piii(1)) \ra \tH^1(M(2))$ associated to the exact sequence (A) 
tensorized by $\sco_\piii(2)$ is zero. This implies that the element $\xi \in 
\tH^1(M(1))$ defining the extension $0 \ra M(1) \ra G(1) \ra \sco_\piii 
\ra 0$ is zero hence $G \simeq \sco_\piii(-1) \oplus M$. Since 
$\h^2(G(-2)) = \h^2(M(-2)) = 1$ one has $r = 4$. Since 
$\text{Ext}^1(M(2) , \sco_\piii) \simeq \tH^1(M^\vee(-2)) \simeq \tH^1(M(-2))$ 
is 1-dimensional and since the extension $0 \ra \sco_\piii \ra K(2) \ra M(2) 
\ra 0$ is non-trivial, one gets that $F \simeq \sco_\piii(1) \oplus K(2)$.  

\vskip2mm 

\noindent 
{\bf Claim 4.}\quad \emph{If} $c_2 = 12$ \emph{then the bundle} $M$ 
\emph{from} (B) \emph{satisfies} $\tH^1(M(-2)) = 0$. 

\vskip2mm 

\noindent 
\emph{Indeed}, one can use the same argument as in the proof of Claim 1 with 
Lemma~\ref{L:h0f(-1)c2=12} instead of Lemma~\ref{L:h0f(-1)c2=11}. 

\vskip2mm 

\noindent 
{\bf Claim 5.}\quad \emph{If} $G$ \emph{is as in} (B) \emph{then} $F$ 
\emph{is as in item} (iii) \emph{from the statement}. 

\vskip2mm 

\noindent 
\emph{Indeed}, since $G(2)$ is globally generated the map $\tH^0(G(2)) 
\ra \tH^0(\sci_L(1))$ must be surjective. Applying the Snake Lemma to the 
diagram$\, :$ 
\[
\SelectTips{cm}{12}\xymatrix{0\ar[r] & \sco_\piii(-1)\ar[r]\ar[d] & 
2\sco_\piii\ar[r]\ar[d] & \sci_L(1)\ar[r]\ar @{=}[d] & 0\\
0\ar[r] & M(2)\ar[r] & G(2)\ar[r] & \sci_L(1)\ar[r] & 0} 
\]  
one gets an exact sequence$\, :$ 
\[
0 \lra \sco_\piii(-1) \lra M(2) \oplus 2\sco_\piii \lra G(2) \lra 0\, . 
\]
Since $\tH^2(M(-2)) = 0$ it follows that $\h^2(G(-2)) = 2$ hence $F$ has 
rank $r = 5$. Using the fact that $\text{Ext}^1(M(2) , \sco_\piii) \simeq 
\tH^1(M^\vee(-2)) \simeq \tH^1(M(-2)) = 0$ one gets a commutative diagram$\, :$ 
\[
\SelectTips{cm}{12}\xymatrix{0\ar[r] & \sco_\piii(-1)\ar[r]\ar[d] & 
M(2) \oplus 2\sco_\piii\ar[r]\ar[d] & G(2)\ar[r]\ar @{=}[d] & 0\\ 
0\ar[r] & 2\sco_\piii\ar[r] & F\ar[r] & G(2)\ar[r] & 0}  
\] 
from which one deduces an exact sequence$\, :$ 
\[
0 \lra \sco_\piii(-1) \xra{\left(\begin{smallmatrix} u\\ 
v\end{smallmatrix}\right)} M(2) \oplus 4\sco_\piii \lra F \lra 0\, , 
\]
Since $\tH^i(F^\vee) = 0$, $i = 0,\, 1$, $\tH^0(v^\vee) \colon \tH^0(4\sco_\piii) 
\ra \tH^0(\sco_\piii(1))$ is an isomorphism hence $v$ is defined by 4 linearly 
independent linear forms. One gets, now, easily, the exact sequence from item 
(iii) of the conclusion.  
\end{proof} 

\begin{lemma}\label{L:h2f(-3)neq0}
Let $F$ be a globally generated vector bundle on $\piii$, of rank $r$,  
with $c_1 = 5$, $c_2 \leq 12$, such that ${\fam0 H}^i(F^\vee) = 0$, $i = 0,\, 
1$, and ${\fam0 H}^0(F(-2)) = 0$. If ${\fam0 H}^2(F(-3)) \neq 0$ then $r \geq 
5$ and $F$ can be realized as an extension$\, :$ 
\[
0 \lra F_1 \lra F \lra {\fam0 T}_\piii(-1) \lra 0\, , 
\] 
with $F_1$ a vector bundle of rank $r - 3$ which, in turn, can be realized as 
an extension$\, :$ 
\[
0 \lra (r - 5)\sco_\piii \lra F_1 \lra \scf_1(2) \lra 0\, , 
\]
where $\scf_1$ is a stable rank $2$ reflexive sheaf with $c_1(\scf_1) = 0$, 
$c_2(\scf_1) = c_2 - 9$ and $c_3(\scf_1) = c_3 - c_2$. 
\end{lemma} 

We recall, in connection with the above lemma, that the stable rank 2 
reflexive sheaves $\scf$ on $\piii$ with $c_1(\scf) = 0$ and $c_2(\scf) \leq 
3$ are studied by Chang in \cite{ch2}. 

\begin{proof}[Proof of Lemma~\emph{\ref{L:h2f(-3)neq0}}] 
According to Lemma~\ref{L:erat(-1)}, there exists an epimorphism $\e \colon 
F \ra \text{T}_\piii(-1)$. Let $F_1$ be its kernel. It has $c_1(F_1) = 4$ and 
$\tH^0(F_1(-2)) = 0$ hence it must have rank at least 2. One also has 
$\tH^0(F_1^\vee) = 0$.   

Now, if $W$ is a general vector subspace of dimension $r - 1$ of $\tH^0(F)$ 
then one has an exact sequence$\, :$ 
\[
0 \lra W \otimes_k \sco_\piii \lra F \lra \sci_Y(5) \lra 0\, , 
\]
where $Y$ is a nonsingular curve which is \emph{connected} by 
\cite[Lemma~1.1]{acm3}. Since $W$ is general, one can also assume that 
$\tH^0(\e)$ maps $W$ surjectively onto $\tH^0(\text{T}_\piii(-1))$ (the map 
$\tH^0(\e) \colon \tH^0(F) \ra \tH^0(\text{T}_\piii(-1))$ is surjective  
because the only vector subspace of $\tH^0(\text{T}_\piii(-1))$ generating 
$\text{T}_\piii(-1)$ globally is $\tH^0(\text{T}_\piii(-1))$). One gets a 
commutative diagram$\, :$ 
\[
\SelectTips{cm}{12}\xymatrix{0\ar[r] & (r-5)\sco_\piii\ar[r]\ar[d] & 
W \otimes \sco_\piii\ar[r]\ar[d] & 4\sco_\piii\ar[r]\ar[d] & 0\\ 
0\ar[r] & F_1\ar[r] & F\ar[r]^-\e & \text{T}_\piii(-1)\ar[r] & 0}
\]
If $\sce_1$ is the cokernel of $(r-5)\sco_\piii \ra F_1$ then it sits into an 
exact sequence$\, :$ 
\[
0 \lra \sco_\piii(-1) \lra \sce_1 \lra \sci_Y(5) \lra 0\, . 
\]
Applying $\sch om_{\sco_\piii}(\ast , \sco_\piii(-1))$, one gets an exact 
sequence$\, :$ 
\[
0 \lra \sco_\piii(-6) \lra \sce_1^\vee(-1) \lra \sco_\piii \overset{\delta}{\lra} 
\omega_Y(-2) \lra \sce xt_{\sco_\piii}^1(\sce_1 , \sco_\piii(-1)) \lra 0\, . 
\]
One cannot have $\delta = 0$ because, otherwise, $\sce_1^\vee(-1) \simeq 
\sco_\piii \oplus \sco_\piii(-6)$ and this would contradict the fact that 
$\tH^0(F_1^\vee(-1)) = 0$. Since $\delta \neq 0$ and since $Y$ is a connected 
nonsingular curve one gets that the support of $\sce xt_{\sco_\piii}^1(\sce_1 , 
\sco_\piii(-1))$ is 0-dimensional or empty which implies that $\sce_1$ is 
reflexive (of rank 2). $\scf_1 := \sce_1(-2)$ has the Chern classes from the 
statement and $\tH^0(\scf_1) = 0$ (i.e., $\scf_1$ is stable) because  
$\tH^0(F_1(-2)) = 0$. 
\end{proof}  

\begin{remark}\label{R:h2f(-3)neq0} 
As we saw in the above proof, the map 
$\tH^0(F) \ra \tH^0(\text{T}_\piii(-1))$ is surjective. One deduces easily an 
exact sequence$\, :$ 
\[
0 \lra \sco_\piii(-1) \xra{\left(\begin{smallmatrix} u\\ 
v\end{smallmatrix}\right)} F_1 \oplus 4\sco_\piii \lra F \lra 0\, , 
\]
with $v$ defined by 4 linearly independent linear forms. It follows that if 
the multiplication map $\tH^0(F_1) \otimes \tH^0(\sco_\piii(1)) \ra 
\tH^0(F_1(1))$ is surjective then, up to an automorphism of $F_1 \oplus 
4\sco_\piii$, one can assume that $u = 0$ hence $F \simeq \text{T}_\piii(-1) 
\oplus F_1$.  
\end{remark} 

\begin{prop}\label{P:c2geq9} 
Let $F$ be a globally generated vector bundle on $\piii$ with $c_1 = 5$ and 
such that ${\fam0 H}^i(F^\vee) = 0$, $i = 0,\, 1$, and ${\fam0 H}^0(F(-2)) = 
0$. Then $c_2 \geq 9$ and if $c_2 = 9$ then $c_3 = 5$ and one of the following 
holds$\, :$ 
\begin{enumerate}
\item[(i)] $F \simeq \Omega_\piii(3)$$\, ;$ 
\item[(ii)] $F \simeq \sco_\piii(1) \oplus N(2)$, where $N$ is a 
nullcorrelation bundle. 
\end{enumerate}
\end{prop} 

\begin{proof}
Assume, firstly, that $F$ has rank 2. In this case, $F = M(3)$ where $M$ is a 
rank 2 vector bundle with $c_1(M) = -1$ and $\tH^0(M(1)) = 0$. In particular, 
$M$ is stable (i.e., $\tH^0(M) = 0$). It follows that $c_2(M) \geq 2$ (use 
\cite[Cor.~3.3]{ha} and the fact that $c_2(M) \equiv 0 \pmod{2}$). But, as 
shown by Hartshorne and Sols \cite{hs} and by Manolache \cite{ma}, if 
$c_2(M) = 2$ then $\tH^0(M(1)) \neq 0$. It remains that $c_2(M) \geq 4$ hence 
$c_2 = c_2(M) + 3c_1(M) + 3^2 \geq 10$. 

If $\text{rk}\, F \geq 3$ then, according to Prop.~\ref{P:fprimunstable}, one 
has to consider three cases$\, :$ 

\vskip2mm 

\noindent
{\bf Case 1.}\quad $F$ \emph{as in Prop.}~\ref{P:fprimunstable}(i). 

\vskip2mm 

\noindent
In this case, $c_2 = c_2(M) + 8$. Since $M$ is stable it follows 
that $c_2(M) \geq 1$ hence $c_2 \geq 9$. Moreover, if $c_2 = 9$, i.e., 
if $c_2(M) = 1$, then $M$ is isomorphic to a nullcorrelation bundle 
$N$. Since $\tH^1(N(1)) = 0$ it follows that $F \simeq \sco_\piii(1) 
\oplus N(2)$. 

\vskip2mm 

\noindent
{\bf Case 2.}\quad $F$ \emph{as in Prop.}~\ref{P:fprimunstable}(iii). 

\vskip2mm 

\noindent
In this case, $c_2 = c_2(M) + 9 \geq 10$. 

\vskip2mm 

\noindent
{\bf Case 3.}\quad \emph{The rank} $3$ \emph{vector bundle} $G$ 
\emph{associated to} $F$ \emph{in the statement of 
Prop.}~\ref{P:fprimunstable} \emph{is stable}. 

\vskip2mm

\noindent
The first two Chern classes of $G$ are $c_1(G) = -1$, $c_2(G) = c_2 - 8$. 
According to the results of 
Schneider \cite{sch}, $c_2(G) \geq 1$ and if $c_2(G) = 1$ then $G \simeq 
\Omega_\piii(1)$. One deduces that $c_2 \geq 9$ and if $c_2 = 9$ then  
$G \simeq \Omega_\piii(1)$. The formula $r = 3 + \h^2(G(-2))$ 
(deduced at the beginning of the proof of Prop.~\ref{P:fprimunstable}) implies 
that, in the case $c_2 = 9$, $r = 3$ hence $F = G(2) \simeq \Omega_\piii(3)$.  
\end{proof} 

\begin{remark}\label{R:fprimstable} 
Let $F$ be a globally generated vector bundle of rank $r \geq 3$ on $\piii$, 
with Chern classes $c_1 = 5$, $10 \leq c_2 \leq 12$, $c_3$, such that 
$\tH^i(F^\vee) = 0$, $i = 0,\, 1$, and $\tH^0(F(-2)) = 0$. According to 
Prop.~\ref{P:fprimunstable}, except for the cases stated in 
the conclusion of that proposition, $F$ can be realized as an extension$\, :$ 
\[
0 \lra (r - 3)\sco_\piii \lra F \lra G(2) \lra 0\, , 
\]
for some \emph{stable} rank 3 vector bundle $G$ with Chern classes $c_1(G) = 
-1$, $c_2(G) = c_2 - 8$, $c_3(G) = c_3 - 2c_2 + 12$. Moreover, $r = 3 + 
\h^2(G(-2))$ (as we saw at the beginning of the proof of 
Prop.~\ref{P:fprimunstable}). The intermediate cohomology of $G$ can be 
described, in part, by a sequence of integers $(k_1,\, k_2,\, \ldots ,\, k_m)$,
$k_1 \geq \cdots \geq k_m$, called the \emph{spectrum} of $G$ and denoted by 
$k_G$, according to the following formulae$\, :$ 
\begin{enumerate} 
\item[(i)] $\h^1(G(l)) = \h^0(\bigoplus_{i=1}^m\sco_\pj(k_i + l + 1))$ for 
$l \leq -1$$\, ;$ 
\item[(ii)] $\h^2(G(l)) = \h^1(\bigoplus_{i=1}^m\sco_\pj(k_i + l + 1))$ for 
$l \geq -2$.   
\end{enumerate} 
Moreover, one has$\, :$ 
\begin{enumerate} 
\item[(iii)] $m = c_2(G) = c_2 - 8$ and $-2\sum k_i = c_3(G) + c_2(G) = 
c_3 - c_2 + 4$$\, ;$ 
\item[(iv)] If $k \geq 0$ occurs in the spectrum then $0,\, 1, \ldots ,\, k$ 
occur too$\, ;$ 
\item[(v)] If $k \leq -1$ occurs in the spectrum then $-1,\, -2, \ldots ,\, k$ 
occur too$\, ;$ 
\item[(vi)] If $0$ does not occur in the spectrum then $-1$ occurs at least  
twice$\, ;$ 
\item[(vii)] If $-1 \geq k_{i-1} > k_i > k_{i+1}$ for some $i$ with $2 \leq i 
\leq m-1$ then $k_{i+1} > k_{i+2} > \cdots > k_m$ and $F$ has an \emph{unstable 
plane} $H$ of order $-k_m$, that is, $\tH^0(F_H^\vee(k_m)) \neq 0$ and 
$\tH^0(F_H^\vee(k_m-1)) = 0$.
\end{enumerate} 
Proofs of the above facts can be found in the papers of Okonek and Spindler 
\cite{oks1}, \cite{oks2} and of Coand\u{a} \cite{co1}. These proofs use the 
approach of Hartshorne \cite{ha}, \cite{ha2}, who considered the case of 
stable rank 2 reflexive sheaves on $\piii$. Compact, self-contained arguments 
can be also found in \cite[Appendix~B]{acm3}.    
\end{remark} 

\begin{lemma}\label{L:spectrumg} 
Let $G$ be a stable rank $3$ vector bundle on $\piii$ with $c_1(G) = -1$, 
$c_2(G) = m$, and let $k_G = (k_i)_{1 \leq i \leq m}$ be its spectrum. Assume 
that $2 \leq m \leq 4$ and that $G(2)$ is globally generated. Then 
$1 \geq k_1 \geq \cdots \geq k_m \geq -2$. 
\end{lemma}

\begin{proof} 
Consider the universal extension$\, :$ 
\[
0 \lra \tH^1(G^\vee(-2))^\vee \otimes_k \sco_\piii \lra F \lra G(2) \lra 0\, . 
\]
$F$ is a globally generated vector with $c_1 = 5$, $10 \leq c_2 \leq 12$,  
and such that $\tH^i(F^\vee) = 0$, $i = 0,\, 1$, and $\tH^0(F(-2)) = 0$. 
It follows, from Lemma~\ref{L:h2f(-2)=0}, that $\tH^1(F(-5)) = 0$ and 
$\tH^2(F(-2)) = 0$ hence $\tH^1(G(-3)) = 0$ and $\tH^2(G) = 0$. Using the 
definition of the spectrum one gets the conclusion of the lemma.  
\end{proof} 

\begin{lemma}\label{L:impossiblespectra} 
Let $G$ be a stable rank $3$ vector bundle on $\piii$ with $c_1(G) = -1$, 
$2 \leq c_2(G) \leq 3$, and such that $G(2)$ is globally generated. Then 
$G$ cannot have any of the following spectra$\, :$ $(1 , 0 , -1)$, 
$(0 , -1 , -2 , -2)$, $(1 , 0 , -1 , -2)$, $(1 , 0 , -1 , -1)$.   
\end{lemma} 

\begin{proof} 
We make, firstly, the following observation$\, :$ let $F$ be a globally 
generated vector bundle on $\piii$ with $c_1(F) = 5$, $c_2(F) \in \{11,\, 
12\}$. It follows, from the proof of \cite[Prop.~3.6]{acm1}, that 
$\tH^0(F_H(-3)) = 0$, for every plane $H \subset \piii$. Applying the 
Bilinear Map Lemma \cite[Lemma~5.1]{ha} to the multiplication map 
$\tH^1(F(-4)) \otimes \tH^0(\sco_\piii(1)) \ra \tH^1(F(-3))$ one deduces that 
if $\tH^1(F(-4)) \neq 0$ then $\h^1(F(-3)) \geq \h^1(F(-4)) + 3$.   

Now, if $G$ has spectrum $(1 , 0 , -1)$ then $\h^1(G(-2)) = 1$ and $\h^1(G(-1)) 
= 3$ hence, according to the above observation (applied to $F := G(2)$), 
$G(2)$ cannot be globally generated. 

The spectra $(1 , 0 , -1 , -2)$ and $(1 , 0 , -1 , -1)$ can be eliminated 
similarly. 

Finally, assume, by contradiction, that $G$ has spectrum $(0 , -1 , -2 , -2)$ 
and that $G(2)$ is globally generated. The Chern classes of $G$ are 
$c_1(G) = -1$, $c_2(G) = 4$, $c_3(G) = 6$ hence, by Riemann-Roch, 
$\chi(G(1)) = 2$. It follows that $\h^0(G(1)) \geq 2$. As in the proof of 
Lemma~\ref{L:h0f(-1)c2=12}, one has exact sequences$\, :$ 
\begin{gather*}
0 \lra 2\sco_\piii \lra G(2) \lra \sci_Y(5) \lra 0\, ,\\
0 \lra \sco_\piii(-2) \lra \scg \lra \sci_{Y^\prime}(1) \lra 0\, , 
\end{gather*}    
where $Y$ is a nonsingular connected curve of degree 12, $Y^\prime$ is a 
locally Cohen-Macaulay curve of degree 4, locally complete intersection 
except at finitely many points, directly linked to $Y$ by a complete 
intersection of type $(4 , 4)$, and $\scg$ is a stable reflexive sheaf with 
$c_1(\scg) = -1$, $c_2(\scg) = 2$, $c_3(\scg) = c_3(G(2)) - 16 = 2$. According 
to \cite[Lemma~2.4]{ch2}, $\scg$ can be realized as an extension$\, :$ 
\[
0 \lra \sco_\piii(-1) \lra \scg \lra \sci_Z \lra 0\, , 
\] 
where $Z$ is either the union of two disjoint lines or a divisor of the form 
$2L$ on a nonsingular quadric surface, $L$ being a line. It follows that 
$\tH^1(\scg(1)) = 0$ hence $\tH^1(\sci_{Y^\prime}(2)) = 0$. But, by the well 
known behaviour of the Hartshorne-Rao module $\tH^1_\ast(\ast)$ under liaison, 
$\tH^1(\sci_{Y^\prime}(2)) \simeq \tH^1(\sci_Y(2))^\vee$ hence $\tH^1(\sci_Y(2)) 
= 0$. This implies that $\tH^1(G(-1)) = 0$ which \emph{contradicts} the 
fact that the spectrum of $G$ is $(0 , -1 , -2 , -2)$. 
\end{proof} 

\begin{lemma}\label{L:h0fH(-3)neq0} 
Let $F$ be a globally generated vector bundle on $\piii$ of rank $r \geq 3$, 
with $c_1 = 5$, $10 \leq c_2 \leq 12$, such that ${\fam0 H}^i(F^\vee) = 0$, 
$i = 0,\, 1$, and ${\fam0 H}^0(F(-2)) = 0$. If there exists a plane $H_0 
\subset \piii$ such that ${\fam0 H}^0(F_{H_0}(-3)) \neq 0$ then $c_2 = 10$, 
$c_3 = 4$ and $F$ is the kernel of an epimorphism $\sco_\piii(3) \oplus 
3\sco_\piii(2) \ra \sco_\piii(4)$. 
\end{lemma}

\begin{proof} 
If $c_2 \in \{11,\, 12\}$ then the proof of \cite[Prop.~3.6]{acm1} shows that 
$\tH^0(F_H(-3)) = 0$ for every plane $H \subset \piii$. Assume, now, that 
$c_2 = 10$. If $M$ is a rank 2 vector bundle on $\piii$ with $\tH^0(M(-1)) = 0$ 
and $\tH^1(M(-2)) = 0$ then $\tH^0(M_H(-1)) = 0$ for every plane $H \subset 
\piii$. It follows that $F$ does not satisfy the hypothesis of 
Prop.~\ref{P:fprimunstable} hence it can be realized as an extension$\, :$ 
\[
0 \lra (r-3)\sco_\piii \lra F \lra G(2) \lra 0\, , 
\] 
for some \emph{stable} rank 3 vector bundle $G$ with $c_1(G) = -1$, $c_2(G) = 
2$. Since $\tH^0(F(-3)) = 0$ and $\tH^0(F_{H_0}(-3)) \neq 0$ one deduces that 
$\tH^1(F(-4)) \neq 0$ hence $\tH^1(G(-2)) \neq 0$. The only possible 
spectrum for $G$ is, therefore, $k_G = (1 , 0)$. It follows that $c_3(G) = -4$ 
hence $c_3 = 4$ and, since $\h^2(G(-2)) = 0$, $r = 3$, i.e., $F = G(2)$ (look  
at the beginning of the proof of Prop.~\ref{P:fprimunstable}). 

Now, one has $\tH^1(G(l)) = 0$ for $l \leq -3$, $\h^1(G(-2)) = 1$ and 
$\h^1(G(-1)) = 3$. Since $\tH^2(G(-2)) = 0$ and $\tH^3(G(-3)) \simeq 
\tH^0(G^\vee(-1))^\vee = 0$ it follows, from the Castelnuovo-Mumford lemma, that 
the graded $S$-module $\tH^1_\ast(G)$ is generated in degrees $\leq -1$. 

\vskip2mm 

\noindent 
{\bf Claim.}\quad \emph{The multiplication map} $\tH^1(G(-2)) \otimes 
\tH^0(\sco_\piii(1)) \ra \tH^1(G(-1))$ \emph{is surjective}. 

\vskip2mm 

\noindent 
\emph{Indeed}, if it is not then there exist two linearly independent 
linear forms $h_0$ and $h_1$ annihilating $\tH^1(G(-2))$ inside $\tH^1_\ast(G)$. 
Let $L \subset \piii$ be the line of equations $h_0 = h_1 = 0$. Tensorizing by 
$G$ the exact sequence $0 \ra \sco_\piii(-2) \ra 2\sco_\piii(-1) \ra \sci_L 
\ra 0$ one deduces that $\tH^0(\sci_L \otimes G) \neq 0$ which 
\emph{contradicts} the fact that $\tH^0(G) = 0$. 

\vskip2mm 

Consider, now, the universal extension$\, :$ 
\[
0 \lra G \lra B \lra \sco_\piii(2) \lra 0\, . 
\]
$B$ is a rank 4 vector bundle with $\tH^1(B(-1)) = 0$, $\tH^2(B(-2)) \simeq 
\tH^2(G(-2)) = 0$ and $\tH^3(B(-3)) \simeq \tH^3(G(-3)) = 0$. It follows that 
$B$ is 0-regular. One has 
$\h^0(B(-1)) = \h^0(\sco_\piii(1)) - \h^1(G(-1)) = 1$ and $\h^0(B) = \chi(B) =  
\chi(G) + \chi(\sco_\piii(2)) = 7$. One deduces that the graded $S$-module 
$\tH^0_\ast(B)$ has one minimal generator of degree $-1$ and three minimal 
generators of degree 0. The epimorphism $\sco_\piii(1) \oplus 3\sco_\piii \ra B$ 
defined by these generators must be an isomorphism because $B$ has rank 4. 
\end{proof}

\begin{lemma}\label{L:h2f(-3)=0}  
Let $F$ be a globally generated vector bundle on $\piii$ of rank $r \geq 3$, 
with $c_1 = 5$, $10 \leq c_2 \leq 12$, such that ${\fam0 H}^i(F^\vee) = 0$, 
$i = 0,\, 1$. Assume that $F$ can be realized as an extension$\, :$ 
\[
0 \lra (r - 3)\sco_\piii \lra F \lra G(2) \lra 0\, , 
\] 
where $G$ is a stable rank $3$ vector bundle with $c_1(G) = -1$ $($see 
Remark~\emph{\ref{R:fprimstable}}$)$. Assume, moreover, that 
${\fam0 H}^2(F(-3)) = 0$ and that $F$ is not the bundle from the conclusion of 
Lemma~\emph{\ref{L:h0fH(-3)neq0}}. Put $s := {\fam0 h}^1(F(-3)) - 
{\fam0 h}^1(F(-4))$. Then$\, :$ 

\emph{(a)} ${\fam0 H}^0(F_H^\vee) = 0$ and ${\fam0 h}^1(F_H^\vee) = s$, 
for any plane $H \subset \piii$$\, ;$

\emph{(b)} The graded $S$-module ${\fam0 H}_\ast^1(F)$ is generated in 
degrees $\leq -2$$\, ;$ 

\emph{(c)} If ${\fam0 H}^1(F(-4)) \neq 0$ then $s \geq 3$ and if, moreover, 
$s = 3$ then the multiplication map ${\fam0 H}^1(F(-4)) \otimes 
{\fam0 H}^0(\sco_\piii(1)) \ra {\fam0 H}^1(F(-3))$ is surjective$\, ;$ 

\emph{(d)} ${\fam0 H}^1(F^\vee(l)) = 0$ for $l \leq 0$ and, if 
${\fam0 h}^1(F(-4)) \leq 1$, then the graded $S$-module 
${\fam0 H}^1_\ast(F^\vee)$ is generated by ${\fam0 H}^1(F^\vee(1))$$\, ;$ 

\emph{(e)} ${\fam0 h}^1(F_H^\vee(1)) = {\fam0 h}^1(F^\vee(1)) + 
{\fam0 h}^2(F^\vee)$, for any plane $H \subset \piii$$\, ;$ 

\emph{(f)}  ${\fam0 h}^1(F_H^\vee(l)) \leq 
\max \, ({\fam0 h}^1(F_H^\vee(l-1)) - 1 , 0)$, $\forall \, l \geq 1$,  
for any plane $H \subset \piii$.    
\end{lemma} 

\begin{proof} 
(a) Since $\tH^i(F^\vee) = 0$, $i = 0,\, 1$, one has $\tH^0(F_H^\vee) \izo 
\tH^1(F^\vee(-1))$. But $\tH^1(F^\vee(-1)) \simeq \tH^2(F(-3))^\vee = 0$ hence 
$\tH^0(F_H^\vee) = 0$. For the second relation one uses the exact 
sequence$\, :$ 
\[
0 = \tH^1(F^\vee) \ra \tH^1(F_H^\vee) \ra \tH^2(F^\vee(-1)) \ra \tH^2(F^\vee) 
\ra \tH^2(F_H^\vee) 
\] 
and the fact that, by Lemma~\ref{L:h0fH(-3)neq0}, $\tH^2(F_H^\vee) \simeq 
\tH^0(F_H(-3))^\vee = 0$. 

\vskip2mm 

(b) This follows from the Castelnuovo-Mumford lemma (in the slightly more 
general form stated in \cite[Lemma~1.21]{acm1}) because $\tH^2(F(-3)) = 0$ 
and $\tH^3(F(-4)) \simeq \tH^0(F^\vee)^\vee = 0$. 

\vskip2mm 

(c) Since $\tH^0(F_H(-3)) = 0$, for any plane $H \subset \piii$, by 
Lemma~\ref{L:h0fH(-3)neq0}, one deduces that the multiplication by any 
non-zero linear form $h \colon \tH^1(F(-4)) \ra \tH^1(F(-3))$ is injective. 
One applies, now, the Bilinear Map Lemma \cite[Lemma~5.1]{ha}.  

\vskip2mm 

(d) One has, by hypothesis, $\tH^1(F^\vee) = 0$ and $\tH^1(F^\vee(-1)) \simeq 
\tH^2(F(-3))^\vee = 0$. On the other hand, for $l \leq -2$, $\tH^1(F^\vee(l)) 
\simeq \tH^2(F(-l-4))^\vee = 0$, by Lemma~\ref{L:h2f(-2)=0}(b). 

Now, if $\tH^1(F(-4)) = 0$ then, by Serre duality, $\tH^2(F^\vee) = 0$ and  
$\tH^3(F^\vee(-1)) \simeq \tH^0(F(-3))^\vee = 0$. It follows, from the 
Castelnuovo-Mumford lemma, that $\tH^1_\ast(F^\vee)$ is generated in degrees 
$\leq 1$ hence, actually, by $\tH^1(F^\vee(1))$. 

Assume, finally, that $\h^1(F(-4)) = 1$ and consider the extension$\, :$ 
\[
0 \lra F(-4) \lra A \lra \sco_\piii \lra 0 
\]
defined by a non-zero element of $\tH^1(F(-4)) \simeq k$. One has $\tH^1(A) 
= 0$. Moreover, as we noticed in the proof of (c), the multiplication by 
any non-zero linear form $h \colon \tH^1(F(-4)) \ra \tH^1(F(-3))$ is injective. 
This implies that $\tH^0(A(1)) = 0$. It follows, by Serre duality, that 
$\tH^2(A^\vee(-4)) = 0$ and $\tH^3(A^\vee(-5)) = 0$. One deduces, from the 
Castelnuovo-Mumford lemma, that $\tH^1_\ast(A^\vee)$ is generated in degrees 
$\leq -3$. But $\tH^1_\ast(A^\vee) \izo \tH^1_\ast(F^\vee(4))$. 

\vskip2mm 

(e) One uses the exact sequence$\, :$ 
\[
0 = \tH^1(F^\vee) \ra \tH^1(F^\vee(1)) \ra \tH^1(F_H^\vee(1)) \ra \tH^2(F^\vee) 
\ra \tH^2(F^\vee(1)) 
\] 
and the fact that $\tH^2(F^\vee(1)) \simeq \tH^1(F(-5))^\vee = 0$, by 
Lemma~\ref{L:h2f(-2)=0}(a).  

\vskip2mm 

(f) We treat, firstly, the case $l = 1$.  
If $H \subset \piii$ is an arbitrary plane then $\h^1(F_H^\vee) = s$ (by 
(a)) and $\h^1(F_H^\vee(1)) = \h^1(F^\vee(1)) + \h^2(F^\vee)$ (by (e)). It 
follows that, in order to prove the inequality from the statement for $l = 1$, 
one can assume that $H$ is a \emph{general} plane. We shall, actually, assume 
that $G_H$ is stable (using the restriction theorem of 
Schneider \cite{sch}$\, ;$ see, also, Ein et al. \cite[Thm.~3.4]{ehv}).   
By Serre duality on $H$, one has $\h^1(F_H^\vee) = \h^1(F_H(-3))$ and 
$\h^1(F_H^\vee(1)) \simeq \h^1(F_H(-4))$. Using the exact sequence$\, :$ 
\[
0 \lra (r - 3)\sco_H \lra F_H \lra G_H(2) \lra 0\, ,  
\]  
and applying \cite[Prop.~1.6(b)]{co1} (with $E = G_H$ and $N_l = 
\tH^1(F(l-2))$), one gets that $\h^1(F_H(-4)) = 0$ or $\h^1(F_H(-4)) < 
\h^1(F_H(-3))$. 

Assume, now, that $l \geq 2$ and let $H \subset \piii$ be an arbitrary plane. 
$r - 2$ general global sections of $F_H$ define an exact sequence$\, :$ 
\[
0 \lra (r - 2)\sco_H \lra F_H \lra Q^\prime \lra 0\, , 
\]
with $Q^\prime$ a rank 2 vector bundle on $H$ with $c_1(Q^\prime) = 5$. 
Consider the normalized rank 2 vector bundle $Q := Q^\prime(-3)$ which has 
$c_1(Q) = -1$. Since $\tH^0(F_H(-3)) = 0$, by Lemma~\ref{L:h0fH(-3)neq0}, 
it follows that $\tH^0(Q) = 0$, i.e., $Q$ is stable. Applying 
\cite[Thm.~5.3]{ha} (with $\sce = Q$ and $N_{-l} = \tH^1(F_H(-l-3))$) one gets 
that $\tH^1(F_H(-l-3)) = 0$ or $\h^1(F_H(-l-3)) < \h^1(F_H(-l-2))$ , i.e., 
$\tH^1(F_H^\vee(l)) = 0$ or $\h^1(F_H^\vee(l)) < \h^1(F_H^\vee(l-1))$. 
\end{proof} 

\begin{lemma}\label{L:(1,0,0,-1)} 
Let $F$ be a globally generated vector bundle on $\piii$ of rank $r \geq 3$, 
with $c_1 = 5$, $c_2 = 12$, and such that ${\fam0 H}^i(F^\vee) = 0$, $i = 0,\, 
1$. Assume that $F$ can be realized as an extension$\, :$ 
\[
0 \lra (r - 3)\sco_\piii \lra F \lra G(2) \lra 0\, , 
\] 
where $G$ is a stable rank $3$ vector bundle with $c_1(G) = -1$, $c_2(G) = 4$ 
and spectrum $k_G = (1 , 0 , 0 , -1)$ $($see 
Remark~\emph{\ref{R:fprimstable}}$)$. Then $F$ is the kernel of an 
epimorphism $4\sco_\piii(2) \oplus \sco_\piii(1) \ra \sco_\piii(4)$. 
\end{lemma} 

\begin{proof} 
One has $c_3(G) = -4$ hence $c_3 = 8$. Moreover, $r = 4$ because $\h^2(G(-2)) 
= 1$ (one uses the spectrum). By Lemma~\ref{L:h2f(-3)=0}(b), the graded 
$S$-module $\tH^1_\ast(F)$ is generated in degrees $\leq -2$.  
One has $\h^1(F(-4)) = \h^1(G(-2)) = 1$ 
and $\h^1(F(-3)) = 4$ hence, by Lemma~\ref{L:h2f(-3)=0}(c), the multiplication 
map $\tH^1(F(-4)) \otimes \tH^0(\sco_\piii(1)) \ra \tH^1(F(-3))$ is bijective. 
If $H \subset \piii$ is a general plane, of equation $h = 0$, then $G_H$ is 
stable. In particular, $\tH^0(F_H(-2)) = \tH^0(G_H) = 0$ hence multiplication 
by $h \colon \tH^1(F(-3)) \ra \tH^1(F(-2))$ is injective. Since $\h^1(F(-2)) 
= 6$ (see Lemma~\ref{L:h2f(-2)=0}(b)), one deduces that the multiplication 
map $\tH^1(F(-3)) \otimes \tH^0(\sco_\piii(1)) \ra \tH^0(F(-2))$ has 
corank $\leq 2$. It follows that the graded $S$-module $\tH^1_\ast(F)$ has 
one minimal generator of degree $-4$ and at most two minimal generators of 
degree $-2$. 

On the other hand, by Lemma~\ref{L:h2f(-3)=0}(d), the graded $S$-module 
$\tH^1_\ast(F^\vee)$ is generated by $\tH^1(F^\vee(1))$. We want to estimate 
$\h^1(F^\vee(1))$. Let $H \subset \piii$ be a plane. $\h^1(F_H^\vee) = 3$ (by 
Lemma~\ref{L:h2f(-3)=0}(a)) hence $\h^1(F_H^\vee(1)) \leq 2$ (by 
Lemma~\ref{L:h2f(-3)=0}(f)). Since $\h^2(F^\vee) = \h^1(F(-4)) = 1$, it 
follows, from Lemma~\ref{L:h2f(-3)=0}(e), that $\h^1(F^\vee(1)) \leq 1$. 

By what has been proven so far, $F(-2)$ is the cohomology sheaf of a Horrocks 
monad of the form$\, :$ 
\[
0 \lra \sco_\piii(-1) \overset{\beta}{\lra} B \overset{\alpha}{\lra} 
\sco_\piii(2) \oplus 2\sco_\piii \lra 0\, , 
\]  
where $B$ is a direct sum of line bundles. $B$ must have rank 8, $\h^0(B) = 
\h^0(\sco_\piii(2) \oplus 2\sco_\piii) - \h^1(F(-2)) = 6$, $\h^0(B(-1)) = 0$ 
and $\tH^0(B^\vee(-2)) = 0$ (because $\tH^0(F^\vee) = 0$). It follows that 
$B \simeq 6\sco_\piii \oplus 2\sco_\piii(-1)$. Since there is no epimorphism 
$2\sco_\piii(-1) \ra \sco_\piii$, the component $6\sco_\piii \ra 2\sco_\piii$ of 
$\alpha$ must be surjective hence $F(-2)$ is the cohomology sheaf of a 
monad of the form$\, :$ 
\[
0 \lra \sco_\piii(-1) \overset{\beta^\prim}{\lra} 4\sco_\piii \oplus 
2\sco_\piii(-1) \overset{\alpha^\prime}{\lra} \sco_\piii(2) \lra 0\, . 
\]
In order to complete the proof of the lemma, it suffices to verify the 
following$\, :$ 

\vskip2mm 

\noindent 
{\bf Claim.}\quad \emph{The component} $\sco_\piii(-1) \ra 2\sco_\piii(-1)$ 
\emph{of} $\beta^\prim$ \emph{is non-zero}. 

\vskip2mm 

\noindent 
\emph{Indeed}, assume, by contradiction, that this component is zero. 
Then one has an exact sequence$\, :$ 
\[
0 \lra F(-2) \lra \text{T}_\piii(-1) \oplus 2\sco_\piii(-1) 
\overset{\alpha^\secund}{\lra} \sco_\piii(2) \lra 0\, . 
\] 
Let $\alpha_1^\secund \colon \text{T}_\piii(-1) \ra \sco_\piii(2)$ and 
$\alpha_2^\secund \colon 2\sco_\piii(-1) \ra \sco_\piii(2)$ be the components of 
$\alpha^\secund$. $\Cok \alpha_1^\secund \simeq \sco_Z(2)$, for some closed 
subscheme $Z$ of $\piii$. Let $\pi$ denote the composite epimorphism$\, :$ 
\[
2\sco_\piii(-1) \overset{\alpha_2^\secund}{\lra} \sco_\piii(2) \lra \sco_Z(2)\, . 
\]
Restricting to $Z$ the exact sequence$\, :$ 
\[
0 \lra \Ker \alpha_1^\secund \lra F(-2) \lra 2\sco_\piii(-1) \overset{\pi}{\lra} 
\sco_Z(2) \lra 0 
\]
one gets an epimorphism $F_Z(-2) \ra \sco_Z(-4)$. Since $F$ is globally 
generated, it follows that $\dim Z \leq 0$. Since $c_3(\Omega_\piii(3)) = 5$, 
$Z$ is a 0-dimensional subscheme of $\piii$ of length 5. $\alpha_1^\secund$ 
can be extended to a Koszul resolution of $\sco_Z(2)$$\, :$ 
\[
0 \lra \sco_\piii(-3) \lra \Omega_\piii \lra \text{T}_\piii(-1) 
\overset{\alpha_1^\secund}{\lra} \sco_\piii(2) \lra \sco_Z(2) \lra 0
\] 
(we used the fact that ${\bigwedge}^2(\text{T}_\piii(-1)) \simeq 
\Omega_\piii(2)$). One gets an exact sequence$\, :$ 
\[
0 \lra \sco_\piii(-3) \lra \Omega_\piii \lra F(-2) \lra 2\sco_\piii(-1) 
\overset{\pi}{\lra} \sco_Z(2) \lra 0\, . 
\]
Since $\sci_Z(1)$ is not globally generated the map $\tH^0(\pi(1)) \colon 
\tH^0(2\sco_\piii) \ra \tH^0(\sco_Z(3))$ is injective. One gets that 
$\tH^0(F(-1)) = 0$ hence $\h^1(F(-1)) = 3$ (by Lemma~\ref{L:h2f(-2)=0}(b)). 
It follows, from the last asssertion in Lemma~\ref{L:h2f(-2)=0}(b), that 
$\tH^1(F) = 0$. The above exact sequence implies, now, that 
$\tH^1(\Ker \pi(2)) = 0$ and that $\Ker \pi(2)$ is globally generated. 
Using the exact sequence$\, :$ 
\[
0 \lra \Ker \pi(2) \lra 2\sco_\piii(1) \lra \sco_Z(4) \lra 0 
\]   
one deduces that $\h^0(\Ker \pi(2)) = \h^0(2\sco_\piii(1)) - \h^0(\sco_Z(4)) 
= 3$. One obtains, now, an exact sequence$\, :$ 
\[
3\sco_\piii \lra 2\sco_\piii(1) \lra \sco_Z(4) \lra 0\, . 
\]
But such an exact sequence cannot exist because $Z$ has codimension 3 in 
$\piii$. This \emph{contradiction} shows that the component $\sco_\piii(-1) 
\ra 2\sco_\piii(-1)$ of $\beta^\prim$ is non-zero and the claim is proven. 
\end{proof} 

\begin{remark}\label{R:(1,0,0,-1)} 
One can actually show that, under the hypothesis of Lemma~\ref{L:(1,0,0,-1)}, 
$F \simeq \sco_\piii(1) \oplus F_0$, where $F_0$ is the kernel of an 
epimorphism $4\sco_\piii(2) \ra \sco_\piii(4)$$\, :$ see Claim 4.5 in the proof 
of \cite[Prop.~4.13]{acm3}. 
\end{remark} 

\begin{lemma}\label{L:(1,1,0,-1)} 
Let $F$ be a globally generated vector bundle on $\piii$ of rank $r \geq 3$, 
with $c_1 = 5$, $c_2 = 12$, and such that ${\fam0 H}^i(F^\vee) = 0$, $i = 0,\, 
1$. Assume that $F$ can be realized as an extension$\, :$ 
\[
0 \lra (r - 3)\sco_\piii \lra F \lra G(2) \lra 0\, , 
\] 
where $G$ is a stable rank $3$ vector bundle with $c_1(G) = -1$, $c_2(G) = 4$ 
and spectrum $k_G = (1 , 1 , 0 , -1)$ $($see 
Remark~\emph{\ref{R:fprimstable}}$)$. Then $r = 4$, $c_3 = 6$ 
and ${\fam0 H}^1(F^\vee(1)) = 0$. 
\end{lemma}

\begin{proof} 
Using Remark~\ref{R:fprimstable} one sees easily that $r = 4$ and $c_3(G) = 
-6$ hence $c_3 = 6$. Lemma~\ref{L:h2f(-3)=0}(a) implies that $\h^1(F_H^\vee) 
= 3$, for every plane $H \subset \piii$, while item (f) of the same lemma 
implies, now, that $\h^1(F_H^\vee(1)) \leq 2$. Since $\h^2(F^\vee) = 
\h^1(F(-4)) = 2$ (use the spectrum), one deduces, from 
Lemma~\ref{L:h2f(-3)=0}(e), that $\h^1(F^\vee(1)) = 0$.  
\end{proof}

\begin{remark}\label{R:(1,1,0,-1)} 
One can show that there is no bundle $F$ satisfying the hypothesis of 
Lemma~\ref{L:(1,1,0,-1)}$\, :$ see Case~7 in the proof of 
\cite[Prop.~4.13]{acm3}.  
\end{remark}

\end{document}